\documentclass[10.5pt]{amsart}%
\usepackage{amsmath}
\usepackage{amsfonts}
\usepackage{amssymb}
\usepackage{color}
\usepackage{hyperref}
\usepackage{a4wide}
\usepackage{amsthm}
\usepackage[french]{babel}
\usepackage{version}
\usepackage{accents}
\usepackage[latin1]{inputenc}%
\usepackage{graphicx}
\providecommand{\U}[1]{\protect\rule{.1in}{.1in}}

\newtheorem{theorem}{Théorème}[section]

\newtheorem{corollary}{Corollaire}[section]

\newtheorem{definition}{Définition}[section]

\newtheorem{lemma}{Lemme}[section]
\newtheorem{notation}{Notation}[section]

\newtheorem{proposition}{Proposition}[section]
\newtheorem{remark}{Remarque}

\numberwithin{equation}{section}

\begin{document}
\title[Introduction à l'équation de Burgers stochastique et à la Burgulence]
{Introduction à l'équation de Burgers stochastique \\ et à la Burgulence}
\author{Takfarinas Kelaï \& Sergei Kuksin}
\address{Institut de Mathématiques de Jussieu-Paris Rive Gauche, UMR 7586, Univ. Paris
Diderot, Sorbonne Paris Cité, Sorbonne Universités, UPMC Univ. Paris 06,
F-75013, Paris, France}
\email{takfarinas.kelai@imj-prg.fr}
\address{CNRS, Institut de Mathématiques de Jussieu-Paris Rive Gauche, UMR 7586, Univ.
Paris Diderot, Sorbonne Paris Cité, Sorbonne Universités, UPMC Univ. Paris 06,
F-75013, Paris, France}
\email{sergei.kuksin@imj-prg.fr}
\date{\today}
\maketitle

\begin{abstract}
Cet article propose une introduction à l'équation de Burgers visqueuse
stochastique unidimensionnelle périodique, perturbée par une force aléatoire
de type bruit blanc en temps, et suffisamment régulière en espace. Nous
prouvons des résultats classiques sur l'existence et l'unicité des solutions,
nous étudions leur régularité et nous discutons leurs propriétés quand le temps tend
vers l'infini ou quand la viscosité $\nu$ tend vers zéro. La dernière limite
décrit la turbulence dans l'équation de Burgers, nommée burgulence par U.
Frisch. Notre article sert d'introduction élémentaire aux méthodes
modernes de l'analyse des équations aux dérivées partielles stochastiques.

\end{abstract}
\tableofcontents

\section{Introduction}

On se propose d'étudier les propriétés qualitatives de l'équation de Burgers
visqueuse stochastique avec une force aléatoire $\eta^{\omega}:$
\begin{equation}
\left\{
\begin{array}
[c]{c}%
\operatorname{u}%
_{t}(t,x)+%
\operatorname{u}%
(t,x)%
\operatorname{u}%
_{x}(t,x)-\nu%
\operatorname{u}%
_{xx}(t,x)=\eta^{\omega}(t,x),\\
t\geq0,\text{ }x\in\mathbb{S}^{1}=%
\mathbb{R}
/%
\mathbb{Z}
.
\end{array}
\right. \tag{B}\label{B}%
\end{equation}
Ici $\nu\in(0,1]$ est la \textit{viscosité}. On suppose que pour tout $t\geq0$
:%
\[
\int_{\mathbb{S}^{1}}\eta^{\omega}(t,x)dx=\int_{\mathbb{S}^{1}}%
\operatorname{u}%
(0,x)dx=0.
\]
D'ici et d'après (\ref{B}), on a pour tout $t\geq0$%
\[
\int_{\mathbb{S}^{1}}%
\operatorname{u}%
(t,x)dx=0.
\]
On note $H$ l'espace de Hilbert défini par
\[
H=\{v\in L_{2}(\mathbb{S}^{1}):\int_{\mathbb{S}^{1}}v(x)dx=0\},
\]
muni de la norme usuelle et du produit scalaire de $L_{2}$ qu'on notera
$||\cdot||$ et $\langle\cdot,\cdot\rangle,$ et de la base Hilbertienne
$\{e_{k}(x),k\in%
\mathbb{Z}
^{\ast}\}$ avec%
\begin{equation}
\left\{
\begin{array}
[c]{c}%
e_{s}=\sqrt{2}\cos(2\pi sx)\\
e_{-s}=\sqrt{2}\sin(2\pi sx)
\end{array}
\right.  s\in%
\mathbb{N}
^{\star}.\label{base}%
\end{equation}
Notons que si $%
\operatorname{u}%
(x)=\underset{s\in%
\mathbb{Z}
^{\ast}}{\sum}%
\operatorname{u}%
_{s}e_{s}(x)$ alors
\begin{equation}%
\operatorname{u}%
(x)=\underset{s\in%
\mathbb{Z}
^{\ast}}{\sum}\widehat{%
\operatorname{u}%
}_{s}e^{2i\pi sx},\quad\widehat{%
\operatorname{u}%
}_{s}=\overline{\widehat{%
\operatorname{u}%
}}_{-s}=\left(  \sqrt{2}\right)  ^{-1}\left(
\operatorname{u}%
_{s}-i%
\operatorname{u}%
_{-s}\right)  \text{ pour tout }s\in%
\mathbb{N}
^{\star}.\label{expforme}%
\end{equation}
Dans la suite, on donne $\eta^{\omega}$ de la forme :%
\begin{equation}
\eta^{\omega}(t,x)=\partial_{t}\xi^{\omega}(t,x),\quad\text{et }\xi^{\omega
}(t,x)=\sum_{s\in%
\mathbb{Z}
^{\ast}}b_{s}\beta_{s}^{\omega}(t)e_{s}(x),\label{B*}%
\end{equation}
où $b_{s}\in%
\mathbb{R}
$ et $\beta_{s}^{\omega}$ sont des processus de Wiener standard indépendants
(appendice A), définie sur un espace de probabilité standard $(\Omega
,\mathcal{F},\mathbb{P})$ \cite{KaS}, \cite{Shir}. Pour tout $m\geq0 $, on
pose
\begin{equation}
B_{m}=\underset{s\in%
\mathbb{Z}
^{\ast}}{\sum}|s|^{2m}b_{s}^{2}\leq+\infty.\label{B_m}%
\end{equation}
Nous supposerons toujours que $B_{0}<\infty.$ On appelle la force aléatoire
$\eta^{\omega}$ le \textit{Bruit blanc}\emph{\ }dans l'espace $H$ et
$\xi^{\omega}$ le \textit{processus de Wiener dans }$H.$ On dit que $%
\operatorname{u}%
^{\omega}$ est une solution de (\ref{B}), (\ref{B*})$,$ si
\[%
\operatorname{u}%
^{\omega}(t,x)-%
\operatorname{u}%
^{\omega}(0,x)+\int_{0}^{t}\left(
\operatorname{u}%
(s,x)%
\operatorname{u}%
_{x}(s,x)-\nu%
\operatorname{u}%
_{xx}(s,x)\right)  ds=\xi^{\omega}(t,x)-\xi^{\omega}(0,x),
\]
pour tout $t\geq0,$ $x\in\mathbb{S}^{1},$ et pour tout $\omega\notin Q,$ où
$Q$ est un ensemble négligeable (C'est-à-dire, $Q\in\mathcal{F}$ tel que $\mathbb{P(}%
Q\mathbb{)=}0).$

L'équation (\ref{B}), (\ref{B*}) est un modèle populaire en physique théorique
moderne \cite{BK07}. Nous étudions le problème de Cauchy relatif à (\ref{B}),
ainsi que les propriétés qualitatives de ses solutions. À savoir, dans les
sections 2-6, nous prouvons les propriétés d'existence et d'unicité des
solutions de (\ref{B}) et discutons les processus de Markov correspondants.
Puis, dans les sections 7-8, nous obtenons des bornes supérieures et inférieures
pour les normes Sobolev des solutions, lesquelles sont asymptotiquement
précises quand $\nu\rightarrow0$. Dans la section 9, nous déduisons de ces
estimations que l'échelle d'espace pour les solutions de (\ref{B}) est égale à
$\nu$. Les résultats des sections 7-9 sont basés sur les thèses \cite{Birphd},
\cite{Borphd} (voir aussi \cite{Biry01}, \cite{Bor}), où les méthodes
suggérées dans \cite{Kuk97}, \cite{Kuk99} pour étudier l'équation de
Ginzburg-Landau complexe, sont appliquées à (\ref{B}). Dans les sections
10-12, nous étudions les propriétés de mélange de l'équation de Burgers,
utilisant l'approche du couplage de Döblin et suivant \cite{Bor},
\cite{KuShir}. Enfin, dans les sections 13-14, on montre que les estimations
des sections 7-8 entraînent que le spectre de l'energie des solutions de
(\ref{B}) est de la forme de la loi de Kolmogorov \cite{Frisch}, $E_{k}\sim
k^{-2}$. Là, nous suivons \cite{Borphd}, \cite{Bor}, où la dérivation
heuristique du spectre, suggérée dans \cite{AFLV92}, est rigoureusement
justifiée, ceci étant basé sur les résultats des sections précédentes.

Cet article est basé sur les notes de cours pour la seconde année de Master de
Mathématiques Fondamentales à l'université de Paris Diderot, donné par S. K
lors des années universitaires 2012/2013 et 2014/2015.

\section{Estimations de la force $\xi^{\omega}$\label{sectionksi}}

Pour $T>0$, notons par $X^{T}$ l'espace de Banach $\mathcal{C}(0,T;H),$ muni
de la norme $||\xi||_{X^{T}}=\underset{t\in\lbrack0,T]}{\sup}||\xi(t)||.$

\begin{theorem}
\label{th1}Il existe une ensemble négligeable $Q\in\mathcal{F}$ tel que si
$\omega\notin Q$, alors pour tout $T>0$ nous avons $\xi^{\omega}\in X^{T}$. De
plus, il existe $\alpha(T,B_{0})>0,$ $C^{\prime}(T,B_{0})>0,$ et pour chaque
$p\geq1,$ il existe $C(p,T,B_{0})>0$ telle que%
\begin{align}
\mathbb{E[}e^{\alpha||\xi^{\omega}(t)||^{2}}\mathbb{]} &  \leq C^{\prime
},\label{aa}\\
\mathbb{E}[\sup_{t\in(0,T)}||\xi^{\omega}(t)||]^{p} &  \leq C.\label{bb}%
\end{align}

\end{theorem}

Si $\omega\in Q$, alors on redéfinit $\xi^{\omega}$ par $\xi^{\omega}=0$.
Ainsi $\xi^{\omega}\in X^{T}$ pour tout $\omega.$

\begin{proof}
\textit{Étape 1 : troncature. Fixons }$T>0$ et posons pour tout $N\in%
\mathbb{N}
:$%
\[
H^{(N)}=vect\{e_{j},|j|\leq N\}\subset H\text{ et }\xi^{N\omega}%
(t,x)=\sum_{|s|\leq N}b_{s}\beta_{s}^{\omega}(t)e_{s}(x).
\]
Alors $\xi^{N\omega}$ est un processus aléatoire continu dans $H^{(N)}%
$\footnote{C'est à dire, pour presque tout $\omega$, l'application $t\longmapsto
\xi^{N\omega}(t)$ $\in H^{(N)}$ est continue.}. Pour $t\geq0,$ on a%
\[
||\xi^{N\omega}(t)||^{2}=\sum_{|s|\leq N}b_{s}^{2}\beta_{s}^{\omega}(t)^{2}%
\]
Maintenant on va estimer les moments exponentiels de $\underset{t\in
(0,T)}{\sup}||\xi^{N\omega}(t)||.$ On a immédiatement pour tout $\alpha>0:$
\[
\mathbb{E}[e^{\alpha||\xi^{N\omega}(T)||^{2}}]=\mathbb{E}\left[
e^{\alpha\underset{|s|\leq N}{\sum}b_{s}^{2}\beta_{s}^{\omega}(T)^{2}}\right]
=%
{\displaystyle\prod\limits_{|s|\leq N}}
\mathbb{E}[e^{\alpha b_{s}^{2}\beta_{s}^{\omega}(T)^{2}}],
\]
car les processus $\{\beta_{s}\}$ sont indépendants. Comme $\beta
_{s}(T)=\mathcal{N}(0,\sqrt{T}),$ nous avons que%
\[
\mathbb{E}[e^{\alpha b_{s}^{2}\beta_{s}^{\omega}(T)^{2}}]=\frac{1}{\sqrt{2\pi
T}}\int_{%
\mathbb{R}
}e^{\alpha b_{s}^{2}x^{2}}e^{-\frac{x^{2}}{2T}}dx=\frac{1}{\sqrt{1-2T\alpha
b_{s}^{2}}},\quad\text{si }\alpha<\frac{1}{2Tb_{s}^{2}}.
\]
En notant $b_{\max}=\sup_{s\geq1}|b_{s}|,$ on obtient que%
\begin{equation}
\mathbb{E}[e^{\alpha||\xi^{N\omega}(T)||^{2}}]=e^{-\frac{1}{2}\underset
{|s|\leq N}{\sum}\ln(1-2T\alpha b_{s}^{2})},\text{ si }\alpha<\frac
{1}{2Tb_{\max}^{2}}.\label{Mexp2}%
\end{equation}
Comme pour tout $|x|\leq\frac{1}{2}$ nous avons $-\ln(1-x)\leq2x$, alors
(\ref{Mexp2}) implique
\begin{equation}
\mathbb{E}[e^{\alpha||\xi^{N\omega}(T)||^{2}}]\leq e^{\underset{|s|\leq
N}{\sum}2T\alpha b_{s}^{2}}=e^{2T\alpha B^{(N)}},\text{ si }\alpha<\frac
{1}{2Tb_{\max}^{2}},\label{C}%
\end{equation}
avec $B^{(N)}=\underset{|s|\leq N}{\sum}b_{s}^{2}\leq B_{0}.$ Or, pour tout
$p\geq1$ et $\alpha>0$ il existe $C(p,\alpha)>0$ telle que pour chaque $x\geq0$,
on a $x^{p}\leq Ce^{\alpha x^{2}}$. Donc, en utilisant l'inégalité (\ref{C}), nous
obtenons
\[
\mathbb{E}[||\xi^{N\omega}(T)||^{p}]\leq C(p,T,B_{0}).
\]
Par l'inégalité de Doob (\ref{doob}) (appendice A), cette inégalité implique que :
\begin{equation}
\mathbb{E}[\sup_{t\in(0,T)}||\xi^{N\omega}(t)||]^{p}\leq C^{\prime}%
(p,T,B_{0}),\quad\text{pour }p>1.\label{C'}%
\end{equation}
L'inégalité pour $p=1$ suit de (\ref{C'}) avec $p=2.$

\textit{Étape 2 : passage à la limite}. Pour $N<M,$ considérons
\[
\xi_{N}^{M}:=\xi^{M\omega}-\xi^{N\omega}=\underset{N<|s|\leq M}{\sum}%
b_{s}\beta_{s}^{\omega}e_{s}.
\]
Comme $\mathbb{E}||\xi_{N}^{M}(t)||^{2}=\underset{N<|s|\leq M}{\sum}b_{s}^{2}%
$, alors par l'inégalité de Doob, nous avons
\begin{equation}
\mathbb{E}\left[  ||\xi_{N}^{M}||_{X^{T}}^{2}\right]  =\mathbb{E}\left[
\sup_{t\in\lbrack0,T]}||\xi_{N}^{M}(t)||^{2}\right]  \leq4\underset{N<|s|\leq
M}{\sum}b_{s}^{2}.\label{NM}%
\end{equation}
On considère l'espace de Hilbert $\mathbb{L}_{2}=L_{2}(\Omega;X^{T}).$ Par
(\ref{NM}), pour $N<M,$ on a $||\xi^{M\omega}-\xi^{N\omega}||_{\mathbb{L}_{2}%
}^{2}\leq4\underset{N<|s|\leq M}{\sum}b_{s}^{2}.$ Comme $\sum b_{s}^{2}%
=B_{0}<\infty$, alors $\{\xi^{M\omega}\}\subset\mathbb{L}_{2}$ est une suite
de Cauchy et $\xi^{M\omega}\rightarrow\xi^{\omega}\in\mathbb{L}_{2}$ quand
$M\rightarrow+\infty.$ Comme la convergence dans $\mathbb{L}_{2}$ implique la
convergence p.p pour une sous-suite, alors, il existe un ensemble
négligeable $Q^{T}$ tel que $\xi^{M_{j}\omega}\rightarrow\xi^{\omega}$ dans
$X^{T}$ pour une sous-suite $M_{j}\rightarrow+\infty$ et $\omega\notin Q^{T}.$
Soit $Q=Q^{1}\cup Q^{2}\cup\cdots.$ Par le procédé diagonale de Cantor, on
construit une suite $M_{j}\rightarrow+\infty$ telle que $\xi^{M_{j}\omega
}\rightarrow\xi^{\omega}=:\underset{s\in%
\mathbb{Z}
^{\ast}}{\sum}b_{s}\beta_{s}^{\omega}e_{s}$ dans $X^{T}$, $T=1,2,\cdots$ et
pour tout $\omega\notin Q.$

En utilisant cette convergence, la relation (\ref{C}) et le théorème de Beppo-Levi,
on obtient l'inégalité (\ref{aa}). De même, (\ref{C'}) et le théorème de
Beppo-Levi nous donnent (\ref{bb}).
\end{proof}

Dans la suite, on désigne par $H^{m}$ l'espace de Sobolev défini par%
\[
H^{m}=\{v\in H,v^{(m)}\in H\},
\]
où $u^{(m)}:=\partial_{x}^{m}u$ représente la dérivée faible en $x$ à l'ordre
$m$ de $u.$ On munit $H^{m}$ du produit scalaire homogène
\[
\langle u,v\rangle_{m}=\int_{\mathbb{S}^{1}}u^{(m)}(x)v^{(m)}(x)dx,
\]
On note par $||\cdot||_{m}$ la norme associée au produit scalaire. Notons que
$\partial_{x}:H^{m+1}\rightarrow H^{m}$ est un isomorphisme et $||\partial
_{x}^{k}u||_{m}=||u||_{m+k},$ pour tout $m,k\in%
\mathbb{N}
.$

Si $u\in H$ alors $u$ s'écrit dans la base trigonométrique (\ref{base}) comme
$u(x)=\underset{s\in%
\mathbb{Z}
^{\ast}}{\sum}u_{s}e_{s}(x),$ et la norme de $u$ dans $H^{m}$ est%
\[
||u||_{m}^{2}=(2\pi)^{m}\sum_{s\in%
\mathbb{Z}
^{\ast}}|s|^{2m}|u_{s}|^{2}.
\]
On utilise cette caractérisation pour définir $H^{m}$ pour tout $m\geq0:$%
\[
H^{m}=\{u\in H,||u||_{m}<\infty\}.
\]
Pour $m<0$, on définit $H^{m}$ comme le complété de $H$ par rapport à la norme
$||\cdot||_{m}.$\newline On rappelle l'injection de Sobolev :%
\begin{equation}
H^{m}\hookrightarrow\mathcal{C}^{k}(\mathbb{S}^{1})\Longleftrightarrow
m>k+\frac{1}{2},\label{sob}%
\end{equation}
et que l'espace $H^{m}$ pour $m>\frac{1}{2}$ est une algèbre Hilbertienne :%
\begin{equation}
||uv||_{m}\leq c_{m}||u||_{m}||v||_{m}.\label{alg}%
\end{equation}
Pour $T>0,$ on note par $X_{m}^{T}$ l'espace de Banach $\mathcal{C}%
(0,T;H^{m})$ muni de la norme $||u||_{X_{m}^{T}}=\underset{t\in\lbrack
0,T]}{\sup}||u(t)||_{m}$. Avec ces notations, on a le théorème suivant :

\begin{theorem}
\label{th2}Soit $m\in%
\mathbb{N}
$ et supposons que $B_{m}<\infty.$ Alors pour chaque $\omega\in\Omega$ et pour
tout $T>0$ nous avons $\xi^{\omega}\in X_{m}^{T}$, et il existe $\alpha
(T,B_{m})>0$ telle que%
\begin{align}
\mathbb{E[}e^{\alpha||\xi^{\omega}(T)||_{m}^{2}}\mathbb{]} &  \leq
C_{m}^{\prime},\label{csqfern}\\
\mathbb{E}[\sup_{t\in(0,T)}||\xi^{\omega}(t)||_{m}]^{p} &  \leq C_{m}%
^{\prime\prime},\quad\forall p\geq1,
\end{align}
pour des constantes convenables $C_{m}^{\prime}(T,B_{m})>0$, $C_{m}%
^{\prime\prime}(p,T,B_{0})>0.$
\end{theorem}

Ici, comme pour le théorème \ref{th1}, nous avons modifié $\xi^{\omega}$ sur
un ensemble négligeable.

La démonstration est similaire à celle du théorème \ref{th1}.

\section{Problème de Cauchy}

Dans cette section on étudie le problème aux limites périodique (\ref{B}),
(\ref{B*}) avec la condition initiale.%
\begin{equation}%
\operatorname{u}%
(0,x)=%
\operatorname{u}%
_{0}(x).\label{B2}%
\end{equation}

\subsection{Préliminaires\label{preliminaire}}

Soit $m\in%
\mathbb{N}
$. On considère l'équation de la chaleur
\begin{equation}
\left\{
\begin{array}
[c]{c}%
v_{t}(t,x)-\nu v_{xx}(t,x)=\xi_{t}(t,x),\quad\xi\in X_{m}^{T}\\
v(0,x)=%
\operatorname{u}%
_{0}(x).
\end{array}
\right. \tag{CH}\label{CH}%
\end{equation}
On dit que $v\in X_{m}^{T}$ est une solution de (\ref{CH}) si pour tout
$t\in\lbrack0,T]:$%
\[
v(t)-\nu\int_{0}^{t}v_{xx}(s)ds=%
\operatorname{u}%
_{0}+\xi(t).
\]

Le lemme suivant est un résultat élémentaire des équations aux dérivées
partielles paraboliques. Pour le démontrer, on décompose la force $\xi$ et $v$
dans la base trigonométrique (\ref{base}) et on retrouve les coefficients de
Fourier de $v.$

\begin{lemma}
\label{lem3}Soit $m\geq0$ et $T>0.$ Alors pour tout $\xi\in X_{m}^{T}$ et $%
\operatorname{u}%
_{0}\in H^{m}$, il existe une unique solution $v$ $\in X_{m}^{T}$ de
(\ref{CH}). De plus l'application $H^{m}\times X_{m}^{T}\rightarrow X_{m}%
^{T},$ $(%
\operatorname{u}%
_{0},\xi)\longmapsto v$ est linéaire et continue.
\end{lemma}

\subsection{Décomposition des solutions de (\ref{B})}

On écrit maintenant une solution $%
\operatorname{u}%
$ de (\ref{B}), (\ref{B2}) comme%
\[%
\operatorname{u}%
(t,x)=w(t,x)+v(t,x),
\]
où $v$ est la solution de (\ref{CH}). Ainsi, $w$ vérifie l'équation de Burgers
perturbée :%
\begin{equation}
\left\{
\begin{array}
[c]{c}%
w_{t}-\nu w_{xx}+\frac{1}{2}((w+v)^{2})_{x}=0,\\
w(0)=0.
\end{array}
\right. \tag{BP}\label{BP}%
\end{equation}
L'avantage de se ramener de l'équation (\ref{B}) à (\ref{BP}) est le fait que
tous les coefficients de cette dernière sont réguliers. Dans la suite on
résout l'équation (\ref{BP}). On commence par deux lemmes qui portent sur des
inégalités fonctionnelles.

\begin{lemma}
\label{GNtermeL}(inégalité de Gagliardo-Niremberg \cite{Tay3})

Pour tout $(m,(r,q))\in%
\mathbb{N}
^{\ast}\times\lbrack1,\infty]^{2}$ et $\beta\in\lbrack0,m-1],$ il existe
$C(\beta,q,r,m)>0$ telle que
\begin{equation}
|w^{(\beta)}|_{r}\leq C||w||_{m}^{\theta}|w|_{q}^{1-\theta},\quad\theta
(\beta,q,r,m)=\frac{\beta+\frac{1}{q}-\frac{1}{r}}{m+\frac{1}{q}-\frac{1}{2}%
}.\label{GN}%
\end{equation}

\end{lemma}

\begin{lemma}
\label{GNtermeNL}Soit $m\in%
\mathbb{N}
^{\ast}$, $q\in\lbrack1,\infty]$ et $w\in H^{m+1}$. Alors, il existe $C(m)>0$
telle que%
\begin{equation}
|\langle\partial_{x}^{2m}w,\partial_{x}w^{2}\rangle|\leq C||w||_{m+1}%
^{1+\theta}|w|_{q}^{2-\theta},\quad\theta(m,q)=\frac{m+\frac{2}{q}-\frac{1}%
{2}}{m+\frac{1}{q}+\frac{1}{2}},\label{GNNL}%
\end{equation}

\end{lemma}

\begin{proof}
Par la formule de Leibniz, nous avons%
\begin{equation}
|\langle\partial_{x}^{2m}w,\partial_{x}w^{2}\rangle|\leq C(m)\sum_{n=0}%
^{m}\int|w^{(n)}w^{(m-n)}w^{(m+1)}|dx.\label{leibnizI1}%
\end{equation}
Nous majorons l'intégrale du membre de droite de (\ref{leibnizI1}) par
l'inégalité de Hölder. On obtient
\[
\int|w^{(n)}w^{(m-n)}w^{(m+1)}|dx\leq|w^{(n)}|_{p_{1}}|w^{(m-n)}|_{p_{2}%
}||w||_{m+1},
\]
avec $\frac{1}{p_{1}}+\frac{1}{p_{2}}=\frac{1}{2}.$ D'où en utilisant (\ref{GN}) avec
$m:=m+1,$ et $r=n$ puis $r=m-n,$ nous avons
\[
\int|w^{(n)}w^{(m-n)}w^{(m+1)}|dx\leq C||w||_{m+1}^{1+\theta}|w|_{q}%
^{2-\theta},
\]
où $\theta=\theta(n,q,p_{1},m+1)+\theta(m-n,q,p_{2},m+1)=\frac{m+\frac{2}%
{q}-\frac{1}{2}}{m+\frac{1}{q}+\frac{1}{2}}.$ De cette dernière inégalité et de
la relation (\ref{leibnizI1}), on déduit (\ref{GNNL}).
\end{proof}

\begin{theorem}
\label{thwhm}Soit $m\geq1$, $B_{m}<\infty$, $T>0$ et $v\in X_{m}^{T}$. Alors
il existe une unique solution $w\in X_{m}^{T}$ de (\ref{BP}) et il existe
$C_{m}(T,\nu,||v||_{X_{m}^{T}})>0$ telle que
\begin{equation}
||w(t)||_{X_{m}^{T}}+\int_{0}^{T}||w(t)||_{m+1}^{2}dt\leq C_{m}.\label{wHm}%
\end{equation}

\end{theorem}

\begin{proof}
\textit{Étape 0 : unicité. }Soient $w^{\prime},w^{\prime\prime}\in X_{m}^{T}$
solutions de (\ref{BP}) et $v\in X_{m}^{T}$. Alors $w:=w^{\prime}%
-w^{\prime\prime}$ vérifie l'équation%
\[
w_{t}-\nu w_{xx}=\frac{1}{2}((w^{\prime}+w^{\prime\prime}+2v)w)_{x},\quad
w(0)=0.
\]
En prenant le produit scalaire $L_{2}$ de cette équation avec $w$, et en
faisant des intégrations par partie, nous obtenons
\[
\frac{1}{2}\frac{d}{dt}||w(t)||^{2}+\nu||w(t)||_{1}^{2}=-\frac{1}{2}%
\int(w^{\prime}(t)+w^{\prime\prime}(t)+2v(t))w(t)w_{x}(t)dx.
\]
En utilisant l'inégalité de Cauchy-Schwarz puis l'injection de $H^{1}$ dans
$L_{\infty},$ et l'inégalité de Young, le membre à droite est majoré par
\begin{align*}
\int|(w^{\prime}(t)+w^{\prime\prime}(t)+2v)w(t)w_{x}(t)|dx  & \leq
||(w^{\prime}(t)+w^{\prime\prime}(t)+2v(t))w(t)||\cdot||w_{x}(t)||\\
& \leq C(\frac{C}{2\nu}||w(t)||^{2}+\frac{\nu}{2C}||w(t)||_{1}^{2}).
\end{align*}

Donc $\frac{d}{dt}||w(t)||^{2}\leq C_{1}||w(t)||^{2}$. 
Comme $w(0)=0,$ le lemme de Gronwall implique que $||w(t)||^{2}=0$. D'où l'unicité.

\textit{Étape 1 : estimations à priori.} Supposons que$\ v\in X_{m}^{T}$ et
que $w\in X_{m}^{T}$ est la solution de (\ref{BP}) et montrons d'abord
l'inégalité (\ref{wHm}), où à gauche, $m$ est remplacer par $m=0$. Après avoir multiplié (\ref{BP}) par $w$
et intégré en espace, le membre de gauche s'écrit comme
\[
\int(w_{t}-\nu w_{xx})wdx=\frac{1}{2}\frac{d}{dt}||w(t)||^{2}+\nu
||w(t)||_{1}^{2},
\]
et le membre de droite devient%
\[
-\frac{1}{2}\int((w+v)^{2})_{x}wdx=\frac{1}{2}\int(w^{2}w_{x}+v^{2}%
w_{x}+2vww_{x})dx=\frac{1}{2}\int(v^{2}w_{x}+2vww_{x})dx.
\]
Par l'inégalité de Cauchy-Schwarz puis celle de Young, et utilisant
(\ref{sob}) avec $k=0$, nous obtenons que le terme à droite est borné par
\begin{equation}
\frac{\nu}{4}||w(t)||_{1}^{2}+c||v||_{X_{1}^{T}}^{4}+\frac{\nu}{4}%
||w(t)||_{1}^{2}+c||v||_{X_{1}^{T}}^{2}||w(t)||^{2},\quad c=c(\nu).\nonumber
\end{equation}
Donc, on a
\[
\frac{d}{dt}||w(t)||^{2}+\nu||w(t)||_{1}^{2}\leq c_{1}||v||_{X_{1}^{T}}%
^{2}||w(t)||^{2}+c_{2}||v||_{X_{1}^{T}}^{4}.
\]
Par le lemme de Gronwall, il en résulte que
\begin{equation}
||w(t)||^{2}\leq tc_{2}||v||_{X_{1}^{T}}^{4}e^{c_{1}||v||_{X_{1}^{T}}^{2}%
t}\leq c_{0}(T),\quad0\leq t\leq T.\label{estimationwm=0}%
\end{equation}
Maintenant, on va utiliser (\ref{estimationwm=0}) pour estimer $||w(t)||_{m}.
$ Multiplions l'équation dans (\ref{BP}) par $w^{(2m)}.$ Alors, par une
intégration par parties en espace, nous obtenons%
\begin{align*}
& \frac{1}{2}\frac{d}{dt}||w(t)||_{m}^{2}+\nu||w(t)||_{m+1}^{2}\leq\left\vert
\langle\frac{d^{m}}{dx^{m}}(w(t)+v(t))^{2},\frac{d^{m+1}}{dx^{m+1}}%
w(t)\rangle\right\vert \\
& \leq\underset{=:I_{1}}{\underbrace{\left\vert \langle w^{2}(t)^{(1)}%
,w(t)^{(2m)}\rangle\right\vert }}+2\underset{=:I_{2}}{\underbrace{\left\vert
\langle(wv)(t)^{(m)},w(t)^{(m+1)}\rangle\right\vert }}\underset{=:I_{3}%
}{+\underbrace{\left\vert \langle v^{2}(t)^{(m)},w(t)^{(m+1)}\rangle
\right\vert }}.
\end{align*}
Par (\ref{GNNL}) (avec $\theta=\theta(m,2)$), nous avons%
\[
|I_{1}|\leq C_{1}||w(t)||_{m+1}^{1+\theta}||w(t)||^{2-\theta},\quad
\theta=\frac{2m+1}{2m+2}.
\]
L'inégalité de Young appliquée au membre à droite de cette inégalité implique que
\[
I_{1}\leq\frac{\nu}{4}||w(t)||_{m+1}^{2}+C_{1}^{^{\prime}}||w(t)||^{c_{1}%
^{\prime}}.
\]
Nous avons aussi%
\[
I_{2}\leq\frac{\nu}{4}||w(t)||_{m+1}^{2}+C_{2}^{^{\prime}}(||v||_{X_{m}^{T}%
})||w(t)||^{c_{2}^{\prime}},\quad I_{3}\leq\frac{\nu}{4}||w(t)||_{m+1}%
^{2}+C_{3}^{^{\prime}}(||v||_{X_{m}^{T}})||w(t)||^{c_{3}^{\prime}}.
\]
Pour estimer $I_{2}=2\langle wv(t),w_{x}(t)\rangle_{m},$ on utilise
(\ref{alg}) et (\ref{GN}) :
\begin{align*}
2\langle wv(t),w_{x}(t)\rangle_{m}  & \leq C||wv(t)||_{m}||w(t)||_{m+1}\leq
C_{1}||w(t)||_{m}||v(t)||_{m}||w(t)||_{m+1}\\
& \leq C_{m}||v(t)||_{m}||w(t)||^{1-\frac{m}{m+1}}||w(t)||_{m+1}^{1+\frac
{m}{m+1}},
\end{align*}
et on conclut par l'inégalité de Young. De manière similaire, on dérive une
estimation du même type pour $I_{3}.$ En utilisant (\ref{estimationwm=0}), on a
\begin{equation}
\frac{1}{2}\frac{d}{dt}||w(t)||_{m}^{2}+\frac{\nu}{4}||w(t)||_{m+1}^{2}\leq
C^{^{\prime}}(||v||_{X_{m}^{T}})c_{0}(T)^{c^{\prime}}.\label{gronwallpret}%
\end{equation}
Finalement, l'intégration de cette relation sur $[0,T]$ permet de conclure à
la validité de (\ref{wHm}) pour tout $m\geq1.$

\textit{Étape 2 : approximation de Galerkin. }Pour tout $N\in%
\mathbb{N}
,$ soit
\[
H^{(N)}=vect\{e_{j},|j|\leq N\}\subset H,
\]
et $\{w^{N}\}$ la suite des solutions approchées de Galerkin du problème
(\ref{BP}) définie par
\[
w^{N}(t)=\sum_{|s|\leq N}\alpha_{s}(t)e_{s},
\]
où les $\alpha_{s}(t)$ sont à déterminer. En effet, en substituant $w^{N}$
dans l'équation de (\ref{BP}) et en prenant le produit scalaire $L_{2}$ dans
les deux membres de l'égalité avec $e_{j}$, $|j|\leq N,$ nous obtenons une EDO
de la forme :%
\begin{equation}
\frac{d\alpha_{j}}{dt}(t)=-(2\pi j)^{2}\alpha_{j}(t)+P_{j}(\alpha
(t),t),\quad\alpha_{j}(0)\equiv0,\quad|j|\leq N,\label{edo}%
\end{equation}
où $P_{j}(\alpha,t)$ est un polynôme fini en $\alpha=(\alpha_{k})_{|k|\leq N}$
et en $(v_{k}(t))_{k\in%
\mathbb{Z}
^{\ast}}.$ Donc, $P_{j}(\alpha,t)$ est une fonction lisse en $\alpha$ et
continue en $t$, on conclut à l'existence d'une unique solution $w^{N}%
\in\mathcal{C}^{1}([0,T^{N}[;H^{(N)})$ de (\ref{edo}) , où $T^{N}$ est le
temps maximal d'existence.

Exactement comme dans l'étape 1, on montre que $w^{N}$ vérifie l'inégalité
(\ref{wHm}) et donc $||w^{N}(t)||$ est borné uniformément en $N$. En
particulier $\underset{t\rightarrow T^{N}}{\lim\sup}||w^{N}(t)||<\infty$ ce
qui implique que $T^{N}=T$.

\textit{Étape 3 : Convergence. }En utilisant l'équation (\ref{BP}), on trouve
que $\{w^{N}\}$ est borné dans l'espace de Hilbert
\[
W^{m}=\{u\in L_{2}(0,T;H^{m+1}):u_{t}\in L_{2}(0,T;H^{m-1})\}.
\]
Comme $W^{m}\Subset L_{2}(0,T;H^{m})$ \cite[Théorème 5.1]{Lions}, alors il
existe une sous-suite $\left\{  w^{N_{j}}\right\}  $ qui converge vers $w\in
W^{m}$ faiblement dans $W^{m}$ et fortement dans $L_{2}(0,T;H^{m}).$ On
procède comme dans \cite{Lions} pour vérifier que $w$ est solution de
(\ref{BP}) et satisfait bien les estimations à priori. De plus, $W^{m}\subset
X_{m}^{T}$ \cite{LM}, donc $w\in X_{m}^{T}$.

La suite $\{w^{N}\}$ est bornée dans $X_{m}^{T}$ et $H^{m}$ est relativement
compact dans $H^{m-1}$%
. Comme $\{\partial_{t}w^{N}\}$ est borné dans $L_{2}(0,T;H^{m-1})$, alors la
suite $\{w^{N}\}$ est équicontinue dans $H^{m-1}$. Donc, par le théorème
d'Ascoli, $\{w^{N}\}$ est précompact dans $X_{m-1}^{T}$, et $\left\{
w^{N_{j}}\right\}  $ converge vers $w$ dans $X_{m-1}^{T}$ (Nous utiliserons
cette convergence plus tard).
\end{proof}

\begin{remark}
\label{nlch}Une version simplifiée de cet argument nous montre que pour tout
$T>0,$ $m\geq1,$ $v\in X_{m}^{T}$ et $z\in X_{m-1}^{T}$, l'équation linéaire
\begin{equation}
\left\{
\begin{array}
[c]{c}%
w_{t}-\nu w_{xx}+(vw)_{x}=z,\\
w_{|t=0}=0,
\end{array}
\right. \label{pnll}%
\end{equation}
admet une unique solution $w\in X_{m}^{T}$ vérifiant :%
\begin{equation}
||w||_{X_{m}^{T}}\leq C||z||_{X_{m-1}^{T}},\quad C=C(\nu,T,||v||_{X_{m}^{T}%
}).\label{estipnll}%
\end{equation}

\end{remark}

\subsection{Existence et unicité pour Burgers}

D'après le lemme \ref{lem3} et le théorème \ref{thwhm} on peut énoncer le résultat
suivant :

\begin{theorem}
\label{burgersok}Pour tout $T>0$ $m\geq1$, $%
\operatorname{u}%
_{0}\in H^{m}$ et $\xi\in X_{m}^{T}$ il existe une unique solution $%
\operatorname{u}%
$ dans $X_{m}^{T}$ de (\ref{B}), (\ref{B*}),(\ref{B2}). De plus, $%
\operatorname{u}%
$ satisfait%
\begin{equation}
||%
\operatorname{u}%
||_{X_{m}^{T}}^{2}+\int_{0}^{T}||%
\operatorname{u}%
(t)||_{m+1}^{2}dt\leq C_{m},\quad C_{m}=C_{m}(T,\nu,||%
\operatorname{u}%
_{0}||_{m},||\xi||_{X_{m}^{T}})>0
\end{equation}

\end{theorem}

\subsection{Analyticité du flot}

Désormais, nous étudions la régularité de l'application
\[
H^{m}\times X_{m}^{T}\rightarrow X_{m}^{T},(%
\operatorname{u}%
_{0},\xi)\longmapsto%
\operatorname{u}%
\]
où $%
\operatorname{u}%
$ est la solution de (\ref{B}), (\ref{B2}) avec $\eta=\partial_{t}\xi$.

Considérons l'équation de la chaleur
\begin{equation}
w_{t}-\nu w_{xx}=z,\text{\quad}w_{|t=0}=0.\label{pnl}%
\end{equation}
Par la remarque \ref{nlch}, si $m\geq1$ et $z\in X_{m-1}^{T}$ alors
(\ref{pnl}) a une unique solution $w\in X_{m}^{T}$ et
\begin{equation}
||w||_{X_{m}^{T}}\leq C||z||_{X_{m-1}^{T}},\quad C=C(m,\nu,T).\label{estipnl}%
\end{equation}
Soit $m\geq1$ et notons par $X_{m,0}^{T} l'ensemble \{w\in X_{m}^{T}:w(0)=0\},$ et soit
$\mathcal{L}$ l'application définie par
\[
\mathcal{L}:X_{m+1,0}^{T}\cap\mathcal{C}^{1}(0,T;H^{m-1})\rightarrow
X_{m-1}^{T},\text{\quad}w\longmapsto w_{t}-\nu w_{xx}.
\]
Par (\ref{estipnl}), l'application $\mathcal{L}^{-1}:X_{m-1}^{T}\rightarrow
X_{m,0}^{T}$ est une immersion continue. On définit l'espace $Z_{m}%
^{T}=\mathcal{L}^{-1}(X_{m-1}^{T})$ muni de la norme $||w||_{Z_{m}^{T}%
}=||\mathcal{L}w||_{X_{m-1}^{T}}$ qui est un espace de Banach. Cet espace
s'injecte continûment dans $X_{m,0}^{T}$ et $\mathcal{L}$ définit une
isométrie entre $Z_{m}^{T}$ et $X_{m-1}^{T}.$

Considérons l'équation (\ref{BP}) avec $v\in X_{m}^{T}.$ Par le théorème
\ref{thwhm} la solution $w\in X_{m}^{T}.$ Donc, $\partial_{x}((w+v)^{2})\in
X_{m-1}^{T}$ et $w=\mathcal{L}^{-1}\left(  -\frac{1}{2}\partial_{x}%
((w+v)^{2})\right)  \in Z_{m}^{T}.$ Écrivons (\ref{BP}) comme%

\begin{equation}
\Phi(w,v)=0,\quad\Phi(w,v)=w+\frac{1}{2}\mathcal{L}^{-1}\circ\partial
_{x}((w+v)^{2}).\label{phi=0}%
\end{equation}
Notons que pour chaque $v\in X_{m}^{T}$ l'équation $\Phi(w,v)=0$ possède une
unique solution $w\in Z_{m}^{T}.$

\begin{lemma}
\label{phianaly}Pour tout $T>0$ et $m\geq1$, $\Phi$ définit une application
analytique de $Z_{m}^{T}\times X_{m}^{T}$ dans $Z_{m}^{T}$.
\end{lemma}

\begin{proof}
L'application $X_{m}^{T}\rightarrow X_{m}^{T},w\longmapsto(v+w)^{2}$ est
analytique car polynomiale et continue \cite{Car}. Comme $Z_{m}^{T}%
\hookrightarrow X_{m,0}^{T}$ alors l'application
\[
Z_{m}^{T}\times X_{m}^{T}\rightarrow Z_{m}^{T},(w,v)\longmapsto(v+w)^{2}%
\longmapsto\partial_{x}((w+v)^{2})\longmapsto\mathcal{L}^{-1}\circ\partial
_{x}((w+v)^{2})
\]
est aussi analytique. Donc $\Phi$ est analytique.
\end{proof}

Pour tout $(w,v)\in Z_{m}^{T}\times X_{m}^{T}$, la différentielle de $\Phi$ en
$w\in Z_{m}^{T}$ évaluée en $h\in Z_{m}^{T}$ est l'application $d_{w}%
\Phi(w,v)(h):Z_{m}^{T}\rightarrow Z_{m}^{T},$ telle que $d_{w}\Phi
(w,v)(h)=h+\mathcal{L}^{-1}\circ\partial_{x}(w+v)h$ \cite{Car}. Cette
application est continue par le lemme \ref{phianaly}.

\begin{lemma}
\label{isomorphisme}Pour tout $T>0,$ $m\geq1$ et $(w,v)\in Z_{m}^{T}\times
X_{m}^{T}$, $d_{w}\Phi$ est un isomorphisme de $Z_{m}^{T}.$
\end{lemma}

\begin{proof}
Il est clair que pour tout $w,h,g\in Z_{m}^{T}$ et $v\in X_{m}^{T},$ on a
$d_{w}\Phi(h)=g$ si et seulement si
\begin{equation}
h_{t}-\nu h_{xx}+\partial_{x}(w+v)h=\mathcal{L}g.\label{equaLh}%
\end{equation}
Comme $(w+v,$ $\mathcal{L}g)\in X_{m}^{T}\times X_{m-1}^{T}$, la remarque
\ref{nlch} implique qu'il existe une unique solution $h\in X_{m,0}^{T}$ de
(\ref{equaLh}) et que l'application $X_{m}^{T}\rightarrow X_{m,0}^{T},$
$\mathcal{L}g\longmapsto h$ est continue.
\end{proof}

Finalement, par les lemmes \ref{phianaly} et \ref{isomorphisme}, et la version
analytique du théorème des fonctions implicites \cite[Appendix B]{PT}, nous
obtenons que l'unique solution $w\in Z_{m}^{T}$ de l'équation $\Phi(w,v)=0$
est une fonction analytique de $v\in X_{m}^{T}.$ Comme nous avons écrit la
solution $%
\operatorname{u}%
$ de l'équation de Burgers comme $%
\operatorname{u}%
=v+w,$ on a par le lemme \ref{lem3}, le

\begin{theorem}
\label{th6}Soit $m\geq1$ et $T>0$. Si $%
\operatorname{u}%
$ est la solution de (\ref{B}),(\ref{B2}) avec $\eta=\partial_{t}\xi$, alors
l'application $\mathcal{M}:H^{m}\times X_{m}^{T}\rightarrow X_{m}^{T},(%
\operatorname{u}%
_{0},\xi)\longmapsto%
\operatorname{u}%
$ est analytique. En particulier $\mathcal{M}$ est continue.
\end{theorem}

On munit l'espace $H^{m}\times X_{m}^{T}$ de la norme :
\[
|||(f,h)|||=||f||_{H^{m}}+||h||_{X_{m}^{T}}\text{.}%
\]

\begin{proposition}
\label{Mlip}l'application $\mathcal{M}$ du théorème \ref{th6} est localement
Lipschitzienne\footnote{C'est à dire, il existe une fonction $C\in\mathcal{C}(%
\mathbb{R}
^{+};%
\mathbb{R}
^{+})$ telle que pour tout $%
\operatorname{u}%
_{0}%
\operatorname{u}%
_{0}^{^{\prime}}\in H^{m}$ et $\xi,\xi^{^{\prime}}\in X_{m}^{T}:$
\[
||\mathcal{M}(%
\operatorname{u}%
_{0},\xi)-\mathcal{M}(%
\operatorname{u}%
_{0}^{^{\prime}},\xi^{^{\prime}})||_{X_{m}^{T}}\leq C(\max(|||%
\operatorname{u}%
_{0},\xi|||,|||%
\operatorname{u}%
_{0}^{\prime},\xi^{\prime}|||)|||(%
\operatorname{u}%
_{0},\xi)-(%
\operatorname{u}%
_{0}^{^{\prime}},\xi^{^{\prime}})|||.
\]
}.
\end{proposition}

\begin{proof}
Soit $m\geq1$ et $T>0$. Nous notons par $H^{m%
\mathbb{C}
}$ l'espace de Sobolev complexe (la complixification de l'espace $H^{m}$). On
définit $X_{m}^{T%
\mathbb{C}
}$ et $Z_{m}^{T%
\mathbb{C}
}$ comme les analogues complexes des espaces $X_{m}^{T}$ et $Z_{m}^{T}.$

En utilisant la version analytique complexe du théorème des fonctions
implicites \cite[Appendix B]{PT}, nous trouvons que si $w^{\prime}\in Z_{m}^{T%
\mathbb{C}
}$, $v^{\prime}\in X_{m}^{T%
\mathbb{C}
}$ et $\Phi(w^{\prime},v^{\prime})=0$ (cf. (\ref{phi=0})) alors $w^{\prime}$
est une fonction analytique de $v^{\prime}$, et il existe $\varepsilon
=\varepsilon(v^{\prime})>0$ telle que l'application $v^{\prime}\longmapsto
w^{\prime}(v^{\prime})$ vérifiant $\Phi(w^{\prime}(v^{\prime}),v^{\prime})=0$
se prolonge en une application analytique complexe :%
\[
\Phi:B_{v^{\prime}}^{\varepsilon}:=\{v\in X_{m}^{T%
\mathbb{C}
}:||v-v^{\prime}||_{X_{m}^{T%
\mathbb{C}
}}<\varepsilon\}\rightarrow Z_{m}^{T%
\mathbb{C}
},
\]
qui est bornée en norme par une constance $K(v^{\prime})>0.$ De plus, nous
pouvons choisir $\varepsilon=\varepsilon(||v^{\prime}||_{X_{m}^{T%
\mathbb{C}
}})$ et $K=K(||v^{\prime}||_{X_{m}^{T%
\mathbb{C}
}})$ telles que $\varepsilon,K$ soient des fonctions continues positives. Donc,
il existe $\widetilde{\varepsilon},\widetilde{K}\in\mathcal{C}(%
\mathbb{R}
^{+},%
\mathbb{R}
^{+})$ tels que l'application $\mathcal{M}$ du théorème \ref{th6} se prolonge
en une application analytique complexe :%
\begin{equation}
\mathcal{M}:\mathcal{O}^{%
\mathbb{C}
}\rightarrow X_{m}^{T%
\mathbb{C}
},\label{MM}%
\end{equation}
où $\mathcal{O}^{%
\mathbb{C}
}=\{%
\operatorname{u}%
_{0}\in H^{m%
\mathbb{C}
}:||\operatorname{Im}%
\operatorname{u}%
_{0}||_{m}\leq\widetilde{\varepsilon}(||\operatorname{Re}%
\operatorname{u}%
_{0}||_{m})\}\times\{\xi\in X_{m}^{T%
\mathbb{C}
}:||\operatorname{Im}\xi||_{X_{m}^{T%
\mathbb{C}
}}\leq\widetilde{\varepsilon}(||\operatorname{Re}\xi||_{X_{m}^{T%
\mathbb{C}
}})\}$, et telle que $||\mathcal{M}(%
\operatorname{u}%
_{0},\xi)||_{X_{m}^{T%
\mathbb{C}
}}\leq\widetilde{K}(||%
\operatorname{u}%
_{0}||_{H^{m%
\mathbb{C}
}}+||\xi||_{X_{m}^{T%
\mathbb{C}
}}).$ Chaque couple $(%
\operatorname{u}%
_{0},\xi)\in H^{m}\times X_{m}^{T}$ se plonge dans $\mathcal{O}^{%
\mathbb{C}
}\subset H^{m%
\mathbb{C}
}\times X_{m}^{T%
\mathbb{C}
}$ avec son voisinage complexe de rayon $\widetilde{\varepsilon}(|||(%
\operatorname{u}%
_{0},\xi)|||)$. Donc, par l'analyticité de (\ref{MM}) et l'inégalité de Cauchy
\cite[Section VII.A]{KP}, nous avons que pour tout $%
\operatorname{u}%
_{0}\in H^{m}$ et $\xi\in X_{m}^{T}:$%
\[
||d\mathcal{M}(%
\operatorname{u}%
_{0},\xi)||_{H^{m}\times X_{m}^{T},X_{m}^{T}}\leq\min\left(  \widetilde
{\varepsilon}(||%
\operatorname{u}%
_{0}||_{m}),\widetilde{\varepsilon}(||\xi||_{X_{m}^{T}})\right)
^{-1}\widetilde{K}(||%
\operatorname{u}%
_{0}||_{H^{m%
\mathbb{C}
}}+||\xi||_{X_{m}^{T%
\mathbb{C}
}}).
\]
Donc, $\mathcal{M}_{t}$ est localement Lipschitzienne.
\end{proof}

Pour $t\in\lbrack0,T]$ et $m\geq1$, nous notons par $\mathcal{M}_{t}$
l'application
\begin{equation}
\mathcal{M}_{t}:H^{m}\times X_{m}^{T}\rightarrow H^{m},(%
\operatorname{u}%
_{0},\xi)\longmapsto%
\operatorname{u}%
(t)=\mathcal{M}_{|t=\text{const}}\text{.}\label{Mt}%
\end{equation}
Elle est analytique par le théorème \ref{th6}.

Parfois, on notera 
\[\operatorname{u}(\cdot;\operatorname{u}_{0},\xi)=\mathcal{M}(\operatorname{u}_{0},\xi)\] 
la solution de (\ref{B}), (\ref{B2}) avec $\eta=\partial_{t}\xi$, et
on notera \[%
\operatorname{u}%
^{\omega}(\cdot;%
\operatorname{u}%
_{0})=\mathcal{M}(%
\operatorname{u}%
_{0},\xi^{\omega})\] la solution de (\ref{B}), (\ref{B*}), (\ref{B2}).

\section{Principe du maximum de Kru\v{z}kov}

Le but de cette section est d'estimer $%
\operatorname{u}%
,$ $%
\operatorname{u}%
_{x}$ et $%
\operatorname{u}%
_{x}^{+}:=\max(0,%
\operatorname{u}%
_{x})$, en se basant sur la méthode de Kru\v{z}kov \cite{Kru64}. Commençons
avec un lemme évident.

\begin{lemma}
\label{lem18}Si $f\in\mathcal{C}^{1}(\mathbf{S}^{1};%
\mathbb{R}
)$ telle que $\int_{\mathbb{S}^{1}}f(x)dx=0$ et $f^{\prime}(x)\leq c^{\ast}$
pour tout $x$. Alors $|f|_{\infty}\leq c^{\ast}$ et $|f^{\prime}|_{1}%
\leq2c^{\ast}.$
\end{lemma}

Le théorème suivant nous fournit des estimations portant sur $%
\operatorname{u}%
_{x}^{+}(t),$ $%
\operatorname{u}%
(t)$ dans $L_{\infty},$ et sur $%
\operatorname{u}%
_{x}(t)$ dans $L_{1}.$

\begin{theorem}
\label{th17}Soit $B_{4}<\infty,$ $T>0,$ $t\in\lbrack0,T]$ et $%
\operatorname{u}%
(t,x)$ solution de (\ref{B}), (\ref{B2}) avec $\eta(t,x)=\partial_{t}\xi
(t,x)$, $\xi\in X_{4}^{T}$ et $%
\operatorname{u}%
(0)\in H^{1}$. Alors, il existe $C(T)>0$ indépendante de $%
\operatorname{u}%
_{0}$ et de $\nu\in(0,1],$ telle que pour tout $\theta\in(0,T]$%
\begin{align}
\sup_{\theta\leq t\leq T}|%
\operatorname{u}%
_{x}^{+}(t,\cdot)|_{\infty} &  \leq C\theta^{-1}(1+||\xi||_{X_{4}^{T}%
}),\label{a}\\
\sup_{\theta\leq t\leq T}(|%
\operatorname{u}%
(t,\cdot)|_{\infty}+|%
\operatorname{u}%
_{x}(t,\cdot)|_{1}) &  \leq C\theta^{-1}(1+||\xi||_{X_{4}^{T}}).\label{b}%
\end{align}

\end{theorem}

\begin{proof}
Pour simplifier les notations, nous supposons que $T=2$. Supposons d'abord que
$%
\operatorname{u}%
(0)\in H^{4}.$ Alors, le théorème \ref{burgersok} implique que $%
\operatorname{u}%
\in X_{4}^{T}$. On se propose d'écrire la solution $%
\operatorname{u}%
$ de (\ref{B}), (\ref{B*}),(\ref{B2}) comme
\[%
\operatorname{u}%
=\xi+v.
\]
Par conséquent, $v$ est solution de
\begin{equation}
v_{t}+%
\operatorname{u}%
\operatorname{u}%
_{x}-\nu v_{xx}=\nu\xi_{xx}.\label{equa_v}%
\end{equation}
On dérive cette équation par rapport à $x$ et on multiplie par $t,$ on trouve%
\begin{equation}
tv_{tx}+t(%
\operatorname{u}%
_{x}^{2}+%
\operatorname{u}%
\operatorname{u}%
_{xx})=\nu tv_{xxx}+\nu t\xi_{xxx}.\label{tv_x}%
\end{equation}
On pose $w:=tv_{x}$ et on réécrit (\ref{tv_x}) comme%
\begin{equation}
(w_{t}-v_{x})+t%
\operatorname{u}%
_{x}^{2}+(t%
\operatorname{u}%
\xi_{xx}+%
\operatorname{u}%
w_{x})=\nu w_{xx}+\nu t\xi_{xxx}.\label{equaw=tv_x}%
\end{equation}
Maintenant, on considère la fonction $w(t,x)$ sur le cylindre $Q=[0,2]\times
\mathbb{S}^{1}.$ Comme $w_{|t=0}=0$ et $\int w(x)dx=0$ alors, soit $w$ est
identiquement nulle, soit $w$ atteint son maximum $M>0$ en $(t_{1},x_{1})\in Q$
avec $t_{1}>0.$ Si $w$ est identiquement nulle alors $v_{x}(t,x)\equiv0,$ et
$v(t,x)=0$ sur $Q.$ D'où (\ref{a}) et (\ref{b}). Maintenant, supposons que $w$
atteint son maximum $M>0$ en $(t_{1},x_{1})\in Q$ et notons%
\begin{equation}
K:=\max(1,\max_{0\leq t\leq2}|\xi(t,\cdot)|_{\mathcal{C}^{3}(\mathbb{S}^{1}%
)}).\label{KmaxC3}%
\end{equation}
Alors%
\[
K\leq\max(1,C||\xi(t,\cdot)||_{X_{4}^{T}})<\infty,\quad C>0
\]
Démontrons que
\begin{equation}
M\leq9K.\label{claim M<=9K}%
\end{equation}
Par un raisonnement par l'absurde, supposons que $M>9K.$ Comme $\xi,%
\operatorname{u}%
\in X_{4}^{T}$, on trouve de l'équation (\ref{equaw=tv_x}) que $w_{t}%
,w_{x},w_{xx}\in\mathcal{C}(Q)$ \cite{Evans}. Alors les conditions
d'optimalité impliquent que
\begin{equation}
w_{t}\geq0,\text{ }w_{x}=0\text{ et }w_{xx}\leq0\quad\text{au point }%
(t_{1},x_{1}).\label{op}%
\end{equation}
Par (\ref{op}) et (\ref{equaw=tv_x}), on a au point $(t_{1},x_{1}),$ la
relation%
\[
-v_{x}+t%
\operatorname{u}%
_{x}^{2}+t%
\operatorname{u}%
\xi_{xx}\leq\nu t\xi_{xxx}.
\]
On multiplie cette dernière inégalité par $t$ et en utilisant le fait que
$w(t_{1},x_{1})=M$ et $t^{2}%
\operatorname{u}%
_{x}^{2}=(w+t\xi_{x})^{2}$, nous avons au point $(t_{1},x_{1})$,
\begin{equation}
-M+(M+t\xi_{x})^{2}+t^{2}%
\operatorname{u}%
\xi_{xx}\leq\nu t^{2}\xi_{xxx}.\label{Op}%
\end{equation}
Comme $t\leq2$, par (\ref{KmaxC3}) on a
\begin{equation}
|t\partial_{x^{j}}^{j}\xi|\leq2K,\quad j=0,\ldots,3.\label{|tksi_(j)|<2K}%
\end{equation}
De plus,%
\begin{equation}
(M+t\xi_{x})^{2}\geq(M-2K)^{2},\quad\text{au point }(t_{1},x_{1}),\label{M-2K}%
\end{equation}
car $M>2K.$ De plus, $\int_{\mathbb{S}^{1}}tv(x)dx=0$ et $tv_{x}=w\leq M$,
alors le lemme \ref{lem18} implique que $|tv|\leq M$, ce qui entraîne
\begin{equation}
|t%
\operatorname{u}%
|\leq|t\xi|+M\leq2K+M,\quad\forall(t,x)\in Q,\label{|tu|<2K+M}%
\end{equation}
et
\[
|t^{2}%
\operatorname{u}%
\xi_{xx}|=|(t%
\operatorname{u}%
)(t\xi_{xx})|\leq(2K+M)(2K),\quad\forall(t,x)\in Q.
\]
Comme $\nu t^{2}\xi_{xxx}\leq4K,$ alors en utilisant (\ref{|tu|<2K+M}), on en déduit que%
\[
-M+(M-2K)^{2}\leq2K(2K+M)+4K,
\]
Cette dernière relation implique que $M<6K+2.$ Comme $M>9K$ et $K\geq1$ alors
$6K+2<9K.$ D'où la contradiction, et (\ref{claim M<=9K}) est établi.

Comme $t%
\operatorname{u}%
_{x}=w+t\xi_{x}$ et $M\leq9K,$ alors pour tout $(t,x)\in Q, $ on a $t%
\operatorname{u}%
_{x}\leq11K.$ Donc, $%
\operatorname{u}%
$ satisfait (\ref{a}) si $%
\operatorname{u}%
(0)\in H^{4}$. Encore une fois, le lemme \ref{lem18} implique que
\[
t|%
\operatorname{u}%
|_{\infty},t|%
\operatorname{u}%
_{x}|_{1}\leq22K\leq22(1+||\xi||_{X_{4}^{T}}),
\]
c'est-à-dire, $%
\operatorname{u}%
$ satisfait (\ref{b}) si $%
\operatorname{u}%
(0)\in H^{4}$.

Maintenant, supposons que $%
\operatorname{u}%
(0)\in H^{1}$ est montrons que la solution $%
\operatorname{u}%
\in H^{1}$ de (\ref{B}) vérifie bien (\ref{a}) et (\ref{b}). Choisissons une
suite $\{%
\operatorname{u}%
^{j}(0)\}_{j\in%
\mathbb{N}
}$ $\in H^{4}$ qui converge vers $%
\operatorname{u}%
(0)$ dans $H^{1}$. Pour chaque $j\in%
\mathbb{N}
,$ $%
\operatorname{u}%
^{j}(0)$ définit une solution $%
\operatorname{u}%
^{j}\in H^{4}$ de (\ref{B}) qui satisfait (par ce qui précède) (\ref{a}) et
(\ref{b}). Par conséquent, pour tout $x\in\mathbb{S}^{1}$%
\begin{equation}
|%
\operatorname{u}%
_{x}^{+}(t,x)|\leq|%
\operatorname{u}%
_{x}^{+}(t,x)-%
\operatorname{u}%
_{x}^{j+}(t,x)|+C(1+||\xi||_{X_{4}^{T}}),\quad C>0.\label{ux+H1}%
\end{equation}
Par le théorème \ref{th6} avec $m=1,$
\begin{equation}%
\operatorname{u}%
^{j}(t,\cdot)\rightarrow%
\operatorname{u}%
(t,\cdot)\text{ dans }H^{1}.\label{**}%
\end{equation}
D'ici $\{%
\operatorname{u}%
_{x}^{j}(t,x)\}$ converge (à extraction d'une sous-suite) vers $%
\operatorname{u}%
_{x}(t,x)$ p.p en $x.$ En passant à la limite dans (\ref{ux+H1}) quand
$j\rightarrow+\infty$ nous obtenons que $|%
\operatorname{u}%
_{x}^{+}(t,x)|\leq C(1+||\xi||_{X_{4}^{T}})$ p.p, et (\ref{a}) est vérifiée.

La convergence (\ref{**}) et le fait que (\ref{b}) est vrai pour $%
\operatorname{u}%
^{j}$ implique que (\ref{b}) reste vrai pour $%
\operatorname{u}%
.$
\end{proof}

Par le théorème \ref{th2} et le théorème \ref{th17}, nous avons immédiatement
le :

\begin{corollary}
\label{cor17}Si $B_{4}<\infty$ et $%
\operatorname{u}%
(0)\in H^{1}$, alors pour tout $p\geq1$ et $0<\theta\leq T,$ il existe
$C(p,T,B_{4})>0$ telle que
\begin{align*}
\mathbb{E}[\sup_{\theta\leq t\leq T}|%
\operatorname{u}%
_{x}^{+}(t,\cdot)|_{\infty}^{p}] &  \leq C\theta^{-p},\\
\mathbb{E}[\sup_{\theta\leq t\leq T}(|%
\operatorname{u}%
(t,\cdot)|_{\infty}^{p}+|%
\operatorname{u}%
_{x}(t,\cdot)|_{1}^{p})] &  \leq C\theta^{-p}.
\end{align*}

\end{corollary}

\section{Loi de la solution $%
\operatorname{u}%
$ et les deux semi groupes de Markov}

Soit $(X,\mathcal{B}_{X})$ et $(Y,\mathcal{B}_{Y})$ deux espaces mesurables.
Par exemple, $X$ et $Y$ sont deux sous-ensembles fermés d'espaces de Banach,
munis de leurs tribus Boréliennes \cite{Shir}. Soit $F:X\rightarrow Y$ une
application mesurable. Notons par $\mathcal{P}(X,\mathcal{B}_{X})$ (resp.
$\mathcal{P}(Y,\mathcal{B}_{Y})$) l'espace des mesures de probabilités sur
$(X,\mathcal{B}_{X})$ (resp. $(Y,\mathcal{B}_{Y})$). Rappelons \cite{Shir} que
$F$ définit l'application
\[
F_{\ast}:\mathcal{P}(X,\mathcal{B}_{X})\rightarrow\mathcal{P}(Y,\mathcal{B}%
_{Y}),\quad m\longmapsto F\circ m,
\]
où pour tout $Q\in$ $\mathcal{B}_{Y}:(F\circ m)(Q)=m(F^{-1}(Q)).$
L'application $F_{\ast}$ est linéaire dans le sens où:%
\begin{equation}
F_{\ast}(am_{1}+(1-a)m_{2})=aF_{\ast}(m_{1})+(1-a)F_{\ast}(m_{2}%
),\label{F_*lineaire}%
\end{equation}
pour tout $m_{1},m_{2}\in\mathcal{P}(X,\mathcal{B}_{X})$ et $a\in\lbrack0,1].$

\begin{lemma}
\label{loi}\cite{Shir} Soient $X$ et $Y$ deux espaces de Banach séparables et
$\mu\in\mathcal{P}(X)$. Alors, il existe une variable aléatoire $\varsigma
:\Omega\rightarrow X$ telle que sa loi, notée $\mathcal{D}(\varsigma),$ est égale
à $\mu.$ De plus, si $F:X\rightarrow Y$ est une application mesurable, alors
$F\circ\mu=\mathcal{D}(F\circ\varsigma).$
\end{lemma}

Soit $m\geq0$. On munit l'espace de Banach séparable $X_{m}^{T}$ de sa tribu
Borélienne $\mathcal{B}:\mathcal{=B}_{X_{m}^{T}}.$ On note par $\mathcal{P}%
(X_{m}^{T})$ l'espace des mesures de probabilité Boréliennes sur $(X_{m}%
^{T},\mathcal{B}),$ et par $\mathcal{C}_{b}(X_{m}^{T})$ l'espace des fonctions
continues et bornées sur $X_{m}^{T}$ muni de la norme de convergence uniforme
\cite{Dud}. On rappelle (théorème \ref{th2}) que si $B_{m}<\infty,$ $m\geq1$,
alors $\xi^{\omega}$ définit une variable aléatoire $\xi:\Omega\rightarrow
X_{m}^{T}.$ Sa loi, notée $\mathcal{D}(\xi^{\omega}),$ est la mesure $\pi
=\xi\circ\mathbb{P}$, c'est-à-dire, la mesure telle que pour tout $Q\in\mathcal{B}%
(X_{m}^{T}):\pi(Q)=\mathbb{P}\{\omega:\xi^{\omega}\in Q\}.$ Alors, pour tout
$f\in\mathcal{C}_{b}(X_{m}^{T})$, on a
\[
\langle f,\pi\rangle:=\int_{X_{m}^{T}}f(\xi)\pi(d\xi)=\int_{\Omega}%
f(\xi^{\omega})\mathbb{P}(d\omega)=\mathbb{E}[f(\xi^{\omega})]\text{,}%
\]
(voir \cite{Shir}). Soit $f\in(X_{m}^{T})^{\ast},$ c'est-à-dire, $f$ est une forme
linéaire continue sur $X_{m}^{T}$.$\ $Comme les processus $\{\beta_{s}\}$ sont
indépendants, alors
\[
f\circ\pi=\mathcal{D}\left(  f\left(  \sum_{s\in%
\mathbb{Z}
^{\ast}}b_{s}\beta_{s}^{\omega}(t)e_{s}(x)\right)  \right)  =\mathcal{D}%
\left(  \sum_{s\in%
\mathbb{Z}
^{\ast}}b_{s}f(\beta_{s}^{\omega}(t)e_{s}(x))\right)  =\pi_{1}\star\pi
_{2}\star\cdots.
\]
Ici $\pi_{j}=\mathcal{D}\left(  b_{j}f(e_{j}\beta_{j}^{\omega})\right)  $ et
$\star$ est la convolution des mesures. Comme les $\pi_{j}$ sont des mesures
Gaussiennes alors $f\circ\pi$ l'est aussi \cite{Shir}. C'est à dire, $\pi$ est
une mesure Gaussienne dans l'espace de Banach $X_{m}^{T}$ \cite{Bog},
\cite{Kuo}. Comme $\pi_{j}$ dépend seulement de $f,b_{s}$ et de la loi de
$\beta_{s}^{\omega}$ alors
\begin{equation}
\pi\text{ dépend seulement de la suite }\{b_{k},k\in%
\mathbb{Z}
^{\ast}\}.\label{mesure}%
\end{equation}
Notons qu'on peut obtenir (\ref{csqfern}) comme conséquence du célèbre
théorème de \textit{Fernique\ }qui affirme que si $\mu$ est une mesure
Gaussienne dans un espace de Banach $X$, alors il existe un réel positif
$\alpha$ tel que
\[
\int_{X}e^{(\alpha||x||_{X})}\mu(dx)<\infty\text{ (voir \cite{Bog},
\cite{Kuo})}.
\]

Soit $B_{m}<\infty,$ $m\geq1.$ Considérons le problème de Cauchy (\ref{B}),
(\ref{B*}), (\ref{B2}). Soit $m_{0}\in\mathcal{P}(H^{m})$ et $%
\operatorname{u}%
_{0}^{\omega}\in H^{m}$ une variable aléatoire indépendante de $\xi^{\omega}$,
telle que $\mathcal{D}(%
\operatorname{u}%
_{0}^{\omega})=m_{0}$ (lemme \ref{loi}). Alors
\[
\mathcal{P}(H^{m}\times X_{m}^{T})\ni\mathcal{D}(%
\operatorname{u}%
_{0}^{\omega},\xi^{\omega})=m_{0}\times\mathcal{D}(\xi^{\omega}).
\]
On note dans la suite $\mathcal{D}(\xi^{\omega})$ par $\pi$. On
supposera $\pi$ fixé et on étudie la dépendance de la loi de la solution $%
\operatorname{u}%
^{\omega}(t,x)$ du problème de Cauchy par rapport à $m_{0}.$ On désigne par $%
\operatorname{u}%
^{\omega}(t;%
\operatorname{u}%
_{0}^{\omega})$ la solution issue de $%
\operatorname{u}%
_{0}^{\omega}$ à l'instant $t$, et on construit l'application $S_{t}^{\ast}$ définie par :%
\begin{equation}
S_{t}^{\ast}:\mathcal{P}(H^{m})\rightarrow\mathcal{P}(H^{m}),\quad
m_{0}\mapsto\mathcal{M}_{t}\circ(m_{0}\times\pi):=\mathcal{D}(%
\operatorname{u}%
^{\omega}(t;%
\operatorname{u}%
_{0}^{\omega})),\nonumber
\end{equation}
(cf. théorème \ref{th6} et (\ref{Mt})). L'application $S_{t}^{\ast}$ est
linéaire (cf. (\ref{F_*lineaire})).

\begin{lemma}
\label{00}Soient $\widetilde{%
\operatorname{u}%
}^{\omega}$ une solution de (\ref{B}) avec $\widetilde{\eta}^{\omega}%
=\partial_{t}\widetilde{\xi}^{\omega},$ $\widetilde{\xi}^{\omega
}(t,x)=\underset{s\in%
\mathbb{Z}
^{\ast}}{\sum}b_{s}\widetilde{\beta}_{s}^{\omega}(t)e_{s}(x),$ et
$\{\widetilde{\beta}_{s}^{\omega}(t),$ $s\in%
\mathbb{Z}
^{\ast}\}$ sont des processus de Wiener standard indépendants. Soit
$\widetilde{%
\operatorname{u}%
}^{\omega}(0)=\widetilde{%
\operatorname{u}%
}_{0}^{\omega}\in H^{m} $ une variable aléatoire indépendante de
$\widetilde{\xi}^{\omega}$ telle que $\mathcal{D}(\widetilde{%
\operatorname{u}%
}_{0}^{\omega})=m_{0}. $ Alors $\mathcal{D}(\widetilde{%
\operatorname{u}%
}(t))=S_{t}^{\ast}(m_{0}), $ pour tout $t\in(0,T).$
\end{lemma}

\begin{proof}
Par le théorème \ref{th6}, nous avons $\widetilde{%
\operatorname{u}%
}^{\omega}(t)=\mathcal{M}_{t}(\widetilde{%
\operatorname{u}%
}_{0}^{\omega},\widetilde{\xi}^{\omega})$, et par le lemme \ref{loi}, nous
obtenons que $\mathcal{D}(\widetilde{%
\operatorname{u}%
}(t))=\mathcal{M}_{t}\circ\mathcal{D}(\widetilde{%
\operatorname{u}%
}_{0}^{\omega},\widetilde{\xi}^{\omega}).$ Étant donné (\ref{mesure}) et comme
$\widetilde{%
\operatorname{u}%
}_{0}^{\omega}$ est indépendante de $\widetilde{\xi}^{\omega}$, alors%
\[
\mathcal{D}(\widetilde{%
\operatorname{u}%
}_{0}^{\omega},\widetilde{\xi}^{\omega})=\mathcal{D}(\widetilde{%
\operatorname{u}%
}_{0}^{\omega})\times\mathcal{D}(\widetilde{\xi}^{\omega})=m_{0}\times\pi.
\]
Donc, $\mathcal{D}(\widetilde{%
\operatorname{u}%
}(t))=\mathcal{M}_{t}\circ(m_{0}\times\pi)=S_{t}^{\ast}(m_{0})$.
\end{proof}

Si $%
\operatorname{u}%
_{0}:=v\in H^{m}$ est indépendante de $\omega$, alors $\mathcal{D}(%
\operatorname{u}%
_{0})=\delta_{v}\in\mathcal{P}(H^{m})$ est la mesure de Dirac et $S_{t}^{\ast
}(\delta_{v})=\mathcal{D}(%
\operatorname{u}%
(t;v)).$ On considère l'application
\[
\Sigma:%
\mathbb{R}
^{+}\times H^{m}\rightarrow\mathcal{P}(H^{m}),\quad(t,v)\mapsto S_{t}^{\ast
}(\delta_{v})=\mathcal{D}(%
\operatorname{u}%
(t;v))=\Sigma_{t}(v).
\]
Elle est appelée : \textit{la fonction de transition de Markov.}

\begin{theorem}
\label{th7}$S_{0}^{\ast}=Id$ et pour tout $t_{1},t_{2}\geq0:$%
\begin{equation}
S_{t_{1}}^{\ast}\circ S_{t_{2}}^{\ast}=S_{t_{1}+t_{2}}^{\ast}%
.\label{St*semigroupe}%
\end{equation}

\end{theorem}

On dit que $\{S_{t}^{\ast}\}_{t\geq0}$ est \textit{le semi groupe de Markov dans
}$\mathcal{P}(H^{m})$ de l'équation (\ref{B}).

\begin{proof}
La première affirmation est évidente. Pour démontrer la dernière, on fixe
$t_{1},t_{2}\geq0$ et on pose $\mu_{i}:=\mathcal{D}(%
\operatorname{u}%
^{\omega}(t_{i};%
\operatorname{u}%
_{0}^{\omega}))$ pour $i=1,2.$ Notons $v(t,x)=u(t_{1}+t;%
\operatorname{u}%
_{0}^{\omega}).$ Alors $v$ vérifie le problème de Cauchy
\[
\left\{
\begin{array}
[c]{c}%
v_{t}+vv_{x}-\nu v_{xx}=\widetilde{\eta}^{\omega},\\
v(0,x)=v_{0},
\end{array}
\right.
\]
où $\widetilde{\eta}^{\omega}(t,x)=\eta^{\omega}(t_{1}+t,x)$ et $v_{0}=%
\operatorname{u}%
(t_{1};%
\operatorname{u}%
_{0}^{\omega}).$ C'est à dire,
\[
\widetilde{\eta}^{\omega}(t,x)=\partial_{t}(\xi^{\omega}(t_{1}+t,x))=\partial
_{t}(\xi^{\omega}(t_{1}+t,x)-\xi^{\omega}(t_{1},x)):=\partial_{t}%
(\widetilde{\xi}^{\omega}(t,x)).
\]
et
\[
\widetilde{\xi}^{\omega}(t,x)=\sum_{s\in%
\mathbb{Z}
^{\ast}}b_{s}(\mathbb{\beta}_{s}^{\omega}(t+t_{1})-\mathbb{\beta}_{s}^{\omega
}(t_{1}))e_{s}(x).
\]
Comme $\{\mathbb{\beta}_{s}^{\omega}\}$ sont des processus de Wiener standards
indépendants alors $\{\mathbb{\beta}_{s}^{\omega}(t+t_{1})-\mathbb{\beta}%
_{s}^{\omega}(t_{1})\}$ l'est aussi (appendice A). Ainsi, nous avons que
$\mathcal{D}(\widetilde{\xi}^{\omega})=\mathcal{D}(\xi^{\omega})$. La variable
aléatoire $v_{0}^{\omega}=%
\operatorname{u}%
(t_{1},%
\operatorname{u}%
_{0}^{\omega})$ ne dépend que de $%
\operatorname{u}%
_{0}^{\omega}$ et $(\xi^{\omega}(t),0\leq t\leq t_{1}).$ Par conséquent, elle est
indépendante des incréments $(\xi^{\omega}(t_{1}+t)-\xi^{\omega}(t_{1}%
),t\geq0)$ (qui sont indépendants de $%
\operatorname{u}%
_{0}^{\omega}$ et des $(\xi^{\omega}(t),0\leq t\leq t_{1})$). Donc, par le
lemme \ref{00}, on a que
\begin{equation}
\mathcal{D}(v(t_{2}))=S_{t_{2}}^{\ast}(\mu_{1})=S_{t_{2}}^{\ast}\circ
S_{t_{1}}^{\ast}(m_{0}).\nonumber
\end{equation}
Comme $\mathcal{D}(v(t_{2}))=\mathcal{D}(%
\operatorname{u}%
(t_{1}+t_{2});%
\operatorname{u}%
_{0})=S_{t_{1}+t_{2}}^{\ast}(\mu_{1}),$ alors nous obtenons
(\ref{St*semigroupe}).
\end{proof}

Notons par $L_{b}(H^{m})$ l'espace vectoriel des fonctions bornées sur $H^{m}
$, et pour tout $t\geq0$, considérons l'application $S_{t}:\mathcal{C}%
_{b}(H^{m})\rightarrow L_{b}(H^{m})$ définie par%
\begin{equation}
S_{t}f\left(  v\right)  :=\mathbb{E[}f(%
\operatorname{u}%
(t;v)],\text{\quad}f\in\mathcal{C}_{b}(H^{m}),\text{ }v\in H^{m}.\label{Stdef}%
\end{equation}
Par définition de $S_{t}$, on a
\begin{equation}
S_{t}f\left(  v\right)  =\langle f,\mathcal{D}(%
\operatorname{u}%
(t;v))\rangle=\langle f,\mathcal{M}_{t}\circ(\delta_{v}\times\pi
)\rangle=\langle f,\Sigma_{t}(v))\rangle.\label{Stf(v)1}%
\end{equation}

\begin{theorem}
\label{th8}Soit $m\geq1$ et $t\geq0.$ Alors

\begin{enumerate}
\item $S_{t}(\mathcal{C}_{b}(H^{m}))\subset\mathcal{C}_{b}(H^{m}).$

\item L'application $S_{t}$ est linéaire et positive.

\item L'application $S_{t}:\mathcal{C}_{b}(H^{m})\rightarrow\mathcal{C}%
_{b}(H^{m})$ est continue et sa norme est égale à 1.

\item Pour tout $f\in\mathcal{C}_{b}(H^{m})$ et $\mu$ dans $\mathcal{P}%
(H^{m})$ nous avons que%
\begin{equation}
\langle S_{t}f,\mu\rangle=\langle f,S_{t}^{\ast}(\mu)\rangle.\label{dual}%
\end{equation}

\end{enumerate}
\end{theorem}

\begin{proof}
1) Si $f\in\mathcal{C}_{b}(H^{m})$ alors par (\ref{Stdef}) :%
\[
S_{t}f(v)=\mathbb{E}[f(%
\operatorname{u}%
(t;v)]=\mathbb{E}[f\circ\mathcal{M}_{t}(v,\xi^{\omega})].
\]
Si $v_{n}\rightarrow v$ dans $H^{m}$ alors pour chaque $\omega,$
$f\circ\mathcal{M}_{t}(v_{n},\xi^{\omega})\rightarrow f\circ\mathcal{M}%
_{t}(v,\xi^{\omega})$ dans $H^{m}$ (car $f\circ\mathcal{M}_{t}$ est
continue)$.$ Alors, par le théorème de convergence dominée de Lebesgue :
$S_{t}f(v_{n})\rightarrow S_{t}f(v)$, donc $S_{t}f\in\mathcal{C}_{b}(H^{m})$.

2) La linéarité et la positivité de $S_{t}$ sont évidentes.

3) La relation $S_{t}\boldsymbol{1}=\boldsymbol{1}$ ou $\boldsymbol{1}(u)=1$
pour tout $u$ est triviale. Soit $f$ dans $\mathcal{C}_{b}(H^{m})$ et notons
par $d=||f||.$ Alors%
\[
-d\boldsymbol{1}\leq f\leq d\boldsymbol{1.}%
\]
En utilisant la positivité de $S_{t}$ et cette relation nous avons que
$-d\boldsymbol{1}\leq S_{t}f\leq d\boldsymbol{1}$, d'ici $||S_{t}f||\leq||f||$. Donc $|S_{t}|_{\mathcal{C}_{b},\mathcal{C}_{b}}\leq1,$ et $S_{t}$ est
continue. Finalement, comme $S_{t}\boldsymbol{1}=\boldsymbol{1}$ alors
$|S_{t}|_{\mathcal{C}_{b},\mathcal{C}_{b}}=1$.

3) Nous avons
\[
\langle f,S_{t}^{\ast}(\mu)\rangle=\langle f,\mathcal{M}_{t}\circ(\mu\times
\pi)\rangle=\langle f\circ\mathcal{M}_{t},\mu\times\pi\rangle.
\]
Puis par le théorème de Fubini :%
\[
\langle f\circ\mathcal{M}_{t},\mu\times\pi\rangle=\int_{H^{m}}\underset
{(S_{t}f)v}{\underbrace{\int_{H^{m}}(f\circ\mathcal{M}_{t})(v,\xi)\pi(d\xi)}%
}\mu(dv)=\langle S_{t}f,\mu\rangle.
\]

\end{proof}

On démontre dans ce qui suit que la famille d'opérateurs $\{S_{t},t\geq0\}$
forment un semi groupe.

\begin{theorem}
\label{th9}Nous avons $S_{0}=Id$ et pour tout $t_{1},t_{2}\geq0$%
\begin{equation}
S_{t_{1}}\circ S_{t_{2}}=S_{t_{1}+t_{2}}.\label{Stsemigroupe}%
\end{equation}

\end{theorem}

On dit que $\{S_{t}\}_{t\geq0}$ est \textit{le semi groupe de Markov dans
}$\mathcal{C}_{b}(H^{m})$ de l'équation (\ref{B}).

\begin{proof}
Soit $t_{1},t_{2}>0$, $f\in\mathcal{C}_{b}(H^{m})$ et $v\in H^{m}.$ Par le
théorème \ref{th8}, on a $(S_{t_{1}+t_{2}}f)v=\langle f,S_{t_{1}+t_{2}}^{\ast
}(\delta_{v})\rangle$ et par le théorème \ref{th7} $\langle f,S_{t_{1}+t_{2}%
}^{\ast}(\delta_{v})\rangle=\langle f,S_{t_{1}}^{\ast}\circ S_{t_{2}}^{\ast
}(\delta_{v})\rangle$. Encore une fois par le théorème \ref{th8} $\langle
f,S_{t_{1}}^{\ast}\circ S_{t_{2}}^{\ast}(\delta_{v})\rangle=\langle S_{t_{1}%
}f,S_{t_{2}}^{\ast}(\delta_{v})\rangle=\langle S_{t_{2}}\circ S_{t_{1}%
}f,\delta_{v}\rangle$. Finalement $(S_{t_{1}+t_{2}}f)v=(S_{t_{2}}\circ
S_{t_{1}}f)v.$
\end{proof}

\begin{theorem}
\label{p10}\label{KCH}(Relation de Kolmogorov-Chapman)

Pour tout $t_{1},t_{2}\geq0,$ $m\geq1$ et $v\in H^{m}$, on a
\begin{equation}
\Sigma_{t_{1}+t_{2}}(v)=\int\left(  \Sigma_{t_{2}}(u)\right)  \Sigma_{t_{1}%
}(v)\left(  du\right)  .\label{m0}%
\end{equation}
Ici, l'intégrale de droite est l'intégrale de la fonction $u\longmapsto
\Sigma_{t_{2}}(u)\in\mathcal{P}(H^{m})$ par rapport à la mesure $\Sigma
_{t_{1}}(v)$ \cite{Bourba}. Alors, pour tout $f\in\mathcal{C}_{b}(H^{m}):$%
\begin{equation}
S_{t_{1}+t_{2}}f(v)=\langle\Sigma_{t_{1}+t_{2}}(v),f\rangle=\int\langle
\Sigma_{t_{2}}(u),f\rangle\Sigma_{t_{1}}(v)(du).\label{m1}%
\end{equation}

\end{theorem}

\begin{proof}
L'intégrale à droite de (\ref{m1}) est égale à $\int S_{t_{2}}f(u)\Sigma
_{t_{1}}(v)(du)=\left(  S_{t_{1}}\circ S_{t_{2}}f\right)  (v)$. Alors,
(\ref{m1}) est une conséquence de (\ref{Stsemigroupe}). Comme les termes à
gauche et à droite de (\ref{m1}) sont les intégrales d'une fonction arbitraire $f\in {\mathcal{C}}_{b}(H^{m})$ par rapport aux deux
mesures définies dans (\ref{m0}), alors (\ref{m0}) est une conséquence de
(\ref{m1}).%

\end{proof}

Le théorème \ref{KCH} peut être généralisé :

\begin{proposition}
\label{KCHgen}Pour tout $t_{1},t_{2}\geq0,$ $m\geq1$ et $v\in H^{m}$, si
$f\in\mathcal{C}_{b}(H^{m})$ et $\Phi$ est une fonction bornée et mesurable
sur $X_{m}^{t_{1}},$ alors
\begin{equation}
\mathbb{E}[\Phi(%
\operatorname{u}%
(\tau;v)_{|\tau\in\lbrack0,t_{1}]})f(%
\operatorname{u}%
(t_{1}+t_{2};v))]=\mathbb{E}[\Phi(%
\operatorname{u}%
(\tau;v)_{|\tau\in\lbrack0,t_{1}]})\langle\Sigma_{t_{2}}(%
\operatorname{u}%
(t_{1};v)),f\rangle].\label{m2}%
\end{equation}

\end{proposition}

Notons que (\ref{m2}) est égale à (\ref{m1}) si $\Phi\equiv1.$

\begin{proof}
Nous raisonnons comme dans la démonstration du théorème \ref{th7}. Considérons
des espaces de probabilité $\left(  \Omega_{i},\mathcal{F}_{i},\mathbb{P}%
_{i}\right)  ,$ $i=1,2,$ et $(\Omega,\mathcal{F},\mathbb{P})=(\Omega_{1}%
\times\Omega_{2},,\mathcal{F}_{1}\times\mathcal{F}_{2},\mathbb{P}_{1}%
\times\mathbb{P}_{2}).$ Soit $\{\beta_{s}^{\omega_{i}},$ $s\in%
\mathbb{Z}
^{\ast}\},$ $i=1,2,$ des processus de Wiener standard indépendants définies
sur $\left(  \Omega_{i},\mathcal{F}_{i},\mathbb{P}_{i}\right)  ,$ $i=1,2.$
Pour $t_{1},t_{2}\geq0$ et $\omega=(\omega_{1},\omega_{2})\in\Omega,$ on pose
:%
\[
\beta_{s}^{\omega}(t)=\left\{
\begin{array}
[c]{c}%
\beta_{s}^{\omega_{1}}(t),\quad0\leq t\leq t_{1}\\
\beta_{s}^{\omega_{1}}(t_{1})+\beta_{s}^{\omega_{2}}(t-t_{1}),\quad t_{1}\leq
t\leq t_{2},
\end{array}
\right.
\]
qui sont des processus de Wiener standard indépendants. On choisit cette
forme de $\{\beta_{s}^{\omega},$ $s\in%
\mathbb{Z}
^{\ast}\}$ pour la force $\xi$ (voir (\ref{B*})). Dans ce cas, pour $0\leq
t\leq t_{1,}$ la solution $%
\operatorname{u}%
^{\omega}(t;v)=%
\operatorname{u}%
^{\omega_{1}}(t;v)$ ne dépend pas de $\omega_{2},$ et pour $t_{1}\leq t\leq
t_{2}$, l'équation pour $%
\operatorname{u}%
^{\omega}(t;v)$ dépend seulement de $\omega_{2}.$ Donc,
\[%
\operatorname{u}%
^{\omega}(t_{1}+t;v)=%
\operatorname{u}%
^{\omega_{2}}(t;%
\operatorname{u}%
^{\omega_{1}}(t_{1};v)),
\]
et l'intégrale à gauche de (\ref{m2}) s'écrit :
\begin{align*}
& \int_{\Omega}\Phi(%
\operatorname{u}%
^{\omega}(\tau;v)_{|\tau\in\lbrack0,t_{1}]})f(%
\operatorname{u}%
^{\omega}(t_{1}+t_{2};v)\mathbb{P}(d\omega)\\
& =\int_{\Omega_{1}}\int_{\Omega_{2}}\Phi(%
\operatorname{u}%
^{\omega_{1}}(\tau;v)_{|\tau\in\lbrack0,t_{1}]})f(%
\operatorname{u}%
^{\omega_{1},\omega_{2}}(t_{1}+t_{2};v)\mathbb{P}_{1}(d\omega_{1}%
)\mathbb{P}_{2}(d\omega_{2})\\
& =\int_{\Omega_{1}}\left(  \int_{\Omega_{2}}f(%
\operatorname{u}%
^{\omega_{2}}(t_{2};%
\operatorname{u}%
^{\omega_{1}}(t_{1};v))\mathbb{P}_{2}(d\omega_{2})\right)  \Phi(%
\operatorname{u}%
^{\omega_{1}}(\tau;v)_{|\tau\in\lbrack0,t_{1}]})\mathbb{P}_{1}(d\omega_{1})\\
& =\int_{\Omega}\langle\Sigma_{t_{2}}(%
\operatorname{u}%
(t_{1};v)),f\rangle\Phi(%
\operatorname{u}%
^{\omega}(\tau;v)_{|\tau\in\lbrack0,t_{1}]})\mathbb{P}(d\omega).
\end{align*}
Et (\ref{m2}) est vérifié.
\end{proof}

Le théorème \ref{th7} et (\ref{m1}) impliquent que pour $0<t_{1}<t_{1}+t_{2}$ et
pour chaque donnée initiale $%
\operatorname{u}%
_{0}^{\omega}$ indépendante de la force $\eta^{\omega}$ de la forme
(\ref{B*})$,$ la loi de la solution $%
\operatorname{u}%
(t_{1}+t_{2};%
\operatorname{u}%
_{0})$ est une fonction qui dépend seulement de la loi de $%
\operatorname{u}%
(t_{1};%
\operatorname{u}%
_{0}).$ On appelle cela : \textit{la propriété de Markov. }C'est la raison
pour laquelle il est plus facile d'étudier l'équation de Burgers stochastique
avec une force $\eta^{\omega}$ qui est un bruit blanc plutôt qu'avec une autre
force aléatoire.

\section{Convergence faible de mesures}

On rappelle ici quelques résultats sur la notion de convergence faible de mesures \cite{Billi}, \cite{Dud}, \cite{Shir}. Soit $X$ un espace
de Banach séparable, $O_{X}\subset X$ un fermé (par exemple, $O_{X}=X$). On
note
\[
L(O_{X}):=\{f\in\mathcal{C}_{b}(O_{X}):Lip(f)<\infty\},
\]
où $Lip(f)$ désigne la constante de Lipschitz de $f.$ L'espace $L(O_{X})$ muni
de la norme
\[
||f||_{L}=||f||_{L(X)}:=||f||+Lip(f)
\]
est un espace de Banach non séparable \cite{Dud},\cite{KuShir}.

\begin{definition}
\label{def11}On dit qu'une suite $\{\mu_{n}\}\subset\mathcal{P}(O_{X})$
converge faiblement vers $\mu$ $\in\mathcal{P}(O_{X})$, et on note $\mu
_{n}\rightharpoonup\mu$, si et seulement si%
\[
\forall f\in\mathcal{C}_{b}(O_{X}):\langle\mu_{n},f\rangle\rightarrow
\langle\mu,f\rangle.
\]

\end{definition}

La convergence $\mu_{n}\rightharpoonup\mu$ n'implique pas la convergence de
$\mu_{n}(Q)$ vers $\mu(Q),$ $Q\in\mathcal{B}(O_{X}).$ Pourtant%
\begin{equation}
\mu_{n}(Q)\geq\lim_{n\rightarrow+\infty}\sup\mu_{n}(Q),\text{ si }Q\text{ est
fermé et }\mu_{n}\rightharpoonup\mu\text{ (voir \cite{Billi}, \cite{Dud},
\cite{Shir}).}\label{ferme}%
\end{equation}

\begin{definition}
\label{def12}Si $\mu$ et $\nu$ sont des éléments de $\mathcal{P}(O_{X}),$
alors la distance Lipschitz-duale (ou Lip-duale) entre $\mu$ et $\nu$ est
définie par
\[
||\mu-\nu||_{L(X)}^{\ast}:=\sup_{||f||_{L(X)}\leq1}(\langle f,\mu
\rangle-\langle f,\nu\rangle).
\]

\end{definition}

Le théorème suivant dû à Kantorovitch caractérise la convergence faible de
mesures par la distance Lipschitz-duale \cite{Dud}.

\begin{theorem}
\label{th13}\label{kantrovich}Soit $X$ un espace de Banach séparable,
$O_{X}\subset X$ un fermé. Alors

\begin{enumerate}
\item $(\mathcal{P}(O_{X}),||\cdot||_{L(X)}^{\ast})$ est un espace métrique complet.

\item Si $\{\mu_{n}\}$ $\subset\mathcal{P}(O_{X})$ et $\mu$ $\in
\mathcal{P}(O_{X})$ alors
\[
\mu_{n}\rightharpoonup\mu\Longleftrightarrow||\mu_{n}-\mu||_{L(X)}^{\ast
}\rightarrow0.
\]

\end{enumerate}
\end{theorem}

\begin{theorem}
\label{th14}\label{ST*continuefaible}Pour tout $t\geq0$ et $m\geq1$,
l'application $S_{t}^{\ast}:\mathcal{P}(H^{m})\rightarrow\mathcal{P}(H^{m}) $
est faiblement continue, c'est-à-dire, si $\mu_{n}\rightharpoonup\mu$ alors
$S_{t}^{\ast}\mu_{n}\rightharpoonup S_{t}^{\ast}\mu.$
\end{theorem}

\begin{proof}
Soit $\mu_{n}\rightharpoonup\mu$ et $f\in\mathcal{C}_{b}(H^{m}).$ On a par
(\ref{dual}) : $\langle S_{t}^{\ast}\mu_{n},f\rangle=\langle\mu_{n}%
,S_{t}f\rangle$ et $\langle S_{t}^{\ast}\mu,f\rangle=\langle\mu,S_{t}%
f\rangle.$ Comme $S_{t}f\in\mathcal{C}_{b}(H^{m}),$ on a $\langle\mu_{n}%
,S_{t}f\rangle\rightarrow\langle\mu,S_{t}f\rangle.$ Donc, $S_{t}^{\ast}\mu
_{n}\rightharpoonup S_{t}^{\ast}\mu.$
\end{proof}

\begin{definition}
\label{def15}Soit $M$ un ensemble de mesures de $\mathcal{P}(O_{X}).$

\begin{enumerate}
\item $M$ est dit tendu, si pour tout $\varepsilon>0$, il existe un compact
$K_{\varepsilon}\subset O_{X}$ tel que%
\[
\mu(K_{\varepsilon})\geq1-\epsilon,\text{\quad}\forall\mu\in M.
\]

\item $M$ est dit faiblement relativement compact, si toute suite $\{\mu
_{n}\}\subset M$ admet une sous-suite faiblement convergente dans
$\mathcal{P}(O_{X}).$
\end{enumerate}
\end{definition}

Le théorème suivant dû à Prokhorov caractérise les ensembles de mesures
faiblement relativement compact \cite{Billi}, \cite{Dud}, \cite{Shir}.

\begin{theorem}
\label{th16}\label{prochorov}Un ensemble de mesures $M\subset\mathcal{P}%
(O_{X})$ est tendu si et seulement si $M$ est faiblement relativement compact.
\end{theorem}

\section{Bornes supérieures pour les normes Sobolev de $%
\operatorname{u}%
$}

Dans la suite nous supposerons toujours que $B_{4}<\infty.$

Le but de cette section est d'estimer l'espérance mathématique des normes
$\mathbb{E}[||%
\operatorname{u}%
(t)||_{m}^{2}],$ $m\geq1$, de la solution $%
\operatorname{u}%
(t)$ de (\ref{B}) uniformément en $\nu\in(0,1]$ et en $%
\operatorname{u}%
_{0}\in H^{1}.$

On se donne $%
\operatorname{u}%
$ $\in X_{m}^{T}$ solution de (\ref{B}), et on rappelle brièvement la formule
d'Itô \cite{Gall}, \cite{KaS}, \cite{Oks}. Soit $f$ une fonction sur $H^{m}$
de classe $\mathcal{C}^{2}$. Alors la formule d'Itô implique que\footnote{Dans
le sens où la relation (\ref{ito}) est l'espérance de la formule d'Itô.}
\begin{equation}
\frac{d}{dt}\mathbb{E}[f(%
\operatorname{u}%
(t))]=\mathbb{E}\left[  df(%
\operatorname{u}%
(t))L(%
\operatorname{u}%
(t))+\frac{1}{2}\sum_{s\in%
\mathbb{Z}
^{\ast}}b_{s}^{2}d^{2}f(%
\operatorname{u}%
(t))(e_{s},e_{s})\right]  ,\label{ito}%
\end{equation}
où $L(%
\operatorname{u}%
)=\nu%
\operatorname{u}%
_{xx}-%
\operatorname{u}%
\operatorname{u}%
_{x}.$ La relation (\ref{ito}) est une égalité formelle. Elle devient une
affirmation rigoureuse si on suppose des restrictions sur $f$ \cite[Appendix
A7]{KuShir}. Nous discuterons cela plus tard.

Soit $x\in%
\mathbb{R}
^{+}$, on note par $\lceil x\rceil$ la partie entière par excès de $x:$%
\[
\lceil x\rceil=\min\{l\in%
\mathbb{N}
:l\geq x\}.
\]

\begin{theorem}
\label{th19}Soit $T>\theta>0$, $m\geq1$ et $\nu\in(0,1]$. Si $B_{\lceil
m\rceil}<\infty$, alors il existe $C_{m}(\theta,B_{\lceil m\rceil},B_{4})>0$
telle que pour tout $%
\operatorname{u}%
_{0}\in H^{1}$, la solution $%
\operatorname{u}%
$ de (\ref{B}), (\ref{B*}), (\ref{B2}) satisfait%
\begin{equation}
\mathbb{E}[||%
\operatorname{u}%
(t)||_{m}^{2}]\leq C_{m}\nu^{-(2m-1)},\quad\forall\theta\leq t\leq
T.\label{esperence<visco}%
\end{equation}

\end{theorem}

\begin{proof}
Par $C_{mj},$ $c_{m}^{^{\prime}},$ $C_{mj}^{\prime},$ $\ldots$ etc, nous
notons des constantes positives qui ne dépendent que de $m,$ $\theta,$ $B_{4}$
et $B_{\lceil m\rceil}$. On supposera d'abord que $m\in%
\mathbb{N}
^{\ast}.$

i) En premier lieu, nous présentons une preuve basée sur une application
formelle de la formule (\ref{ito}) et nous expliquerons plus tard comment
convertir cet argument en une preuve rigoureuse. Supposons d'abord que $%
\operatorname{u}%
_{0}\in H^{m},$ alors la solution $%
\operatorname{u}%
=%
\operatorname{u}%
^{\omega}\in X_{m}^{T}$ (cf. théorème \ref{burgersok}) pour chaque $\omega.$

On applique (\ref{ito}) à la fonction $f(%
\operatorname{u}%
)=||%
\operatorname{u}%
||_{m}^{2}=\langle(-\Delta)^{m}%
\operatorname{u}%
,%
\operatorname{u}%
\rangle.$ On a $df(%
\operatorname{u}%
)v=2\langle(-\Delta)^{m}u,v\rangle$ et $d^{2}f(%
\operatorname{u}%
)(e_{s},e_{s})=2\langle(-\Delta)^{m}e_{s},e_{s}\rangle=(2\pi s)^{2m}\langle
e_{s},e_{s}\rangle=(2\pi s)^{2m}$, donc (\ref{ito}) s'écrit comme%
\begin{equation}
\frac{d}{dt}\mathbb{E}[||%
\operatorname{u}%
||_{m}^{2}]=-\mathbb{E}[\langle(-\Delta)^{m}%
\operatorname{u}%
,\partial_{x}%
\operatorname{u}%
^{2}\rangle]-2\nu\mathbb{E}[\langle(-\Delta)^{m}%
\operatorname{u}%
,(-\Delta)%
\operatorname{u}%
\rangle]+B_{m}^{^{\prime}},\label{ito2}%
\end{equation}
où $B_{m}^{^{\prime}}=(2\pi)^{2m}B_{m}=\underset{s\in%
\mathbb{Z}
^{\ast}}{\sum}b_{s}^{2}(2\pi s)^{2m}.$

Par le lemme \ref{GNtermeNL} (avec $q=\infty$), et le fait que $|%
\operatorname{u}%
|_{\infty}\leq|%
\operatorname{u}%
_{x}|_{1},$ nous avons
\begin{equation}
|\mathbb{E}[\langle(-\Delta)^{m}%
\operatorname{u}%
,\partial_{x}%
\operatorname{u}%
^{2}\rangle]|\leq C_{m}\mathbb{E[}||%
\operatorname{u}%
||_{m+1}^{\alpha+1}|%
\operatorname{u}%
_{x}|_{1}^{2-\alpha}],\quad\alpha=\frac{2m-1}{2m+1}.\label{iuinfinimm+1}%
\end{equation}
L'inégalité de Hölder\ appliquée au membre à droite avec $p=\frac{2m+1}{2m}$,
et le corollaire \ref{cor17} impliquent que pour tout $\theta^{\prime}%
:=\frac{\theta}{2}\leq t\leq T:$
\[
\mathbb{E[}|%
\operatorname{u}%
_{x}|_{1}^{2-\alpha}||%
\operatorname{u}%
||_{m+1}^{\alpha+1}]\leq\left(  \mathbb{E}[|%
\operatorname{u}%
_{x}|_{1}^{(2-\alpha)q_{m}}]\right)  ^{\frac{1}{q_{m}}}\left(  \mathbb{E}[||%
\operatorname{u}%
||_{m+1}^{2}]\right)  ^{\frac{1}{p}}\leq C_{m}\left(  \mathbb{E}[||%
\operatorname{u}%
||_{m+1}^{2}]\right)  ^{\frac{2m}{2m+1}}.
\]
On pose $X_{j}(t)=\mathbb{E}[||%
\operatorname{u}%
||_{j}^{2}],$ $j\in%
\mathbb{N}
^{\ast}.$ D'après les estimations précédentes, la relation (\ref{ito2}) s'écrit
\begin{equation}
\frac{d}{dt}X_{m}(t)\leq B_{m}^{^{\prime}}-2\nu X_{m+1}(t)+C_{m}%
X_{m+1}(t)^{\frac{2m}{2m+1}},\quad\theta^{\prime}\leq t\leq T.\label{diffHm}%
\end{equation}
Comme précédemment, l'inégalité (\ref{GN}) (avec $\theta(m,\infty,2,m+1)$),
celle de Hölder et le corollaire \ref{cor17} impliquent que :
\[
X_{m}(t)\leq C_{m1}X_{m+1}(t)^{\frac{2m-1}{2m+1}}\left(  \mathbb{E[}|%
\operatorname{u}%
_{x}|_{1}^{a}]\right)  ^{b}\leq C_{m2}X_{m+1}(t)^{\frac{2m-1}{2m+1}},
\]
pour des constantes convenables $a,b>0$, si $\theta^{\prime}\leq t\leq T$.
Alors%
\begin{equation}
X_{m+1}(t)\geq C_{m3}X_{m}(t)^{\frac{2m+1}{2m-1}},\quad\theta^{\prime}\leq
t\leq T.\label{IXY}%
\end{equation}
La relation (\ref{diffHm}) s'écrit pour tout $\theta^{\prime}\leq t\leq T:$%
\begin{equation}
\frac{d}{dt}X_{m}(t)\leq B_{m}^{^{\prime}}-X_{m+1}(t)^{\frac{2m}{2m+1}}\left(
2\nu X_{m+1}(t)^{\frac{1}{2m+1}}-C_{m}\right)  .\label{Xdiff}%
\end{equation}
Fixons $\beta>1$ et supposons que
\begin{equation}
\exists t_{\ast}\in(2\theta^{\prime},T]\text{ tel que }X_{m}(t_{\ast}%
)>\beta\nu^{-(2m-1)}=:Y.\label{H0}%
\end{equation}
Soit $s=t_{\ast}-t,$ $s\in\lbrack0,t_{\ast}].$ Alors, nous avons de
(\ref{Xdiff}) :%
\begin{equation}
\frac{d}{ds}X_{m}(s)\geq-B_{m}^{^{\prime}}+X_{m+1}(s)^{\frac{2m}{2m+1}}\left(
2\nu X_{m+1}(s)^{\frac{1}{2m+1}}-C_{m}\right)  .\label{Xdiff2}%
\end{equation}
Si $X_{m}(s)>Y,$ alors par (\ref{IXY})
\begin{align*}
2\nu X_{m+1}(s)^{\frac{1}{2m+1}}-C_{m}  & \geq2\nu\left(  C_{m3}%
X_{m}(s)^{\frac{2m+1}{2m-1}}\right)  ^{\frac{1}{2m+1}}-C_{m}\\
& \geq2C_{m3}^{\frac{1}{2m+1}}\beta^{\frac{1}{2m-1}}-C_{m}:=K_{m}(\beta).
\end{align*}
Choisissons $\beta_{0}\gg1$ telle que $K_{m}(\beta)>1$ pour tout $\beta
>\beta_{0}.$ Alors, d'après (\ref{IXY}) et (\ref{Xdiff2}), si $X_{m}(s)>Y$ on déduit que\begin{equation}
\frac{d}{dt}X_{m}(s)\geq-B_{m}^{^{\prime}}+X_{m+1}(s)^{\frac{2m}{2m+1}}%
K_{m}(\beta)\geq-B_{m}^{^{\prime}}+C_{m3}^{\frac{2m}{2m+1}}X_{m}(s)^{\frac
{2m}{2m-1}}K_{m}(\beta)>0,\label{Xdiff3}%
\end{equation}
où la dernière inégalité est valide si $\beta_{0}\gg1.$ D'après (\ref{H0}) et (\ref{Xdiff3}), nous obtenons que%

\begin{equation}
\text{la fonction }s\longmapsto X_{m}(s)\text{ est croissante sur }[0,t_{\ast
}]\text{ et est minorée par }Y.\label{5}%
\end{equation}
En effet, supposons que la fonction n'est pas minorée par $Y$ partout, et
trouvons le premier temps $s_{1}$ tel que que $X_{m}(s_{1})=Y.$ Par
(\ref{H0}), $s_{1}>0.$ Donc $\frac{d}{ds}X_{m}(s)_{|s=s_{1}}\leq0,$ et ceci
contredit (\ref{Xdiff3}). Comme $X_{m}>Y,$ alors $s\longmapsto X_{m}(s)$ est
croissante par (\ref{Xdiff3}).

Par (\ref{5}), (\ref{Xdiff2}) et (\ref{Xdiff3}), nous avons pour tout
$s\in\lbrack0,t_{\ast}]:$%
\[
\frac{d}{dt}X_{m}(s)\geq-B_{m}^{^{\prime}}+C_{m3}^{\frac{2m}{2m+1}}%
X_{m}(s)^{\frac{2m}{2m-1}}K_{m}(\beta)\geq\frac{1}{2}C_{m3}^{\frac{2m}{2m+1}%
}X_{m}(s)^{\frac{2m}{2m-1}}K_{m}(\beta),
\]
si $\beta_{0}\gg1.$ Cette dernière inégalité implique que%
\[
\frac{d}{dt}\left(  X_{m}(s)^{\frac{-1}{2m-1}}\right)  \leq-\frac
{C_{m3}^{\frac{2m}{2m+1}}}{2(2m-1)}K_{m}(\beta)=:-C_{m4}K_{m}(\beta).
\]
En intégrant cette dernière relation en temps entre $0$ et $s$ et en utilisant
(\ref{5}), on a
\[
X_{m}(s)^{\frac{-1}{2m-1}}\leq-C_{m4}K_{m}(\beta)s+X_{m}(0)^{\frac{-1}{2m-1}%
}\leq-C_{m4}K_{m}(\beta)s+\beta^{\frac{-1}{2m-1}}\nu.
\]
Comme $\nu\leq1$, nous pouvons trouver un $s^{\prime}\in(0,t_{\ast}]$
tel que $X_{m}(s^{\prime})^{\frac{-1}{2m-1}}=0,$ si $\beta_{0}\gg1$. D'où la
contradiction, car $X_{m}(s^{\prime})^{\frac{-1}{2m-1}}>0$. Donc, (\ref{H0})
est fausse si $\beta$ est suffisamment grand. Ainsi, nous obtenons
(\ref{esperence<visco}) avec $C_{m}=\beta,$ si $%
\operatorname{u}%
_{0}\in H^{m}.$%

ii) Supposons que $%
\operatorname{u}%
_{0}\in H^{1}.$ Alors il existe une suite $\{%
\operatorname{u}%
_{0}^{j}\}_{j\in%
\mathbb{N}
}\subset H^{m}$ qui converge fortement vers $%
\operatorname{u}%
_{0}$ dans $H^{1}.$ Par le théorème \ref{th6}, pour chaque $\omega$, la solution $%
\operatorname{u}%
(t,%
\operatorname{u}%
_{0}^{j})$ converge fortement dans $H^{1}$ vers la solution $%
\operatorname{u}%
(t;%
\operatorname{u}%
_{0})$ pour tout $0\leq t\leq T.$ Par i), la relation (\ref{esperence<visco})
est déjà vérifiée pour les solutions $%
\operatorname{u}%
(t,%
\operatorname{u}%
_{0}^{j}). $ Pour $N\in%
\mathbb{N}
$, notons par $\Pi_{N}$ le \textit{projecteur de Galerkin }tel que
\[
\Pi_{N}(%
\operatorname{u}%
(x))=\sum_{|s|\leq N}u_{s}e_{s}(x),
\]
et considérons la fonction $f_{N}(%
\operatorname{u}%
)=f\left(  \Pi_{N}(%
\operatorname{u}%
(x))\right)  \wedge N,$ $f(%
\operatorname{u}%
)=||%
\operatorname{u}%
||_{m}^{2}.$ Alors :

a) $f_{N}\in\mathcal{C}_{b}(H^{1})$ et $f_{N}\geq0,$

b) $f_{N}(%
\operatorname{u}%
)\underset{N\rightarrow+\infty}{\rightarrow}f(%
\operatorname{u}%
)\leq\infty,$ pour tout $%
\operatorname{u}%
\in H^{1}.$

Par i), a) et le théorème de convergence dominée de Lebesgue, nous avons
\[
\mathbb{E}[f_{N}(%
\operatorname{u}%
(t;%
\operatorname{u}%
_{0})]=\lim_{j\rightarrow+\infty}[f_{N}(%
\operatorname{u}%
(t;%
\operatorname{u}%
_{0}^{j}))]\leq c_{m}^{^{\prime}}\nu^{-(2m-1)},\quad\forall N\in%
\mathbb{N}
.
\]
En utilisant cette relation et le lemme de Fatou, nous obtenons
\[
\mathbb{E}[||%
\operatorname{u}%
(t;%
\operatorname{u}%
_{0})||_{m}^{2}]\leq\lim_{N}\inf\mathbb{E}[f_{N}(%
\operatorname{u}%
(t;%
\operatorname{u}%
_{0})]\leq c_{m}^{\prime}\nu^{-(2m-1)}.
\]
iii) La démonstration du théorème est finie, si $m\in%
\mathbb{N}
^{\ast}.$ Maintenant, on suppose que $m\geq1$ et $m\notin%
\mathbb{N}
^{\ast}.$ Alors, il existe $j\in%
\mathbb{N}
^{\ast}$ et $s\in(0,1)$ tels que $m=j+s$ et $\lceil m\rceil=j+1.$ Par
l'inégalité d'interpolation et celle de Hölder, nous avons
\[
\mathbb{E}\left[  ||%
\operatorname{u}%
||_{m}^{2}\right]  \leq\mathbb{E}\left[  ||%
\operatorname{u}%
||_{j+1}^{2}\right]  ^{s}\mathbb{E}\left[  ||%
\operatorname{u}%
||_{j}^{2}\right]  ^{1-s}.
\]
Comme (\ref{esperence<visco}) est établie pour $m=j$ et $m=j+1,$ alors le
terme de droite de cette inégalité est majoré par $C_{m}\nu^{-\left(
(2\left[  j+1\right]  -1)s+\left[  2j-1\right]  \left[  1-s\right]  \right)
}=C_{m}\nu^{-(2m-1)}.$

iv) Il reste à justifier l'application de la formule de Itô (\ref{ito}).
D'abord, soit $%
\operatorname{u}%
_{0}\in H^{m+2}$ et $B_{m+2}<\infty$. %
Soient $\{%
\operatorname{u}%
^{N},$ $N\geq1\}$ les approximations de Galerkin pour $%
\operatorname{u}%
$. Elles convergent faiblement vers $%
\operatorname{u}%
$ dans l'espace $X_{m+2}^{T}$ et fortement dans $X_{m+1}^{T}$, pour tout
$\omega$ (cf. preuve du théorème \ref{thwhm}). Considérons l'équation
satisfaite par $%
\operatorname{u}%
^{N}:$%
\begin{equation}
\partial_{t}%
\operatorname{u}%
^{N}+\Pi_{N}\left(
\operatorname{u}%
^{N}%
\operatorname{u}%
_{x}^{N}\right)  -\nu%
\operatorname{u}^{N}%
_{xx}=\partial_{t}\Pi_{N}\xi^{\omega}.\label{G}%
\end{equation}
C'est une équation stochastique dans $%
\mathbb{R}
^{2N}$ telle que l'équation libre qui lui est associée possède une fonction de
Lyapunov quadratique $||%
\operatorname{u}%
^{N}||^{2}$. Alors, la formule d'Itô appliquée à (\ref{G}) avec la
fonctionnelle quadratique $f(%
\operatorname{u}%
)=||%
\operatorname{u}%
||_{m}^{2}$ implique la relation (\ref{ito2}) avec $%
\operatorname{u}%
=%
\operatorname{u}%
^{N}$ et $B_{m}^{\prime}$ remplacé par $\underset{|s|\leq N}{\sum}b_{s}%
^{2}(2\pi s)^{2m}$, voir \cite{Khas}.%

Utilisant le lemme de Fatou, nous passons à la limite dans (\ref{G}) et
obtenons (\ref{esperence<visco}) si $%
\operatorname{u}%
_{0}\in H^{m+2}$ et $B_{m+2}<\infty.$

Finalement, si $%
\operatorname{u}%
_{0}\in H^{1}$ et $B_{m}<\infty$, alors nous approchons $%
\operatorname{u}%
_{0}$ par $\{%
\operatorname{u}%
_{0}^{j}\}\subset H^{m+2}$, et pour $j=1,2, ...$, nous définissons la suite $\{{b_{s}^{j}, s\in \mathbb{Z}^{*}}\}$ par la relation : $b_{s}^{j}=b_{s}$ si $|s|\leq j$, sinon $b_{s}^{j}=0$. Il est clair que $B_{m+2}\{(b_{s}^{j}\})<+\infty$ pour chaque $j$. On définit $\xi^{j\omega}$ utilisant la suite $b_{s}^{j}$. Par les résultats de la section 2, il existe une suite $n_{1}, n_{2}, ...$ telle que $\xi^{n^{j}\omega}$ converge p.p vers $\xi^{\omega}$ dans $X^{T}_{m}$ quand $n_{j}$ tend vers $+\infty$. Alors $\operatorname{u}(\cdot,\operatorname{u}_{0}^{n^{j}},\xi^{n^{j}\omega})$ converge p.p vers $\operatorname{u}(\cdot,\operatorname{u}_{0},\xi^{\omega})$ dans l'espace $X^{T}_{m}$ (voir le théorème \ref{th6}). Ainsi, la relation (\ref{esperence<visco}) est justifiée pour les solutions $\operatorname{u}(\cdot,\operatorname{u}_{0}^{n^{j}},\xi^{n^{j}\omega})$. Nous passons à la limite $n_{j}\rightarrow+\infty$ (cf. ii)) pour obtenir le résultat.

\end{proof}

Par le théorème \ref{th19},
\begin{equation}
\text{si }B_{4},B_{\lceil m\rceil}<\infty,\text{ }m\geq1\text{, }%
\operatorname{u}%
_{0}\in H^{1}\text{ et }t>0,\text{ alors }\mathbb{P}(%
\operatorname{u}%
(t,%
\operatorname{u}%
_{0})\in H^{m})=1.\label{aussilissequelaforce}%
\end{equation}

\begin{corollary}
\label{momentksob}Sous les conditions du théorème \ref{th19}, pour tout
$m,k\geq1,$ il existe $C(k,m,\theta)>0$ telle que
\begin{equation}
\mathbb{E}\left[  ||%
\operatorname{u}%
(t)||_{m}^{k}\right]  \leq C\nu^{-\frac{k}{2}(2m-1)},\quad\forall t\geq
\theta.\label{inegmomentksob}%
\end{equation}

\end{corollary}

\begin{proof}
i) Supposons d'abord que $m \in \mathbb{N}^{*}$. Pour $\mathbb{N}^{*} \ni s>m$ et par le lemme \ref{GNtermeL}, nous avons
\[
\mathbb{E}\left[  ||%
\operatorname{u}%
(t)||_{m}^{k}\right]  \leq C\mathbb{E}\left[  ||%
\operatorname{u}%
(t)||_{s}^{k\lambda_{m}(s)}|%
\operatorname{u}%
(t)|_{\infty}^{k(1-\lambda_{m}(s))}\right], \quad \lambda_{m}(s)=\frac{2m-1}{2s-1},
\]
où $C=C(m,k,s)>0$. Choisissons $s$ suffisamment grand tel que $k\lambda_{m}(s)<2$. Par l'inégalité de Hölder, le terme
de droite est borné par
\[
C^{\prime}\mathbb{E}\left[  ||%
\operatorname{u}%
(t)||_{s}^{2}\right]  ^\frac{k\lambda_{m}(s)}{2}\mathbb{E}\left[  |%
\operatorname{u}%
(t)|_{\infty}^{a}\right]  ^{b},
\]
avec $C^{\prime}(m,k,s),a(m,k,s),b(m,k,s)>0.$ Donc, en utilisant
(\ref{esperence<visco}), nous obtenons (\ref{inegmomentksob}).

ii) Si $m$ n'est pas un entier, alors nous raisonnons comme dans l'étape iii)
de la démonstration du théorème \ref{th19}.%

\end{proof}

\section{Bilan de l'énergie et bornes inférieures}

Si dans l'identité d'Itô (\ref{ito}) on prend $f(%
\operatorname{u}%
(t))=\frac{1}{2}||%
\operatorname{u}%
(t)||^{2},$ alors pour tout $t\geq0$%
\[
\frac{1}{2}\frac{d}{dt}\mathbb{E}[||%
\operatorname{u}%
(t)||^{2}]+\nu\mathbb{E}[||%
\operatorname{u}%
(t)||_{1}^{2}]=\frac{1}{2}B_{0}.
\]
On intègre cette égalité entre $T\geq1$ et $T+\sigma$ $(\sigma>0):$%
\begin{equation}
\frac{1}{2}\mathbb{E}[||%
\operatorname{u}%
(T+\sigma)||^{2}]-\frac{1}{2}\mathbb{E}[||%
\operatorname{u}%
(T)||^{2}]+\nu\int_{T}^{T+\sigma}\mathbb{E}[||%
\operatorname{u}%
(s)||_{1}^{2}]ds=\frac{\sigma}{2}B_{0}\label{BE}%
\end{equation}
La relation (\ref{BE}) est appelée \textit{bilan de l'énergie}. Le terme
$\frac{1}{2}\mathbb{E}[||%
\operatorname{u}%
(t)||^{2}]$ est appelé \textit{l'énergie} de $%
\operatorname{u}%
(t),$ et $\mathbb{E}[||%
\operatorname{u}%
(s)||_{1}^{2}]$ est \textit{le taux de dissipation de l'énergie}.

Le théorème suivant nous donne une encadrement du taux de dissipation de
l'énergie$.$

\begin{theorem}
\label{th21}Soit $B_{4}<\infty$ et $%
\operatorname{u}%
(t)=%
\operatorname{u}%
(t;%
\operatorname{u}%
_{0}),$ $%
\operatorname{u}%
_{0}\in H^{1}.$ Alors, il existe $\sigma_{0}(B_{0},B_{4})>0$ telle que pour tout
$\sigma\geq\sigma_{0}$ et $T\geq1:$%
\[
\frac{1}{4}B_{0}\leq\frac{\nu}{\sigma}\int_{T}^{T+\sigma}\mathbb{E}[||%
\operatorname{u}%
(s)||_{1}^{2}]ds\leq\frac{3}{4}B_{0},
\]
uniformément en $0<\nu\leq1.$
\end{theorem}

\begin{proof}
Par la relation de Kolmogorov-Chapman (\ref{m0}) et le corollaire \ref{cor17} avec $T=\theta=1$ et $p=2$, et en
écrivant $t\geq1$ comme $t=\tau+1,$ $\tau\geq0$, nous avons
\begin{equation}
\mathbb{E}\left[  ||%
\operatorname{u}%
(t)||^{2}\right]  =\int_{H^{1}}||\operatorname{u}||^{2}\Sigma_{\tau+1}(%
\operatorname{u}%
_{0})(d%
\operatorname{u}%
)=\int_{H^{1}}\underset{=\mathbb{E}\left[  ||%
\operatorname{u}%
(1;w)||^{2}\right]  }{\underbrace{\langle ||\operatorname{u}||^{2},\Sigma_{1}(w)\rangle}}\Sigma_{\tau}(%
\operatorname{u}%
_{0})(dw)\leq C(B_{4}).\label{ccc}%
\end{equation}
Soit $\sigma\geq\frac{2C(B_{4})}{B_{0}}=:\sigma_{0}.$ Alors de (\ref{ccc}), on
obtient%
\[
\frac{1}{2\sigma}\mathbb{E}[||%
\operatorname{u}%
(T)||^{2}]\leq\frac{1}{4}B_{0},\quad\frac{1}{2\sigma}\mathbb{E}[||%
\operatorname{u}%
(T+\sigma)||^{2}]\leq\frac{1}{4}B_{0}.
\]
D'où le résultat par (\ref{BE}).
\end{proof}

\begin{notation}
\label{not22}Fixons $T\geq1$ et $\sigma\geq\sigma_{0}$. Si $\xi^{\omega}(t)$ est
un processus aléatoire réel, on pose
\[
\langle\langle\xi^{\omega}\rangle\rangle=\langle\langle\xi^{\omega}%
\rangle\rangle_{T,\sigma}=\frac{1}{\sigma}\int_{T}^{T+\sigma}\mathbb{E}%
[\xi^{\omega}(t)]dt.
\]

\end{notation}

Avec cette notation, l'inégalité du théorème \ref{th21} a pour expression
\[
\frac{1}{4}B_{0}\nu^{-1}\leq\langle\langle||%
\operatorname{u}%
||_{1}^{2}\rangle\rangle\leq\frac{3}{4}B_{0}\nu^{-1}.
\]
Le théorème suivant fournit un encadrement pour $\langle\langle||%
\operatorname{u}%
||_{m}^{2}\rangle\rangle.$

\begin{theorem}
\label{th23}Soit $m\in%
\mathbb{N}
^{\ast},$ $B_{m}<\infty,$ $\sigma\geq\sigma_{0}>0$, $T\geq1$ et $%
\operatorname{u}%
_{0}\in H^{1}$. Il existe $C_{m}(\sigma_{0})>1$ tel que la solution $%
\operatorname{u}%
=%
\operatorname{u}%
(t;%
\operatorname{u}%
_{0})$ satisfait :%
\begin{equation}
C_{m}^{-1}\nu^{-(2m-1)}\leq\langle\langle||%
\operatorname{u}%
||_{m}^{2}\rangle\rangle\leq C_{m}\nu^{-(2m-1)},\label{encadrementth23}%
\end{equation}
uniformément en $0<\nu\leq1.$
\end{theorem}

\begin{proof}
L'inégalité à droité de (\ref{encadrementth23}) suit du théorème \ref{th19},
et pour $m=1,$ l'inégalité à gauche est déjà étbalie. Maintenant, soit $m\in%
\mathbb{N}
$ et $m\geq2.$ Par l'inégalité (\ref{GN})
, on a%
\[
||%
\operatorname{u}%
_{x}||\leq c||%
\operatorname{u}%
_{x}||_{m-1}^{\frac{1}{2m-1}}|%
\operatorname{u}%
_{x}|_{1}^{\frac{2m-2}{2m-1}},\quad c>0.
\]
Donc, en utilisant l'inégalité de Hölder (appliquée à l'intégrale $\frac{1}{\sigma
}\int_{\Omega}\int_{T}^{T+\sigma}\cdot\cdot\cdot dt\mathbb{P}(d\omega)$) et le
corollaire \ref{cor17}, nous avons
\[
\langle\langle||%
\operatorname{u}%
||_{1}^{2}\rangle\rangle\leq c\langle\langle||%
\operatorname{u}%
||_{m}^{2}\rangle\rangle^{\frac{1}{2m-1}}\langle\langle|%
\operatorname{u}%
_{x}|_{1}^{2}\rangle\rangle^{\frac{2m-2}{2m-1}}\leq C_{m}\langle\langle||%
\operatorname{u}%
||_{m}^{2}\rangle\rangle^{\frac{1}{2m-1}}.
\]
C'est-à-dire :
\[
\langle\langle||%
\operatorname{u}%
||_{m}^{2}\rangle\rangle\geq C_{m}^{1-2m}\langle\langle||%
\operatorname{u}%
||_{1}^{2}\rangle\rangle^{2m-1}.
\]
Donc, d'après les théorèmes \ref{th19} et \ref{th21}, on obtient
(\ref{encadrementth23}).
\end{proof}

\begin{corollary}
\label{momentksobaveragecoro}Sous les conditions du théorème précédent, pour
tout $m\in%
\mathbb{N}
^{\ast}$ et $k\geq1,$ il existe $C(k,m,\sigma_{0})>0$ telle que
\begin{equation}
C^{-1}\nu^{-m+\frac{1}{2}}\leq\langle\langle||%
\operatorname{u}%
||_{m}^{k}\rangle\rangle^{\frac{1}{k}}\leq C\nu^{-m+\frac{1}{2}}%
.\label{momentksobaverage}%
\end{equation}

\end{corollary}

\begin{proof}
Après moyennisation danss (\ref{inegmomentksob}), l'inégalité de droite dans
(\ref{momentksobaverage}) est immédiate. Si $k\geq2$, alors l'inégalité de
gauche est une conséquence de l'inégalité de Hölder et de
(\ref{encadrementth23}). Maintenant, par l'inégalité de Hölder, on a%
\[
\langle\langle||%
\operatorname{u}%
||_{m}^{2}\rangle\rangle=\langle\langle||%
\operatorname{u}%
||_{m}^{\frac{2}{3}}||%
\operatorname{u}%
||_{m}^{\frac{4}{3}}\rangle\rangle\leq\langle\langle||%
\operatorname{u}%
||_{m}\rangle\rangle^{\frac{2}{3}}\langle\langle||%
\operatorname{u}%
||_{m}^{4}\rangle\rangle^{\frac{1}{3}}.
\]
Alors, par (\ref{momentksobaverage}) avec $k=2$ et $k=4$,%
\[
\langle\langle||%
\operatorname{u}%
||_{m}\rangle\rangle\geq\langle\langle||%
\operatorname{u}%
||_{m}^{2}\rangle\rangle^{\frac{3}{2}}\langle\langle||%
\operatorname{u}%
||_{m}^{4}\rangle\rangle^{-\frac{1}{2}}\geq\left(  \frac{\nu^{-m+\frac{1}{2}}%
}{C(2,m,\sigma_{0})}\right)  ^{3}\left(  C(4,m,\sigma_{0})\nu^{-m+\frac{1}{2}%
}\right)  ^{-2}=:C^{-1}\nu^{-m+\frac{1}{2}},
\]
et (\ref{momentksobaverage}) est établie pour $k=1$. Finalement, pour
$k\in(1,2),$ l'inégalité de gauche dans (\ref{momentksobaverage}) est une
conséquence de celle avec $k=1$ et de l'inégalité de Hölder.
\end{proof}

Considérons l'espace de probabilité
\[
\left(  Q,\mathcal{T},\rho\right)  :=\left(  [T,T+\sigma]\times\Omega
,\mathcal{L}\times\mathcal{F},\frac{dt}{\sigma}\times\mathbb{P}\right)  ,
\]
où $\sigma\geq\sigma_{0},$ $T\geq1,$ et $\mathcal{L}$ est la tribu Borélienne
sur $[T,T+\sigma].$ Avec cette notation, nous avons le corollaire suivant :

\begin{corollary}
\label{ensemble>0}Pour tout $m\in%
\mathbb{N}
^{\ast},$ il existe $\alpha(m,\sigma_{0}),C(m,\sigma_{0})>0$ telles que
\begin{equation}
\rho_{\nu,m}:=\rho\left(  \alpha\nu^{-m+\frac{1}{2}}\leq||%
\operatorname{u}%
^{\omega}(t)||_{m}\leq\alpha^{-1}\nu^{-m+\frac{1}{2}}\right)  \geq
C(m,\sigma_{0}),\label{ensembleinvariantennu}%
\end{equation}
uniformément en $0<\nu\leq1.$
\end{corollary}

\begin{proof}
Soient $\varepsilon>0$ et $Y_{\varepsilon}=\{(t,\omega)\in Q:||%
\operatorname{u}%
^{\omega}(t)||_{m}\leq\varepsilon\}.$ Le corollaire
\ref{momentksobaveragecoro} (avec $k=1$ et $k=2$) implique que
\begin{align*}
& C(1,m,\sigma_{0})^{-1}\nu^{-m+\frac{1}{2}}\leq\int_{Y_{\varepsilon}}||%
\operatorname{u}%
^{\omega}(t)||_{m}d\rho+\int_{Q\backslash Y_{\varepsilon}}||%
\operatorname{u}%
^{\omega}(t)||_{m}d\rho\\
& \leq\varepsilon+\left(  \int_{Q}||%
\operatorname{u}%
^{\omega}(t)||_{m}^{2}d\rho\right)  ^{\frac{1}{2}}\left(  \rho(Q\backslash
Y_{\varepsilon})\right)  ^{\frac{1}{2}}\leq\varepsilon+C(2,m,\sigma_{0}%
)\nu^{-m+\frac{1}{2}}\left(  \rho(Q\backslash Y_{\varepsilon})\right)
^{\frac{1}{2}}.
\end{align*}
On choisit $\varepsilon=\varepsilon_{\nu}=\frac{1}{2}C(1,m,\sigma_{0})^{-1}%
\nu^{-m+\frac{1}{2}}$, et on note par $Y=Y_{\varepsilon_{\nu}}.$ Alors
\begin{equation}
\rho\left(  ||%
\operatorname{u}%
^{\omega}(t)||_{m}>\varepsilon_{\nu}\right)  =\rho(Q\backslash Y)\geq\left(
\frac{1}{2}\left(  C(1,m,\sigma_{0})C(2,m,\sigma_{0})\right)  ^{-1}\right)
^{2}:=C_{m}^{\prime}.\label{rho}%
\end{equation}
Soit $\alpha^{-1}\geq2C(1,m,\sigma_{0}).$ D'après (\ref{rho}),
(\ref{momentksobaverage}) (avec $k=1)$ et l'inégalité de Tchebychev, nous
avons
\begin{equation}
\rho_{\nu,m}=\rho\left(  ||%
\operatorname{u}%
^{\omega}(t)||_{m}\geq\alpha\nu^{-m+\frac{1}{2}}\right)  -\rho\left(  ||%
\operatorname{u}%
^{\omega}(t)||_{m}\geq\alpha^{-1}\nu^{-m+\frac{1}{2}}\right)  \geq
C_{m}^{\prime}-\alpha C(1,m,\sigma_{0}).\nonumber
\end{equation}
Maintenant, on choisit $\alpha$ de sorte que le membre de droite soit
strictement positif, ce qui implique (\ref{ensembleinvariantennu}).
\end{proof}

\section{Échelle d'espace}

Soit $%
\operatorname{u}%
^{\nu}$ une fonction dérivable sur $\mathbb{S}^{1}$ de moyenne nulle,
dépendant du paramètre $0<\nu\ll1$ Notons par $M:=|%
\operatorname{u}%
^{\nu}|_{\infty}.$ Soit $\gamma\geq0.$ En physique, on dit que
\textit{l'échelle d'espace de }$%
\operatorname{u}%
^{\nu}$ est égale à $\nu^{\gamma}$ s'il existe des points $y\in\mathbb{S}%
^{1}$, tels que les incréments $%
\operatorname{u}%
(y+\nu^{\gamma})-%
\operatorname{u}%
(y)$ de $%
\operatorname{u}%
$ sont de l'ordre de $M$, et pour tout $\gamma^{\prime}>\gamma$, $%
\operatorname{u}%
(y+\nu^{\gamma})-%
\operatorname{u}%
(y)$ est toujours très petit par rapport à $M.$

Soit $v$ une fonction telle que $v(x)=\underset{s\in%
\mathbb{Z}
^{\ast}}{%
{\displaystyle\sum}
}\widehat{v}_{s}e^{2i\pi sx}\in\mathcal{C}^{\infty}(\mathbb{S}^{1})$, non
identiquement nulle. On considère la fonction périodique $%
\operatorname{u}%
^{\nu}(x)=v(x\lceil\nu^{-\gamma}\rceil)$, $\gamma\geq0,$ où $\lceil\cdot
\rceil$ est la fonction partie entière par excès. Il est facile de
vérifier que l'échelle d'espace de $%
\operatorname{u}%
^{\nu}$ est égale à $\nu^{\gamma}.$ Notons par $\widehat{%
\operatorname{u}%
}_{s}^{\nu}$, $s\in%
\mathbb{Z}
^{\ast},$ les coefficients de Fourier de $%
\operatorname{u}%
^{\nu}$, et soit $k_{s}:=\frac{s}{\lceil
\nu^{-\gamma}\rceil}.$ Alors, pour tout $s\in%
\mathbb{Z}
^{\ast}:\widehat{%
\operatorname{u}%
}_{s}^{\nu}=\widehat{v}_{k_{s}}\chi_{\{k_{s}\in%
\mathbb{Z}
^{\ast}\}},$ où $\chi_{\{k_{s}\in%
\mathbb{Z}
^{\ast}\}}$ est la fonction indicatrice de l'ensemble $\{k_{s}\in%
\mathbb{Z}
^{\ast}\}.$

Écrivons le \textit{nombre d'onde }$s$ pour $%
\operatorname{u}%
^{\nu}(x)$ en utilisant l'échelle logarithmique de base $\nu^{-1}:$%
\[
s=s_{\nu}(\alpha)=\lceil\nu^{-\alpha}\rceil,\quad\alpha>0,
\]
alors $\alpha=\alpha(s)$ est de l'ordre $\frac{\ln s}{\ln\nu^{-1}}.$ Soit
$\gamma_{0}>\gamma.$ Comme $v\in\mathcal{C}^{\infty}$, alors il est facile de
vérifier que
\begin{equation}
\text{si }\alpha\geq\gamma_{0}\text{ et }s=s_{\nu}(\alpha)\text{, alors
}\forall N\in%
\mathbb{N}
^{\ast},\text{ }\exists C_{N}>0\text{ telle que }|\widehat{%
\operatorname{u}%
}_{s_{\nu}(\alpha)}^{\nu}|\leq C_{N}|s_{\nu}(\alpha)|^{-N},\label{0.2}%
\end{equation}
pour chaque $\nu$, et $C_{N}$ ne dépendant pas de $\nu\in(0,1)$ (mais dépend a priori de
$\gamma_{0}$).

De plus, nous avons aussi
\begin{equation}
\text{si }\gamma_{0}<\gamma\text{ alors (\ref{0.2}) n'est pas valide pour tout
}0<\nu\leq1.\label{0.3}%
\end{equation}
Pour vérifier (\ref{0.3}), on peut choisir $\alpha=\gamma.$

Cette discussion motive la définition suivante

\begin{definition}
Soit $0<\nu\leq1$ et
\begin{equation}%
\operatorname{u}%
^{\nu}(x)=\underset{s\in%
\mathbb{Z}
^{\ast}}{%
{\displaystyle\sum}
}\widehat{%
\operatorname{u}%
}_{s}^{\nu}e^{2i\pi sx}\in\mathcal{C}^{1}(\mathbb{S}^{1}).\label{0.4}%
\end{equation}

\begin{enumerate}
\item Alors, l'échelle d'espace de $%
\operatorname{u}%
^{\nu}$ est égale à $\nu^{\gamma},$ $\gamma\geq0,$ si on a les deux affirmations suivantes :

i) (\ref{0.2}) est satisfaite pour tout $\gamma_{0}>\gamma.$

ii) si $\gamma>0$, alors (\ref{0.3}) est satisfaite.

\item Soient $t\geq0,$ $x\in\mathbb{S}^{1}$, $\omega\in\Omega$ et $%
\operatorname{u}%
^{\nu\omega}(t,x)$ un champ aléatoire continue en temps $t$, $\mathcal{C}^{1}$
en espace $x$, et de valeur moyenne nulle. On représente $%
\operatorname{u}%
^{\nu\omega}$ comme dans (\ref{0.4}) avec $\widehat{%
\operatorname{u}%
}_{s}^{\nu}=\widehat{%
\operatorname{u}%
}_{s}^{\nu\omega}(t)$. Alors l'échelle d'espace de $%
\operatorname{u}%
^{\nu\omega}$ est égale à $\nu^{\gamma},$ $\gamma\geq0,$ si les conditions i)
et ii) sont vérifiées avec $|\widehat{%
\operatorname{u}%
}_{s}^{\nu}|$ remplacé par $\langle\langle|\widehat{%
\operatorname{u}%
}_{s}^{\nu}|^{2}\rangle\rangle^{\frac{1}{2}}.$
\end{enumerate}
\end{definition}

Soit $%
\operatorname{u}%
(t,x)$ une solution de (\ref{B}) représentée par sa série de Fourier
(\ref{0.4}) avec $\widehat{%
\operatorname{u}%
}_{s}=\widehat{%
\operatorname{u}%
}_{s}^{\nu\omega}(t),$ cf. (\ref{expforme}). Par conséquent,%
\[
||%
\operatorname{u}%
||_{m}^{2}=2\sum_{k=1}^{\infty}|\widehat{%
\operatorname{u}%
}_{k}(t)|^{2}k^{2m}.
\]
Comme $\widehat{%
\operatorname{u}%
}_{k}=\int_{\mathbb{S}^{1}}%
\operatorname{u}%
(x)e^{-2i\pi kx},$ alors par une intégration par partie, nous trouvons que
\[
|\widehat{%
\operatorname{u}%
}_{k}(t)|\leq\frac{1}{2\pi k}|%
\operatorname{u}%
_{x}|_{1},\text{\quad}k\geq1.
\]
Donc, par le corollaire \ref{cor17}, si $\operatorname{u}(t,x)$ est une solution de (\ref{B}), alors 
\begin{equation}
\mathbb{E}[|\widehat{%
\operatorname{u}%
}_{k}(t)|^{2}]\leq Ck^{-2},\quad t\geq1.\label{espuk2k}%
\end{equation}

\begin{theorem}
\label{th24}Supposons que $B_{m}<\infty$ pour tout $m\in%
\mathbb{N}
^{\ast}.$ Alors, pour chaque $%
\operatorname{u}%
_{0}\in H^{1},$ l'échelle d'espace de $%
\operatorname{u}%
=%
\operatorname{u}%
(t;%
\operatorname{u}%
_{0})$ est égale à $\nu.$ C'est-à-dire, pour tout $\gamma_{0}>1$ et $N\in%
\mathbb{N}
^{\ast}:$%
\begin{equation}
\langle\langle|\widehat{%
\operatorname{u}%
}_{k}|^{2}\rangle\rangle\leq C_{N}k^{-N},\quad\text{si }k=\lceil\nu^{-\gamma
}\rceil,\text{ }\gamma>\gamma_{0},\text{ }\nu\in(0,1],\label{rel}%
\end{equation}
où $C_{N}=C_{N}(\gamma_{0})>0$ ne dépend ni de $\nu$, ni de $\gamma$. Par contre, si
$\gamma_{0}<1$, alors (\ref{rel}) n'est pas vrai pour tout $\gamma\geq
\gamma_{0}$ et $\nu\in(0,1],$ avec une certaine constante $C_{N}=C_{N}%
(\gamma_{0})>0$.
\end{theorem}

\begin{proof}
Par le théorème \ref{th23} on peut écrire%
\[
\langle\langle|\widehat{%
\operatorname{u}%
}_{k}|^{2}\rangle\rangle\leq C_{m}\nu(k\nu)^{-2m}.
\]
Si $k=k_{\nu}=\lceil\nu^{-\gamma}\rceil,$ $\gamma\geq\gamma_{0}>1$, nous
obtenons que%
\[
\langle\langle|\widehat{%
\operatorname{u}%
}_{k}|^{2}\rangle\rangle\leq C_{m}\nu(\nu^{1-\gamma})^{-2m}\leq C_{m}%
k^{-2m\frac{\gamma_{0}-1}{\gamma_{0}}},\text{\quad pour tout }m\in%
\mathbb{N}
^{\ast},
\]
d'où (\ref{rel}). Maintenant, supposons (\ref{rel}) avec $\gamma_{0}<1$ et
écrivons $\langle\langle||%
\operatorname{u}%
||_{m}^{2}\rangle\rangle$ comme :%
\[
\langle\langle||%
\operatorname{u}%
||_{m}^{2}\rangle\rangle=\underset{=:J}{\underbrace{2\sum_{1\leq k\leq
\nu^{-\gamma_{0}}}\langle\langle|\widehat{%
\operatorname{u}%
}_{k}|^{2}\rangle\rangle k^{2m}}}+\underset{=:L}{\underbrace{2\sum
_{k>\nu^{-\gamma_{0}}}\langle\langle|\widehat{%
\operatorname{u}%
}_{k}|^{2}\rangle\rangle k^{2m}}}.
\]
Majorons les deux termes du membre de droite. Soit $k>\nu^{-\gamma_{0}}.$
Alors, nous pouvons écrire $k=\lceil\nu^{-\gamma}\rceil,$ $\gamma\geq
\gamma_{0}.$ La relation (\ref{rel}) implique $\langle\langle|\widehat{%
\operatorname{u}%
}_{k}|^{2}\rangle\rangle\leq C_{N}(\gamma_{0})k^{-N}$ pour tout $N\in%
\mathbb{N}
.$ Choisissons $N=2m+2,$ nous avons%
\[
L\leq C(\gamma_{0})\sum_{k>\nu^{-\gamma_{0}}}k^{-2}\leq C^{\prime}(\gamma
_{0})\nu^{\gamma_{0}}.
\]
Aussi, par le théorème \ref{th23} :%
\[
J\leq\nu^{-2\gamma_{0}}\sum_{1\leq k<\nu^{-\gamma_{0}}}(k)^{2m-2}%
\langle\langle|\widehat{%
\operatorname{u}%
}_{k}|^{2}\rangle\rangle\leq\nu^{-2\gamma_{0}}\langle\langle||%
\operatorname{u}%
||_{m-1}^{2}\rangle\rangle\leq c\nu^{-2m-2\gamma_{0}+3}.
\]
D'ici,%
\begin{equation}
\langle\langle||%
\operatorname{u}%
||_{m}^{2}\rangle\rangle\leq C_{m}(\gamma_{0})\nu^{-2m-2\gamma_{0}+3},\quad
\nu\in(0,1].\label{um<nugamma0}%
\end{equation}
D'après le théorème \ref{th23} et (\ref{um<nugamma0}), on a pour tout
$1\geq\nu>0:c_{m}\nu^{-(2m-1)}\leq C_{m}(\gamma_{0})\nu^{-2m-2\gamma_{0}+3},$
ceci implique que $\gamma_{0}\geq1.$ La contradiction implique que $\gamma
_{0}\geq1.$
\end{proof}

\section{Lemmes de récurrence et L$_{1}$-contraction}

On rappelle que nous supposons $B_{4}<\infty.$

\begin{lemma}
\label{lem27}Soit $B_{m}<\infty,$ $m\geq1$ et $T>0.$ Alors pour tout
$\epsilon>0,$ il existe $\gamma_{\epsilon,T}>0$ telle que $\mathbb{P}%
(||\xi||_{X_{m}^{T}}<\epsilon)\geq\gamma_{\epsilon,T}.$
\end{lemma}

C'est un résultat classique sur la théorie des mesures Gaussiennes dans les espaces de Banach
\cite[théorème 3.6.1]{Bog}, \cite[section 2, corollaire 1.1]{Kuo}.

\begin{lemma}
\label{lem28}Soit $%
\operatorname{u}%
_{0}$ dans $H^{1}.$ Pour tout $\epsilon>0,$ il existe $T_{\epsilon
}=T_{\varepsilon}(||%
\operatorname{u}%
_{0}||_{1})>0$ et $\delta_{\epsilon}=\delta_{\epsilon}(||%
\operatorname{u}%
_{0}||_{1})>0$ telles que
\begin{equation}
\mathbb{P}(||%
\operatorname{u}%
(T_{\epsilon};%
\operatorname{u}%
_{0})||<\epsilon)\geq\delta_{\epsilon}.\label{proba(u<epsilon)}%
\end{equation}

\end{lemma}

\begin{proof}
Supposons d'abord que $\xi=0$, $%
\operatorname{u}%
_{0}\in H^{1}$ et notons $%
\operatorname{u}%
^{0}(t)=\operatorname{u}(t;\operatorname{u}_{0},0)$. On multiplie par $%
\operatorname{u}%
^{0}$ l'équation dans (\ref{B}) et on intègre par parties en espace, nous
obtenons%
\[
\frac{1}{2}\frac{d}{dt}||%
\operatorname{u}%
^{0}(t)||^{2}=-\nu||%
\operatorname{u}%
_{x}^{0}(t)||^{2}\leq-\nu||%
\operatorname{u}%
^{0}(t)||^{2}.
\]
Donc, en utilisant le lemme de Gronwall, nous obtenons que $||%
\operatorname{u}%
^{0}(t)||^{2}\leq e^{-2\nu t}||%
\operatorname{u}%
_{0}||^{2}.$ Par conséquent, il existe un $T_{\epsilon}=T_{\epsilon}(||%
\operatorname{u}%
_{0}||)>0$ tel que $||%
\operatorname{u}%
^{0}(T_{\epsilon})||\leq\frac{\epsilon}{2}.$ Par la proposition \ref{Mlip},
l'application $\mathcal{M}$ est Lipschitzienne sur les sous-ensembles bornés
de $H^{1}\times X_{1}^{T_{\varepsilon}}$. Donc, il existe
$\gamma_{\epsilon}=\gamma_{\epsilon}(||%
\operatorname{u}%
_{0}||_{1})>0$ tel que si $||\xi||_{X_{1}^{T}}\leq\gamma_{\epsilon}$ alors $||%
\operatorname{u}%
^{0}(T_{\epsilon})-\operatorname{u}(T_{\epsilon};\operatorname{u}_{0},\xi)||_{1}\leq\frac{\epsilon}{2}$. Ceci implique que si
$||\xi||_{X_{1}^{T}}\leq\gamma_{\epsilon}$ alors $||\operatorname{u}(T_{\epsilon};\operatorname{u}_{0},\xi)||<\epsilon,$ et par le lemme \ref{lem27} avec
$m=1$, on a $\mathbb{P(}||\xi||_{X_{1}^{T}}\leq\gamma_{\epsilon}%
\mathbb{)}=:\delta_{\varepsilon}>0.$ D'où (\ref{proba(u<epsilon)}).
\end{proof}

\begin{lemma}
\label{th29}(Récurrence)

Pour tout $%
\operatorname{u}%
_{0}$ dans $H^{1}$ et $\varepsilon>0$, on a
\begin{equation}
\pi(T,\varepsilon):=\mathbb{P}(\inf_{t\in\lbrack0,T[}||%
\operatorname{u}%
(t;%
\operatorname{u}%
_{0})||\geq\epsilon)\underset{T\rightarrow+\infty}{\rightarrow}%
0,\label{reccurence}%
\end{equation}
où la convergence est uniforme en $%
\operatorname{u}%
_{0}.$
\end{lemma}

Dans la suite on note par \[B_{m}^{R}(v)=\{u:||u-v||_{m}\leq R\}.\]

\begin{proof}
Par l'inégalité de Tchebychev et le corollaire \ref{cor17}, on a pour une
constante $c_{0}>0$ et pour tout $R>0:\mathbb{P}(||%
\operatorname{u}%
(1;%
\operatorname{u}%
_{0})||\geq R)\leq c_{0}^{-1}R.$ Si $R=2c_{0},$ alors $\mathbb{P}(||%
\operatorname{u}%
(1;%
\operatorname{u}%
_{0})||\leq R)\geq\frac{1}{2}.$

Par la relation de Kolmogorov-Chapman (\ref{m0}) appliquée à la fonction
caractéristique de $B_{0}^{\varepsilon}(0)$, et par (\ref{proba(u<epsilon)}),
nous avons que
\[
\mathbb{P(}||%
\operatorname{u}%
(1+T_{\varepsilon};%
\operatorname{u}%
_{0})||<\varepsilon\mathbb{)\geq}\int_{B_{0}^{R}(0)}\mathbb{P(}||%
\operatorname{u}%
(T_{\varepsilon};v)||<\varepsilon)\Sigma_{1}(%
\operatorname{u}%
_{0})(dv)\geq\frac{1}{2}\delta_{\varepsilon},\quad\forall%
\operatorname{u}%
_{0}\in H^{1}.
\]
Alors
\begin{equation}
\mathbb{P(}||%
\operatorname{u}%
(1+T_{\varepsilon};%
\operatorname{u}%
_{0})||\geq\varepsilon\mathbb{)}\leq1-\frac{1}{2}\delta_{\varepsilon}%
,\quad\forall%
\operatorname{u}%
_{0}\in H^{1}.\label{9,7}%
\end{equation}
On note par $T^{j}=j(1+T_{\epsilon})$, par $Q_{j}$ l'évènement $\{||%
\operatorname{u}%
(T^{j})||\geq\varepsilon\}$ et par $Q^{N}=%
{\textstyle\bigcap\limits_{i=1}^{N}}
Q_{i}.$ Nous avons que $Q^{N}=\{\omega\in Q^{N-1}:||%
\operatorname{u}%
(T^{N})||\geq\varepsilon\}$. En utilisant la relation de Kolmogorov-Chapman sous
la forme (\ref{m2}), ainsi que (\ref{9,7}), nous obtenons :%
\begin{align*}
\mathbb{P}(Q^{N})  & =\mathbb{E}[\chi_{\{%
\operatorname{u}%
(\tau;%
\operatorname{u}%
_{0})_{|\tau\in\lbrack0,T^{N-1}]}\in Q^{N-1}\}}\mathbb{P}(||%
\operatorname{u}%
(1+T_{\varepsilon};%
\operatorname{u}%
(T^{N-1};%
\operatorname{u}%
_{0})||\geq\varepsilon)]\\
& \leq(1-\frac{1}{2}\delta_{\varepsilon})\mathbb{E}[\chi_{\{%
\operatorname{u}%
(\tau;%
\operatorname{u}%
_{0})_{|\tau\in\lbrack0,T^{N-1}]}\in Q^{N-1}\}}]=(1-\frac{1}{2}\delta
_{\varepsilon})\mathbb{P}(Q^{N-1}).
\end{align*}
Donc, nous avons par récurrence que $\mathbb{P}(Q^{N})\leq(1-\frac{1}%
{2}\delta_{\varepsilon})^{N}.$ Finalement, comme $\pi(T,\varepsilon
)\leq\mathbb{P}(Q^{N})$ si $T\geq T^{N},$ nous obtenons (\ref{reccurence}) en
passant à la limite $N\rightarrow+\infty.$
\end{proof}

\begin{lemma}
\label{th30}($L_{1}$-contraction)

Pour $j=1,2,$ fixons $%
\operatorname{u}%
_{j}(0)\in H^{1},$ et soit $\xi\in X_{1}^{T}$. Considérons la solution $%
\operatorname{u}%
_{j}(t)=%
\operatorname{u}%
(t;%
\operatorname{u}%
_{j}(0),\xi)\in H^{1}$. Alors
\[
|%
\operatorname{u}%
_{1}(t)-%
\operatorname{u}%
_{2}(t)|_{1}\leq|%
\operatorname{u}%
_{1}(0)-%
\operatorname{u}%
_{2}(0)|_{1},\quad\forall t\geq0.
\]

\end{lemma}

\begin{proof}
Soit $w=%
\operatorname{u}%
_{1}-%
\operatorname{u}%
_{2}$. Alors $w$ vérifie
\begin{equation}
\left\{
\begin{array}
[c]{c}%
w_{t}+\frac{1}{2}(w(%
\operatorname{u}%
_{1}+%
\operatorname{u}%
_{2}))_{x}-\nu w_{xx}=0,\\
w(0)=%
\operatorname{u}%
_{1}(0)-%
\operatorname{u}%
_{2}(0)=:w_{0}.
\end{array}
\right. \label{equaw2}%
\end{equation}
Soit le problème de Cauchy conjugué $:$%
\begin{equation}
\left\{
\begin{array}
[c]{c}%
\phi_{t}+\frac{1}{2}(%
\operatorname{u}%
_{1}+%
\operatorname{u}%
_{2})\phi_{x}+\nu\phi_{xx}=0,\\
\phi(T,x)=\phi_{T}(x),
\end{array}
\right. \label{CC}%
\end{equation}
où $t\in\lbrack0,T],$ $\phi_{T}(x)\in\mathcal{C}^{0}$ et $|\phi_{T}|_{\infty
}=1.$ La théorie des équations parabolique implique qu'il existe une unique
solution $\phi$ au problème (\ref{CC}), et cette solution satisfait le
principe du maximum : $|\phi(t)|_{\infty}\leq1 $ pour tout $t\in[0,T]$
\cite{Evans}.

Après avoir multiplié par $\phi$ l'équation en $w$ dans (\ref{equaw2}),
intégré par parties sur $[0,T]\times\mathbb{S}^{1}$ puis en utilisant
l'équation en $\phi$ dans (\ref{CC}), on obtient que $|\langle w(T),\phi
_{T}\rangle|=|\langle w_{0},\phi(0)\rangle|.$ Donc, nous avons%
\begin{equation}
|\langle w(T),\phi_{T}\rangle|\leq|w_{0}|_{1}=:K\label{wksiK},%
\end{equation}
car $|\phi(0)|_{\infty}\leq1$. Soit $\epsilon>0$ et $\chi_{\epsilon}$ la fonction définie par%
\[
\chi_{\epsilon}(t)=\left\{
\begin{array}
[c]{c}%
-1,\text{ }t\leq-\epsilon,\\
\frac{1}{\epsilon}t,\text{ }-\epsilon\leq t\leq\varepsilon,\\
1,\text{ }t\geq\epsilon,
\end{array}
\right.
\]
alors $\chi_{\epsilon}(t)\underset{\epsilon\rightarrow0}{\rightarrow}$sgn$(t)$
pour chaque $t$, et si $\phi_{T}^{\epsilon}=\chi_{\epsilon}(w(T,x)),$ alors
$\phi_{T}^{\epsilon}$ est continue et $|\phi_{T}^{\varepsilon}|_{\infty}=1$.
Donc, par (\ref{wksiK}) et en utilisant le théorème de convergence dominée, nous avons
\[
|w(T)|_{1}=|\int w(T,x)\text{ sgn}(w(T,x))dx|=\lim_{\varepsilon\rightarrow
0}|\int w(T,x)\phi_{T}^{\epsilon}(x)dx|\leq K.
\]

\end{proof}

\section{Les mesures stationnaires et la méthode de Bogoliubov-Krylov}

On rappelle que nous supposons $B_{4}<\infty.$

\begin{definition}
\label{def31}Soit $\mu\in\mathcal{P}(H^{m}),m\geq1.$

\begin{enumerate}
\item $\mu$ est dite mesure stationnaire de (\ref{B}) si pour tout
$t\geq0:S_{t}^{\ast}\mu=\mu$. C'est-à-dire, si $\mathcal{D}(%
\operatorname{u}%
_{0})=\mu$ alors $\mathcal{D}(%
\operatorname{u}%
(t;%
\operatorname{u}%
_{0}))\equiv\mu.$

\item Une solution $%
\operatorname{u}%
(t)$ de (\ref{B}) telle que $\mathcal{D}(%
\operatorname{u}%
(t;%
\operatorname{u}%
_{0}))\equiv\mu$ est dite solution stationnaire.
\end{enumerate}
\end{definition}

Le but de cette section est de démontrer l'existence d'une mesure stationnaire
de (\ref{B}), (\ref{B*}), (\ref{B2}) par la \textit{méthode de
Bogoliubov-Krylov.}

\begin{theorem}
\label{th32}\label{existencemesurestatio}Il existe une mesure stationnaire
$\mu\in\mathcal{P}(H^{1}).$
\end{theorem}

\begin{proof}
On note par $%
\operatorname{u}%
(t)=%
\operatorname{u}%
(t;0)$ la solution de (\ref{B}) issue de $%
\operatorname{u}%
_{0}=0,$ et par $\mu(t):=\mathcal{D}(%
\operatorname{u}%
(t)).$ Notons que $S_{\theta}^{\ast}(\mu(t))=\mathcal{D}(%
\operatorname{u}%
(t+\theta))=\mu(t+\theta) $ pour tout $t,\theta\geq0.$

Considérons \textit{l'ansatz de Bogoliubov-Krylov}. C'est-à-dire, pour tout $T\geq1$
définissons $m_{T}\in\mathcal{P}(H^{2})$ par
\[
m_{T}(\cdot)=\frac{1}{T}\int_{1}^{1+T}\mu(s)(\cdot)ds.
\]
Par le théorème \ref{th19} (avec $m=2$)$,$ on sait qu'il existe une constante
$C>0$ telle que $\mathbb{E}[||%
\operatorname{u}%
(t)||_{2}^{2}]<C$ pour chaque $t\geq1.$ L'inégalité de Tchebychev implique que
$\mu(t)(B_{2}^{R}(0))=\mathbb{P}(||%
\operatorname{u}%
(t)||_{2}\leq R)\geq1-CR^{-1}$ si $t\geq1$. Donc, nous avons%
\begin{equation}
m_{T}(B_{2}^{R}(0))=\frac{1}{T}\int_{1}^{1+T}\mu(s)(B_{2}^{R}(0))ds\geq
1-\frac{C}{R},\quad\forall R>0.\label{m_Ttendue}%
\end{equation}
De plus, par le théorème de Rellich-Kondrachev, $B_{2}^{R}(0)$ est compact
dans $H^{1}$. Alors, par (\ref{m_Ttendue}), nous obtenons que l'ensemble
$M=\{m_{T}\}$ est tendu dans $\mathcal{P}(H^{1})$. Par le théorème \ref{th16},
cet ensemble est faiblement relativement compact. Donc, la suite $\{m_{1}%
,m_{2},\ldots\}\subset M$ admet une sous-suite $\{m_{j_{k}}\}$ qui converge
faiblement vers une limite $\mu\in\mathcal{P}(H^{1}).$

Il reste à prouver que $\mu$ est stationnaire. Calculons $S_{\theta}^{\ast
}(m_{j_{k}})$, pour tout $\theta>0.$ Comme $S_{t}^{\ast}$ est linéaire et
vérifie $S_{\theta}^{\ast}(\mu(t))=\mu(t+\theta)$, nous avons
\begin{align*}
S_{\theta}^{\ast}(m_{j_{k}}) &  =\frac{1}{j_{k}}\int_{1}^{1+j_{k}}S_{\theta
}^{\ast}(\mu(t))dt=\frac{1}{j_{k}}\int_{1}^{1+j_{k}}\mu(t+\theta)dt=\frac
{1}{j_{k}}\int_{1+\theta}^{1+j_{k}+\theta}\mu(t)dt\\
&  =\underset{:=I}{\underbrace{-\frac{1}{j_{k}}\int_{1}^{1+\theta}\mu(t)dt}%
}+\underset{=m_{j_{k}}}{\underbrace{\frac{1}{j_{k}}\int_{1}^{j_{k}+1}\mu
(t)dt}}+\underset{:=K}{\underbrace{\frac{1}{j_{k}}\int_{j_{k}+1}%
^{1+j_{k}+\theta}\mu(t)dt}}.
\end{align*}
Or, pour tout $f\in\mathcal{C}_{b}(H^{1}),$ les termes$|\langle K,f\rangle| $
et $|\langle I,f\rangle|$ sont bornés par $\frac{1}{j_{k}}|f|_{\infty}\theta$
et$\underset{j_{k}\rightarrow+\infty}{\lim}|\langle I,f\rangle|=\underset
{j_{k}\rightarrow+\infty}{\lim}|\langle K,f\rangle|=0.$ De plus, comme
$m_{j_{k}}\rightharpoonup\mu$ et $S_{\theta}^{\ast}$ est faiblement continue
(cf. théorème \ref{ST*continuefaible}), en passant à la limite dans l'égalité ci
dessus, on a finalement $S_{\theta}^{\ast}\mu=\mu.$
\end{proof}

La mesure stationnaire du théorème \ref{existencemesurestatio} est aussi
lisse que la force $\eta^{\omega}$ dans (\ref{B}) :

\begin{proposition}
\label{p33}\label{Hmsupportemu}Soit $B_{m}<\infty,$ $m\geq1.$ Alors la mesure
stationnaire $\mu$ est supportée par l'espace $H^{m}:\mu(H^{m})=1.$
\end{proposition}

\begin{proof}
Si $B_{m}<\infty$ alors par (\ref{aussilissequelaforce}), $\mu(s)(H^{m})=1$
pour tout $s>0,$ et $m_{T}(H^{m})=1$ pour tout $T\geq1.$ Comme $H^{m}$ est un
sous-ensemble fermé de $H^{1}$ et $m_{j_{k}}\rightharpoonup\mu$ dans $H^{1}$,
alors $\mu(H^{m})=1$ (cf. (\ref{ferme}))$.$
\end{proof}

Plus particulièrement, nous avons aussi le

\begin{corollary}
Si $B_{m}<\infty$ pour tout $m\geq1,$ alors $\mu(\mathcal{C}^{\infty})=1. $
\end{corollary}

\section{L'unicité de la mesure stationnaire et la propriété de mélange}

On renvoie à \cite[p.101-103]{KuShir}, pour le cadre classique du
\textit{couplage de Döblin.} Il est remarquable que le couplage de Döblin
s'applique à l'équation de Burgers avec des modifications minimales; nous
expliquerons cela dans cette section. Par contre, pour les équations aux
dérivées partielles stochastiques plus complexes, l'idée de Döblin est plus
délicate à appliquer; elle s'emploie avec des modifications plus profondes
\cite[Section 3]{KuShir}.

Nous supposerons toujours que $B_{4}<+\infty$. Considérons deux copies
d'équation de Burgers de conditions initiales $%
\operatorname{u}%
^{\prime},%
\operatorname{u}%
^{\prime\prime}\in H^{1}.$ C'est-à-dire :%
\begin{equation}
\left\{
\begin{array}
[c]{c}%
(%
\operatorname{u}%
_{1})_{t}+%
\operatorname{u}%
_{1}(%
\operatorname{u}%
_{1})_{x}-\nu(%
\operatorname{u}%
_{1})_{xx}=\eta,\quad%
\operatorname{u}%
_{1}(0)=%
\operatorname{u}%
^{\prime},\\
(%
\operatorname{u}%
_{2})_{t}+%
\operatorname{u}%
_{2}(%
\operatorname{u}%
_{2})_{x}-\nu(%
\operatorname{u}%
_{2})_{xx}=\eta,\quad%
\operatorname{u}%
_{2}(0)=%
\operatorname{u}%
^{\prime\prime},
\end{array}
\right. \label{copieB}%
\end{equation}
avec $\operatorname{u}^{\prime}$ et $\operatorname{u}^{\prime\prime}$ des variables aléatoires indépendantes de $\eta$. Soit $\operatorname{u}_{1}(t)=\operatorname{u}_{1}(t,\operatorname{u}^{\prime})$ et $\operatorname{u}_{2}(t)=\operatorname{u}_{2}(t,\operatorname{u}^{\prime\prime})$. Alors 
\[U(t)=(\operatorname{u}_{1}(t),\operatorname{u}_{2}(t))\]
est solution de (\ref{copieB}) telle que 
\[U(0)=(\operatorname{u}^{\prime},\operatorname{u}^{\prime\prime})=:U_{0}.\]

Si $\pi_{1}$ et ${\pi_{2}}$ sont les opérateurs définies par :
\[\pi_{1}(\operatorname{u}_{1},\operatorname{u}_{2})=\operatorname{u}_{1},\quad \pi_{2}(\operatorname{u}_{1},\operatorname{u}_{2})=\operatorname{u}_{2},\]
alors $\pi_{j}\circ\mathcal{D}(U(t))=\mathcal{D}(\operatorname{u_{j}(t)})$, $j=1,2$. 
On dit que la mesure $\mathcal{D}(U(t))\in\mathcal{P}(H^{1}\times H^{1})$ est un $\textit{couplage}$ des mesures $\mathcal{D}(\operatorname{u_{1}(t)}$ et $\mathcal{D}(\operatorname{u_{2}(t)}$ (voir \cite{KuShir} pour cet important concept).

En répétant le schéma de la démonstration du lemme \ref{th29}, nous avons le

\begin{lemma}
\label{lem35}Pour tout $\varepsilon>0$ et $U_{0}\in H^{1}\times H^{1}:$
$\mathbb{P}(\inf_{t\in\lbrack0,T]}||U(t)||\geq\varepsilon)\underset
{T\rightarrow\infty}{\rightarrow}0$, uniformément en $U_{0}.$
\end{lemma}

\begin{lemma}
\label{lem36}Soit $m\geq1.$ Il existe $C(\nu)>0$ telle que si $%
\operatorname{u}%
^{\prime},%
\operatorname{u}%
^{\prime\prime}\in H^{1}$ et $|%
\operatorname{u}%
^{\prime}|_{1},|%
\operatorname{u}%
^{\prime\prime}|_{1}\leq\frac{1}{m},$ alors pour tout $t\geq1$, on a :
$||\Sigma_{t}(%
\operatorname{u}%
^{\prime})-\Sigma_{t}(%
\operatorname{u}%
^{\prime\prime})||_{L(H^{1})}^{\ast}\leq Cm^{-\frac{2}{5}}.$
\end{lemma}

\begin{proof}
Par l'inégalité (\ref{GN})%
, nous avons
\[||\operatorname{u}_{1}(t)-\operatorname{u}_{2}(t)||_{1}\leq c|\operatorname{u}_{1}(t)-\operatorname{u}_{2}(t)|_{1}^{\theta}||\operatorname{u}_{1}(t)-\operatorname{u}_{2}(t)||_{2}^{1-\theta},\quad\theta=\frac{2}{5},\quad c>0.\]

Donc, pour tout $f\in\mathcal{C}_{b}(H^{1})$ tel que $||f||_{L(H^{1})}\leq1,$
on a%
\begin{align}
\left\vert \langle\Sigma_{t}(%
\operatorname{u}%
^{^{\prime}}),f\rangle-\langle\Sigma_{t}(%
\operatorname{u}%
^{\prime\prime}),f\rangle\right\vert 
= \left\vert \mathbb{E}(f(\operatorname{u}_{1}(t))-f(\operatorname{u}_{2}(t)))\right\vert 
& \leq\mathbb{E}[||%
\operatorname{u}%
_{1}(t)-%
\operatorname{u}%
_{2}(t)||_{1}]\nonumber\\
&  \leq c\mathbb{E}[|%
\operatorname{u}%
_{1}(t)-%
\operatorname{u}%
_{2}(t)|_{1}^{\theta}||%
\operatorname{u}%
_{1}(t)-%
\operatorname{u}%
_{2}(t)||_{2}^{1-\theta}].\label{gn2/7nvl}%
\end{align}
Comme $|%
\operatorname{u}%
_{1}(0)-%
\operatorname{u}%
_{2}(0)|_{1}\leq\frac{2}{m}$, alors $|%
\operatorname{u}%
_{1}^{\omega}(t)-%
\operatorname{u}%
_{2}^{\omega}(t)|_{1}\leq\frac{2}{m}$ pour tout $\omega$ (voir lemme
\ref{th30}). De plus, $(\alpha-\beta)^{1-\theta}\leq\alpha^{1-\theta}%
+\beta^{1-\theta}$ pour tout $\alpha,\beta\geq0,$ alors le membre de droite de
(\ref{gn2/7nvl}) est borné par
\[
c(\frac{2}{m})^{\theta}(\mathbb{E}[||%
\operatorname{u}%
_{1}(t)||_{2}^{1-\theta}]+\mathbb{E}[%
\operatorname{u}%
_{2}(t)||_{2}^{1-\theta}]).
\]
Par le théorème \ref{th19} et l'inégalité de Hölder, on a $\mathbb{E}[||%
\operatorname{u}%
_{j}(t)||_{2}^{1-\theta}]\leq C_{2}(\nu^{-3})^{\frac{1-\theta}{2}} $, $j=1,2$
et $C_{2}>0$. Donc,
\[
\left\vert \langle\Sigma_{t}(%
\operatorname{u}%
^{\prime}),f\rangle-\langle\Sigma_{t}(%
\operatorname{u}%
^{\prime\prime}),f\rangle\right\vert \leq C\nu^{-\frac{3}{2}(1-\theta
)}m^{-\theta},\quad\text{si }||f||_{L(H^{1})}\leq1,
\]
et on obtient le résultat énoncé.
\end{proof}

\begin{proposition}
\label{p34}Pour tout $%
\operatorname{u}%
^{\prime},%
\operatorname{u}%
^{\prime\prime}\in H^{1},$%
\[
||\Sigma_{t}(%
\operatorname{u}%
^{\prime})-\Sigma_{t}(%
\operatorname{u}%
^{\prime\prime})||_{L(H^{1})}^{\ast}\leq d_{t}^{0},\quad\forall t\geq0,
\]
où $t\longmapsto d_{t}^{0}$ est une fonction positive indépendante de $%
\operatorname{u}%
^{\prime}$ et $%
\operatorname{u}%
^{\prime\prime}$, telle que $\underset{t\rightarrow+\infty}{\lim}d_{t}^{0}=0$.
\end{proposition}

\begin{proof}
On note $d_{t}(%
\operatorname{u}%
^{\prime},%
\operatorname{u}%
^{\prime\prime})=||\Sigma_{t}(%
\operatorname{u}%
^{\prime})-\Sigma_{t}(%
\operatorname{u}%
^{\prime\prime})||_{L(H^{1})}^{\ast}=\underset{||f||_{L}\leq1}{\sup}%
\langle\Sigma_{t}(%
\operatorname{u}%
^{\prime}),f\rangle-\langle\Sigma_{t}(%
\operatorname{u}%
^{\prime\prime}),f\rangle$. On a pour tout $f\in\mathcal{C}_{b}(H^{1})$
\[
\langle\Sigma_{t}(%
\operatorname{u}%
^{\prime}),f\rangle-\langle\Sigma_{t}(%
\operatorname{u}%
^{\prime\prime}),f\rangle=\mathbb{E}[f(%
\operatorname{u}%
_{1}(t))-f(%
\operatorname{u}%
_{2}(t))]=\mathbb{E}[F(U(t))]=:g(t;U_{0}).
\]
où on a noté $f(%
\operatorname{u}%
_{1})-f(%
\operatorname{u}%
_{2})$ par $F(U),$ $U=(%
\operatorname{u}%
_{1},%
\operatorname{u}%
_{2})$. Donc $d_{t}$ s'écrit comme$\underset{||f||_{L}\leq1}{\sup}g(t;U_{0})$
et $g(t;U_{0})\leq d_{t}$. Soit $m\geq1$ et $O_{m}=\{u\in H^{1}:|u|_{1}%
\leq\frac{1}{m}\}$. Considérons \textit{le temps d'atteinte} pour $U^{\omega
}(t)$ dans l'ensemble $O_{m}\times O_{m}:$%
\[
\tau_{m,t}^{\omega}=\inf\{s\in\lbrack0,t]:U^{\omega}(s)\in O_{m}\times
O_{m}\}\quad\text{(cf. appendice B).}%
\]
i) Comme $|%
\operatorname{u}%
_{j}|_{1}\leq||%
\operatorname{u}%
_{j}||$ pour tout $j=1,2,$ le lemme \ref{lem35} implique que $\mathbb{P}%
\mathbb{(}\tau_{m,t}\geq t-1)\ $tend vers $0$, uniformément en $U_{0},$ quand
$t$ tend vers $+\infty.$\newline 
ii) Si $\tau_{m,t}<t$ alors $U(\tau_{m,t})\in
O_{m}\times O_{m}.$\newline 
iii) Comme $g(t;U_{0})\leq d_{t}$, en utilisant le
lemme \ref{lem36}, on obtient que si $U_{0}\in O_{m}\times O_{m}$ alors
$g(t;U_{0})\leq Cm^{-\frac{2}{5}}$, $t\geq1$. \newline
Par la propriété de Markov forte
(\ref{mf}), $\mathbb{E}[F(U(t))]=\mathbb{E}[g(t-\tau_{m,t};U(\tau_{m,t})].$
Donc%
\begin{equation}
\mathbb{E}[F(U(t))]=\mathbb{E}[\chi_{\{\tau_{m,t}\leq t-1\}}g(t-\tau_{m,t}%
;U(\tau_{m,t}))]+\mathbb{E}[\chi_{\{\tau_{m,t}\leq t-1\}}g(t-\tau_{m,t};U(\tau
_{m,t}))],\label{EFU(t)}%
\end{equation}
où $\chi_{\cdot}$ est la fonction caractéristique.
Par ii) et iii), si $\tau_{m,t}\leq t-1$ alors $U(\tau_{m,t})\in O_{m}\times
O_{m}$ et $g(t-\tau_{m,t};U(\tau_{m,t}))\leq Cm^{-\frac{2}{5}}$. Comme
$|F(U(t))|\leq2$, alors (\ref{EFU(t)}) implique que%
\[
\left\vert \mathbb{E}[F(U(t))]\right\vert \leq Cm^{-\frac{2}{5}}%
+2\mathbb{P}\{\tau_{m,t}>t-1\}=:d_{t}^{0}.
\]
Soit $\varepsilon>0$. On choisit $m=m_{\varepsilon}$ de sorte que
$Cm_{\varepsilon}^{-\frac{2}{5}}\leq\frac{1}{2}\varepsilon.$ Par i), nous
avons $2\mathbb{P}\{\tau_{m,t}>t-1\}\leq\frac{1}{4}\varepsilon$, si $t\geq
t_{m_{\varepsilon},\varepsilon}$ est suffisamment grand. Finalement, $\mathbb{E}[F(U(t))]\leq
\varepsilon$ pour tout $t\geq t_{m_{\varepsilon},\varepsilon},$ et $d_{t}(%
\operatorname{u}%
^{\prime},%
\operatorname{u}%
^{\prime\prime})\leq d_{t}^{0}\underset{t\rightarrow+\infty}{\rightarrow}0.$
\end{proof}

\begin{theorem}
\label{th37}La mesure stationnaire $\mu$ donnée par le théorème \ref{th32} est
unique. De plus, pour tout $\lambda$ dans $\mathcal{P}(H^{1}):$%
\begin{equation}
\lim_{t\rightarrow+\infty}||S_{t}^{\ast}\lambda-\mu||_{L(H^{1})}^{\ast
}=0,\label{mel}%
\end{equation}
uniformément en $\lambda$. En particulier, pour tout $%
\operatorname{u}%
_{0}\in H^{1}$, on a $\underset{t\rightarrow+\infty}{\lim}||\Sigma_{t}(%
\operatorname{u}%
_{0})-\mu||_{L(H^{1})}^{\ast}=0$ uniformément en $%
\operatorname{u}%
_{0}.$
\end{theorem}

La propriété (\ref{mel}) est dite\textit{\ de mélange} et $%
\operatorname{u}%
(t)$ est dit \textit{processus mélangeant}.

\begin{proof}
Soit $\mu$ la mesure stationnaire construite par l'ansatz de Bogoliubov-Krylov
(Théorème \ref{th32}) et $\mu_{1}$ une autre mesure stationnaire. Si la
relation (\ref{mel}) est vraie alors $||\mu_{1}-\mu||_{L(H^{1})}^{\ast
}=||S_{t}^{\ast}\mu_{1}-\mu||_{L(H^{1})}^{\ast}\underset{t\rightarrow+\infty
}{\rightarrow}0$ et par conséquent $\mu_{1}=\mu$. À présent, démontrons la
relation (\ref{mel}). Soit $f\in\mathcal{C}_{b}(H^{1})$ tel que $||f||_{L}%
\leq1$ et $X_{t}:=|\langle f,S_{t}^{\ast}\lambda\rangle-\langle f,\mu
\rangle|=|\langle f,S_{t}^{\ast}\lambda\rangle-\langle f,S_{t}^{\ast}%
\mu\rangle|$. Nous avons
\[
\langle f,S_{t}^{\ast}\lambda\rangle=\langle S_{t}f,\lambda\rangle=\int
_{H^{1}}S_{t}f(u)\lambda(du)=\int_{H^{1}}\int_{H^{1}}S_{t}f(u)\lambda
(du)\mu(du^{\prime}).
\]
De même,
\[
\langle f,S_{t}^{\ast}\mu\rangle=\int_{H^{1}}\int_{H^{1}}S_{t}f(u^{\prime}%
)\mu(du^{\prime})\lambda(du).
\]
Ainsi,
\[
X_{t}\leq\int_{H^{1}}\int_{H^{1}}|S_{t}f(u)-S_{t}f(u^{\prime})|\lambda
(du)\mu(du^{\prime}).
\]
Par la proposition \ref{p34}, l'intégrande $|S_{t}f(u)-S_{t}f(u^{\prime
})|=|\langle\Sigma_{t}(%
\operatorname{u}%
^{\prime})-\Sigma_{t}(%
\operatorname{u}%
^{\prime\prime}),f\rangle|$ est bornée par $d_{t}^{0}$. Par conséquent, $X_{t}\leq
d_{t}^{0}$. Donc, nous avons (\ref{mel}) dans la limite $t\rightarrow
+\infty.$
\end{proof}

L'analyse de la proposition \ref{p34} implique que%
\[
||\Sigma_{t}(%
\operatorname{u}%
^{\prime})-\Sigma_{t}(%
\operatorname{u}%
^{\prime\prime})||_{L(L_{1})}^{\ast}\leq Ct^{-\frac{1}{13}},\quad\forall
t\geq0,
\]
voir \cite{Bor}, \cite{Borphd}. Alors la solution $%
\operatorname{u}%
(t)$ est un processus \textit{algébriquement mélangeant}. Or, si la force
$\eta$ définie dans (\ref{B*}) est \textit{non-dégénérée} :%
\begin{equation}
b_{s}\neq0,\quad\text{si }|s|\leq n_{\ast},\label{m.0}%
\end{equation}
avec $n_{\ast}\in%
\mathbb{N}^{\ast}$ convenable, 
alors $%
\operatorname{u}%
(t)$ est un processus \textit{exponentiellement mélangeant :}%
\begin{equation}
||\Sigma_{t}(%
\operatorname{u}%
^{\prime})-\Sigma_{t}(%
\operatorname{u}%
^{\prime\prime})||_{L(L_{1})}^{\ast}\leq Ce^{-ct},\quad\forall t\geq
0.\label{m.1}%
\end{equation}
Le théorème 3.1.7 de \cite{KuShir} implique (\ref{m.1}) si $n_{\ast}=n_{\ast
}(\nu)$ est suffisamment grand (par exemple, si $b_{s}\neq0,$ $\forall s$), et
la méthode développée dans \cite{HM} implique (\ref{m.1}) avec un certain $n_{\ast}$ qui ne dépend pas de $\nu$. La relation
(\ref{m.1}) entraîne que le processus $%
\operatorname{u}%
^{\omega}(t)$ est "bien ergodique". Pou illustrer, prenons $f\in
\mathcal{C}_{b}(H^{1})$ et considérons le processus $f(%
\operatorname{u}%
(t))$. On peut démontrer que $f(%
\operatorname{u}%
(t))$ satisfait \textit{la loi forte des grands nombres} et \textit{le
théorème central limite \cite[Section 4]{KuShir}}. Le problème reste ouvert si
la relation (\ref{m.1}) est vraie avec des constantes $c$ et $C$ qui ne
dépendent pas de $\nu\in(0,1].$

\section{Fonction de structure}

La \textit{fonction de structure }est l'un des objets principaux de la théorie
de la turbulence hydrodynamique \cite{Frisch}. Pour un "fluide
unidimensionnel" qui est décrit par l'équation de Burgers, la fonction de
structure est définie comme suit :

\begin{definition}
\label{def25}Si $%
\operatorname{u}%
$ désigne la solution de l'équation de Burgers alors sa fonction de structure
de degré $p>0$ est%
\begin{equation}
S_{p}(l)=\langle\langle\int_{\mathbb{S}^{1}}|%
\operatorname{u}%
(x+l)-%
\operatorname{u}%
(x)|^{p}dx\rangle\rangle,\text{\quad}l\in(0,1].\label{sp(l)}%
\end{equation}

\end{definition}

\begin{theorem}
\label{th26}Soit $B_{m}<\infty$ pour tout $m\in%
\mathbb{N}
^{\ast}$, $\nu\in(0,c_{\ast}]$, où $c_{\ast}\leq1$ une constante qui dépend
seulement de $B_{1},$ $B_{2},$ $\ldots.$

\begin{enumerate}
\item Soit $p>0$. Il existe $c_{p}>0$ telle que
\begin{equation}
S_{p}(l)\leq\left\{
\begin{array}
[c]{l}%
c_{p}l^{p},\text{\qquad\quad si }p\in(0,1),\\
c_{p}l^{p}\nu^{1-p},\text{\quad si }p\geq1,
\end{array}
\right.  \text{pour tout }l\in(0,1]\text{ et }\nu\in(0,1].\label{spl1}%
\end{equation}

\item Il existe des constantes $c_{1},c_{2}>0$, et pour chaque $p>0,$ il
existe $c_{p}^{\prime}(c_{1},c_{2})>0$ telle que
\begin{equation}
S_{p}(l)\geq\left\{
\begin{array}
[c]{l}%
c_{p}^{\prime}l^{p},\text{  si }p\in(0,1),\\
c_{p}^{\prime}l,\text{\quad si }p\geq1,
\end{array}
\right.  \text{pour tout }l\in\lbrack c_{1}\nu,c_{2}]\text{ et }\nu
\in(0,c_{\ast}].\label{spl2}%
\end{equation}

\end{enumerate}
\end{theorem}

\begin{proof}
1) Si $p\geq1$, alors%
\[
S_{p}(l)\leq\langle\langle\int|%
\operatorname{u}%
(x+l)-%
\operatorname{u}%
(x)|dx\cdot\max_{x}|%
\operatorname{u}%
(x+l)-%
\operatorname{u}%
(x)|^{p-1}\rangle\rangle.
\]
Par l'inégalité de Hölder, on a
\[
S_{p}(l)\leq\underset{=:I}{\underbrace{\langle\langle\left(  \int|%
\operatorname{u}%
(x+l)-%
\operatorname{u}%
(x)|dx\right)  ^{p}\rangle\rangle^{1/p}}}\underset{=:J}{\underbrace
{\langle\langle\max_{x}|%
\operatorname{u}%
(x+l)-%
\operatorname{u}%
(x)|^{p}\rangle\rangle^{(p-1)/p}}}.%
\]
D'une part, remarquons que la moyenne de $x\mapsto%
\operatorname{u}%
(x+t)-%
\operatorname{u}%
(x)$ est nulle, donc par le lemme \ref{lem18}
\[
\int|%
\operatorname{u}%
(x+l)-%
\operatorname{u}%
(x)|dx\leq2\int(%
\operatorname{u}%
(x+l)-%
\operatorname{u}%
(x))^{+}dx\leq2\sup_{x}(%
\operatorname{u}%
_{x})^{+}\cdot l,
\]
et $I\leq2l\langle\langle\lbrack\sup_{x}(%
\operatorname{u}%
_{x})^{+}]^{p}\rangle\rangle^{1/p}.$ Donc, d'après le corollaire \ref{cor17}, on
obtient que $I\leq c_{p}l.$ D'autre part,%
\[
J\leq\langle\langle l^{p}|%
\operatorname{u}%
_{x}|_{\infty}^{p}\rangle\rangle^{(p-1)/p}.
\]
De l'inégalité (\ref{GN}) (avec $\theta(0,1,\infty,m-1)$) et celle de Hölder,
on obtient
\[
\langle\langle l^{p}|%
\operatorname{u}%
_{x}|_{\infty}^{p}\rangle\rangle^{\frac{p-1}{p}}\leq\left[  cl^{p}%
\langle\langle||%
\operatorname{u}%
||_{m}^{\frac{2p}{2m-1}}|%
\operatorname{u}%
_{x}|_{1}^{a}\rangle\rangle\right]  ^{\frac{p-1}{p}}\leq cl^{p-1}%
\langle\langle||%
\operatorname{u}%
||_{m}^{2}\rangle\rangle^{\frac{p-1}{2m-1}}\langle\langle|%
\operatorname{u}%
_{x}|_{1}^{b}\rangle\rangle^{c},
\]
pour certaines constantes $a,b,c>0.$ Maintenant, enutilisant le corollaire
\ref{cor17} et le théorème \ref{th23}, on obtient que
\[
J\leq c^{\prime}l^{p-1}\nu^{-(p-1)}.
\]
Finalement $S_{p}(l)\leq c_{p}l^{p}\nu^{1-p},$ si $p\geq1.$

Si $p\in(0,1),$ alors nous obtenons par l'inégalité de Hölder et par
(\ref{spl1}) avec $p=1:$
\[
S_{p}(l)\leq\langle\langle\int|%
\operatorname{u}%
(x+l)-%
\operatorname{u}%
(x)|dx\rangle\rangle^{p}=S_{1}(l)^{p}\leq(c_{1}l)^{p}.
\]
Ceci achève la preuve de (\ref{spl1}).

2) Notons par $Q_{1}$ l'événement dans (\ref{ensembleinvariantennu}) avec
$m=1$. Alors $\rho(Q_{1})\geq C(1,\sigma_{0})$. Soit $M>0$ et
\[
Q_{2}=\{(t,\omega)\in Q_{1}:|%
\operatorname{u}%
_{x}^{\omega+}(t)|_{\infty}+|%
\operatorname{u}%
_{x}^{\omega}(t)|_{1}+\nu^{\frac{3}{2}}||%
\operatorname{u}%
^{\omega}(t)||_{2}+\nu^{\frac{5}{2}}||%
\operatorname{u}%
^{\omega}(t)||_{3}\}\leq M\}.
\]
Par le corollaire \ref{cor17}, (\ref{momentksobaverage}) et l'inégalité de
Tchebychev :%
\[
\rho(Q_{2})\geq C(1,\sigma_{0})-C_{1}M^{-1}\geq\frac{1}{2}C(1,\sigma_{0}),
\]
pour tout $\nu\in(0,c_{\ast}]$ et si $M$ est suffisamment grand.

Soit $(t,\omega)\in Q_{2}$ et notons $v(x):=%
\operatorname{u}%
^{\omega}(t,x)$. Pour établir (\ref{spl2}), nous montrons que $v$ satisfait
\begin{equation}
\int|v(x+l)-v(x)|^{p}dx\geq Cl^{\min(1,p)},\quad l\in\lbrack c_{1}\nu
,c_{2}],\text{ }p>0\label{77.3},%
\end{equation}
uniformément en $\nu\in(0,1]$, où $C=C(c_{1},c_{2},p)>0.$

D'abord, supposons que $p\geq1$. Notons que
\[
\alpha\nu^{-1}\leq\int|v_{x}|^{2}dx\leq|v_{x}|_{\infty}|v_{x}|_{1}\leq
M|v_{x}|_{\infty},
\]
où $\alpha$ est définie dans le corollaire \ref{ensemble>0}. Donc%
\begin{equation}
|v_{x}|_{\infty}\geq\alpha M^{-1}\nu^{-1}=:\alpha_{2}\nu^{-1},\label{77.1}%
\end{equation}
uniformément en $\nu.$ Comme $|v_{x}^{+}|_{\infty}\leq M$, alors $|v_{x}%
^{+}|_{\infty}\leq\frac{1}{2}\alpha_{2}\nu^{-1}$, si $\nu\leq\frac{1}{2}%
\alpha_{2}M^{-1}=:c_{\ast}$. Ainsi, par (\ref{77.1}), on a
\[
|v_{x}^{+}|_{\infty}\leq\frac{1}{2}\alpha_{2}\nu^{-1},\quad\text{et }%
|v_{x}^{-}|_{\infty}\geq\alpha_{2}\nu^{-1},\quad\text{si }\nu\in
\lbrack0,c_{\ast}].
\]
Notons par $z=z(t,\omega)=\min\{x\in\lbrack0,1):v_{x}^{-}(x)\geq\alpha_{2}%
\nu^{-1}\}:$ $z$ est une fonction mesurable bien définie sur $Q_{2}$, si
$\nu\in\lbrack0,c_{\ast}].$

Nous avons
\begin{equation}
\int|v(x+l)-v(x)|^{p}dx\geq\int_{z-\frac{l}{2}}^{z}|\int_{x}^{x+l}v_{x}%
^{-}(y)dy-\int_{x}^{x+l}v_{x}^{+}(y)dy|^{p}dx.\label{strcturev}%
\end{equation}
Par le lemme \ref{GNtermeL}, on a $|%
\operatorname{u}%
_{xx}|_{\infty}\leq C_{2}||%
\operatorname{u}%
||_{2}^{\frac{1}{2}}||%
\operatorname{u}%
||_{3}^{\frac{1}{2}}$, ce qui implique que $|v_{xx}|_{\infty}\leq C_{2}%
M\nu^{-2}$. Donc, dans l'intervalle $[x,x+\alpha_{3}\nu]$, $\alpha_{3}>0$,
nous avons
\[
v_{x}^{-}\geq\alpha_{2}\nu^{-1}-\alpha_{3}C_{2}M\nu^{-1}=\frac{3}{4}\alpha
_{2}\nu^{-1},
\]
si $\alpha_{3}=\frac{\alpha_{2}}{4C_{1}M}$. Supposons que $l\geq\alpha_{3}%
\nu.$ Comme $v_{x}^{+}\leq M$, alors
\[
\int_{x}^{x+l}v_{x}^{-}(y)dy\geq\int_{x}^{x+\alpha_{3}\nu}v_{x}^{-}%
(y)dy\geq\frac{3}{4}\alpha_{2}\alpha_{3},\quad\text{et }\int_{x}^{x+l}%
v_{x}^{+}(y)dy\leq Ml.
\]
Donc, en utilisant (\ref{strcturev}), nous obtenons que
\[
\int|v(x+l)-v(x)|^{p}dx\geq\int_{z-\frac{l}{2}}^{z}|\frac{3}{4}\alpha
_{2}\alpha_{3}-Ml|^{p}dx\geq\frac{l}{2}\left(  \frac{1}{2}\alpha_{2}\alpha
_{3}\right)  ^{p},
\]
pourvu que $l\in\lbrack\alpha_{3}\nu,\frac{\alpha_{2}\alpha_{3}}{4M}]$ et
$\nu\in\lbrack0,c_{\ast}]$. Ainsi, la relation (\ref{spl2}) est établie avec
$c_{\ast}=\frac{1}{2}\alpha_{2}M^{-1}$, $c_{1}=\alpha_{3}$ et $c_{2}%
=\frac{\alpha_{2}\alpha_{3}}{4M}$, si $p\geq1$.

Maintenant, supposons que $p\in(0,1)$. Soit $f$ une fonction arbitraire positive.
Nous pouvons l'écrire comme $f=f^{\frac{2-2p}{2-p}}f^{\frac{p}{2-p}}$. Donc, d'après l'inégalité de Hölder, nous avons $\left(  \int f\right)  ^{2-p}\leq\left(
\int f^{2}\right)  ^{1-p}\left(  \int f^{p}\right)  $. Par conséquent,
\begin{align*}
\int|v(x+l)-v(x)|^{p}dx  & \geq\int\left(  \left(  v(x+l)-v(x)\right)
^{+}\right)  ^{p}dx\\
& \geq\left(  \int\left(  \left(  v(x+l)-v(x)\right)  ^{+}\right)
^{2}dx\right)  ^{p-1}\left(  \int\left(  \left(  v(x+l)-v(x)\right)
^{+}\right)  dx\right)  ^{2-p}.
\end{align*}
Comme $v_{x}^{+}\leq M$, alors $\left(  v(x+l)-v(x)\right)  ^{+}\leq Ml$. De
plus, $p-1<0$, donc le premier terme du membre de droite de cette dernière
inégalité est minoré par $\left(  M^{2}l^{2}\right)  ^{p-1}$.

Remarquons que la fonction $v(\cdot+l)-v(\cdot)$ est de moyenne nulle. Par
conséquent
\[
\int\left(  v(x+l)-v(x)\right)  ^{+}dx=\frac{1}{2}\int|v(x+l)-v(x)|dx,
\]
et utilisant (\ref{77.3}) avec $p=1$, nous obtenons que le second terme est
minoré par $Cl^{2-p}$. Finalement, (\ref{77.3}) est établi pour $p\in(0,1).$
\end{proof}

\section{Spectre de l'énergie}

\begin{definition}
On appelle l'énergie de nombre d'onde $k$ correspondant à la solution $%
\operatorname{u}%
$ de l'équation de Burgers, la grandeur :%
\begin{equation}
E_{k}=\frac{1}{2k(M-M^{-1})}\sum_{M^{-1}k\leq|n|\leq Mk}\frac{1}{2}%
\langle\langle|%
\operatorname{u}%
_{n}|^{2}\rangle\rangle.\label{energyondek}%
\end{equation}
où $M$ est une constante positive indépendante de $\nu.$ La fonction
$k\longmapsto E_{k}$ est le spectre de l'énergie.
\end{definition}

Le théorème \ref{th24} dit que si $k$ est supérieur au \textit{seuil
critique} égal à $\nu^{-1}$, alors le spectre de l'énergie décroît plus vite
que n'importe quelle puissance négative de $k$, et que cela n'est pas valide
si $k\ll\nu$.

Dans cette partie, nous continuons l'étude du spectre de l'énergie $E_{k}$
quand $k\lesssim\nu^{-1}.$

\begin{theorem}
\label{encadrementEk}Soit $M$ dans (\ref{energyondek}) suffisament grand, et
$c_{1},c_{2}>0$ les constantes du théorème \ref{th26}, alors il existe
$c_{3},c_{4}>0$ telles que pour tout $k$ satisfaisant la condition
\begin{equation}
c_{2}^{-1}\leq k\leq(c_{1}\nu)^{-1},\label{cond}%
\end{equation}
nous avons
\begin{equation}
c_{3}k^{-2}\leq E_{k}\leq c_{4}k^{-2}.\label{E_kcommek^2}%
\end{equation}

\end{theorem}

\begin{proof}
On précise que les constantes $C,$ $C_{k},\ldots$etc, de cette preuve ne
dépendent pas de $M.$ La relation (\ref{espuk2k}) implique que $\langle
\langle|%
\operatorname{u}%
_{k}|^{2}\rangle\rangle\leq Ck^{-2}.$ Notons que cette relation entraîne la
seconde inégalité de (\ref{E_kcommek^2}). Maintenant, nous vérifions la
première. Comme $\langle\langle|%
\operatorname{u}%
_{k}|^{2}\rangle\rangle\leq Ck^{-2}$, alors%
\begin{equation}
\sum_{|n|\leq M^{-1}k}|n|^{2}\langle\langle|%
\operatorname{u}%
_{n}(t)|^{2}\rangle\rangle\leq CM^{-1}k,\label{cm-1k}%
\end{equation}
et%
\begin{equation}
\sum_{|n|\geq Mk}\langle\langle|%
\operatorname{u}%
_{n}(t)|^{2}\rangle\rangle\leq C^{\prime}M^{-1}k^{-1},\label{cm-1k-1}%
\end{equation}
Posons $S=\underset{|n|\leq Mk}{\sum}|n|^{2}\langle\langle|%
\operatorname{u}%
_{n}(t)|^{2}\rangle\rangle.$ Comme $|\sin(\alpha)|\leq|\alpha|$, alors%
\begin{equation}
S\geq\frac{k^{2}}{\pi^{2}}\sum_{|n|\leq Mk}\sin^{2}(\frac{n\pi}{k}%
)\langle\langle|%
\operatorname{u}%
_{n}|^{2}\rangle\rangle=\frac{k^{2}}{\pi^{2}}\left(  \sum_{n\in%
\mathbb{Z}
^{\ast}}\sin^{2}(\frac{n\pi}{k})\langle\langle|%
\operatorname{u}%
_{n}|^{2}\rangle\rangle-\sum_{|n|>Mk}\sin^{2}(\frac{n\pi}{k})\langle\langle|%
\operatorname{u}%
_{n}|^{2}\rangle\rangle\right)  .\label{SS}%
\end{equation}
Notons que par l'identité de Parseval, nous avons
\[
||%
\operatorname{u}%
(\cdot+y)-%
\operatorname{u}%
(\cdot)||^{2}=4\sum_{n\in%
\mathbb{Z}
^{\ast}}\sin^{2}(n\pi y)|%
\operatorname{u}%
_{n}|^{2}.
\]
Donc, (\ref{SS}) et (\ref{cm-1k-1}) impliquent que
\begin{equation}
S\geq\frac{k^{2}}{\pi^{2}}\left(  \frac{1}{4}\langle\langle||%
\operatorname{u}%
(\cdot+\frac{1}{k})-%
\operatorname{u}%
(\cdot)||^{2}\rangle\rangle-\sum_{|n|>Mk}\langle\langle|%
\operatorname{u}%
_{n}|^{2}\rangle\rangle\right)  \geq k^{2}C_{1}S_{2}(\frac{1}{k}%
)-CM^{-1}k.\label{Sbon}%
\end{equation}
Comme $k$ satisfait (\ref{cond}), alors par (\ref{Sbon}) et (\ref{spl2})
($p=2,$ $l=\frac{1}{k}$), nous avons
\begin{equation}
S\geq k^{2}C_{1}c_{2}^{\prime}k^{-1}-C^{^{\prime}}M^{-1}k=k(C^{\prime\prime
}-C^{\prime}M^{-1}).\label{diez'}%
\end{equation}
Notons que
\[
E_{k}\geq E_{k}^{-}=\frac{1}{4k^{3}M^{3}}\sum_{M^{-1}k\leq|n|\leq Mk}%
|n|^{2}\langle\langle|%
\operatorname{u}%
_{n}(t)|^{2}\rangle\rangle.
\]%
Donc, en utilisant (\ref{diez'}) et (\ref{cm-1k}), on obtient que
\[
E_{k}\geq\frac{1}{4k^{3}M^{3}}\left(  S-\sum_{|n|\leq M^{-1}k}|n|^{2}%
\langle\langle|%
\operatorname{u}%
_{n}|^{2}\rangle\rangle\right)  \geq\frac{C^{\prime\prime}-C^{\prime}%
M^{-1}-CM^{-1}}{4k^{2}M^{3}}>k^{-2}M^{-3}C_{3},\quad C_{3}>0
\]
si $M\gg1$. D'où la première inégalité de (\ref{E_kcommek^2}).%

\end{proof}

En ignorant les constantes multiplicatives devant les puissances de $\nu$, les
physiciens écrivent le segment $[c_{2}^{-1},c_{1}^{-1}\nu^{-1}]$ comme
$[\nu^{0},\nu^{-1}]$ et l'appellent la \textit{zone inertielle}. Alors, le
théorème \ref{encadrementEk} dit que dans la zone inertielle l'énergie $E_{k}
$ se comporte comme $k^{-2}.$ De même, on appelle le segment $[\nu
^{-1},+\infty)$ la \textit{zone dissipative}, et dans ce cas, le théorème
\ref{th26} se traduit par le fait que dans cette zone, l'énergie décroît plus
vite que toute puissance négative de $k$. On renvoie à \cite{Bor} pour une
assertion plus forte, qui justifie plus rigoureusement la définition de ces zones.

\section{Appendice A : Le processus de Wiener standard}

Soit $\mathcal{C}(%
\mathbb{R}
^{+})$ l'espace des fonctions continues définies sur $%
\mathbb{R}
^{+}=[0,+\infty\lbrack$, muni de la distance :%
\[
d(u,v)=\sum_{n=1}^{+\infty}2^{-n}\frac{\underset{0\leq t\leq n}{\max
}|u(t)-v(t)|}{1+\underset{0\leq t\leq n}{\max}|u(t)-v(t)|},
\]
$(\mathcal{C}(%
\mathbb{R}
^{+}),d)$ est un espace métrique complet et séparable. Soit $\mathcal{B}$ la
tribu Borélienne de $\mathcal{C}(%
\mathbb{R}
^{+})$ et soit $(\Omega,\mathcal{F},\mathbb{P})$ un espace probabilisé standard
\cite{Dud}. Par exemple, $\Omega=[0,1]$ muni de la tribu Borélienne
$\mathcal{F}$ et de la mesure de Lebesgue.

Considérons l'application mesurable $w:(\Omega,\mathcal{F})\rightarrow\left(
\mathcal{C}(%
\mathbb{R}
^{+}\right)  ,\mathcal{B}),$ $\omega\longmapsto w^{\omega}$, telle que :

\begin{enumerate}
\item $w^{\omega}(0)=0$,
\item pour tout $0\leq t_{1}\leq\ldots\leq t_{N},$ la variable aléatoire
$\omega\longmapsto(w^{\omega}(t_{1}),\ldots,w^{\omega}(t_{N}))\in%
\mathbb{R}
^{N}$ est Gaussienne,

\item pour chaque $t\geq0$, $\mathbb{E}[w^{\omega}(t)]=0$ et $\mathbb{E}%
[w^{\omega}(t)^{2}]=t$, et 

\item si $0\leq t_{1}\leq t_{2}\leq t_{3}\leq t_{4}$ $,$ alors les variables
aléatoires $w(t_{4})-w(t_{3})$ et $w(t_{2})-w(t_{1})$ sont indépendantes.
\end{enumerate}

On appelle une telle application mesurable sur $(\Omega,\mathcal{F}%
,\mathbb{P})$, un \textit{processus de Wiener standard}. Notons que si on pose
$\widetilde{w}(t)=w(\widetilde{t}+t)-w(\widetilde{t})$ pour $\widetilde{t}\geq0$ fixé,
alors $\widetilde{w}$ définit aussi un processus de Wiener standard. Il est
connu qu'on peut construire une famille dénombrable $w_{1},w_{2},\ldots$ de
processus de Wiener standard indépendants \cite{Gall}, \cite{Kuo}.

La célèbre inégalité maximale de Doob \cite{Gall}, \cite{KaS}, \cite{Shir},
appliquée à un processus de Wiener standard $w(t)$ dit que
\begin{equation}
\mathbb{E}[\left(  \sup_{0\leq t\leq T}|w(t)|^{2}\right)  ^{p}]\leq\left(
\frac{p}{p-1}\right)  ^{p}\mathbb{E}[\left(  |w(t)|^{2}\right)  ^{p}%
],\quad\forall p>1.\label{doob}%
\end{equation}
L'inégalité (\ref{doob}) reste vraie si $|w(t)|=\left(  w_{1}(t)^{2}%
+\ldots+w_{N}(t)^{2}\right)  ^{1/2},$ et $w_{1},\ldots,w_{N}$ sont des
processus de Wiener standard indépendants.

\section{Appendice B : Les temps d'atteinte et d'arrêt, et la propriété de
Markov forte}

Soit $t,m\geq0,$ $B_{m}<\infty$ et $%
\operatorname{u}%
(t)\in H^{m},$ une solution de (\ref{B}). Soit $Q\subset H^{m}$ un sous-ensemble fermé, et une constante $T>0$. Notons%
\[
\tau=\tau_{Q,T}^{\omega}=\inf\{0\leq t\leq T:%
\operatorname{u}%
^{\omega}(t)\in Q\}.
\]
Ici, $\tau=T$ si $%
\operatorname{u}%
^{\omega}(t)\notin Q$ pour $t\leq T.$ On appelle $\tau$ le \textit{temps
d'atteinte }(dans l'ensemble fermé $Q$). C'est un cas particulier du
\textit{temps d'arrêt} \cite{Gall}, \cite{Oks}, \cite{KaS}. Notons que
$\tau=0$ si $Q=H^{m},$ et que $\tau=T$ si $Q=\varnothing.$

L'importance de ces concepts vient de la propriété des processus de Markov,
qu'on appelle : \textit{propriété de Markov forte}. Nous la formulons pour une
solution $%
\operatorname{u}%
(t)=%
\operatorname{u}%
^{\omega}(t;%
\operatorname{u}%
_{0}^{\omega})$ de (\ref{B}), où $%
\operatorname{u}%
_{0}$ est une variable aléatoire dans $H^{1}$ indépendante de $\xi^{\omega}$
et telle que $\mathbb{E}[||%
\operatorname{u}%
_{0}||_{1}]<\infty.$ Soit $T_{1}\geq T$ et $f\in\mathcal{C}_{b}(H^{m}),$
alors
\begin{equation}
\mathbb{E}[f(%
\operatorname{u}%
(T_{1})]=\mathbb{E}[S_{T_{1}-\tau_{Q,T}^{\omega}}f(%
\operatorname{u}%
(\tau_{Q,T}^{\omega}))].\label{mf}%
\end{equation}
Ici, la fonction $S_{t}f(v)$ est défini dans (\ref{Stdef}) où nous avons
remplacé $t$ par $T_{1}-\tau_{Q,T}^{\omega}$, et $v$ par $%
\operatorname{u}%
(\tau_{Q,T}^{\omega})$. Notons que si $\tau=T$ et $%
\operatorname{u}%
(0)=const,$ alors (\ref{mf}) coïncide avec la propriété de Kolmogorov-Chapman
(\ref{m1}).

\section{Notation}

Soit $x\in\mathbb{S}^{1}$, $p\in\lbrack1,\infty]$ et $m\in%
\mathbb{R}
$. Pour une fonction $%
\operatorname{u}%
(x)$, nous notons par $|%
\operatorname{u}%
|_{p}$ sa norme dans l'espace de Lebesgue $L_{p}(\mathbb{S}^{1})$, et par $||%
\operatorname{u}%
||_{m}$ sa norme de Sobolev homogène d'ordre $m$. Si $m=0$, nous écrivons $||%
\operatorname{u}%
||:=||%
\operatorname{u}%
||_{0}=|%
\operatorname{u}%
|_{2}$.

Soit $\lambda>0$. Nous désignons par $\lambda\gg1$ (resp. $\lambda\ll1$) si
$\lambda$ est suffisamment grand (resp. $\lambda$ est suffisamment petit).

L'abbreviation p.p désigne "presque partout".

\end{document}